\documentclass{amsart}
\usepackage{amsmath}
\usepackage{amsthm}
\usepackage{amssymb}
\usepackage{enumerate}
\usepackage{hyperref}

\newtheorem{theorem}{Theorem}[section]
\newtheorem{corollary}{Corollary}
\newtheorem*{main}{Main Theorem}
\newtheorem{lemma}[theorem]{Lemma}

\newtheorem{conjecture}{Conjecture}
\newtheorem{problem}{Problem}
\theoremstyle{definition}
\newtheorem{definition}[theorem]{Definition}
\newtheorem{remark}{Remark}

\newtheorem{example}{Example}
\newtheorem*{RP}{Rearrangement Problem (RP)}
\newtheorem*{RRP}{Rotation-invariant Rearrangement Problem (RRP)}

\newcommand{\vek}[1]{\mathbf{#1}}
\newcommand{\mset}[1]{\mathfrak{#1}}
\newcommand{\abs}[1]{\lvert#1\rvert}
\newcommand{\innpro}[2]{\langle #1,#2\rangle}
\newcommand{\mat}[1]{\mathbf{#1}}

\newcommand{\gauss}[3]{\genfrac{[}{]}{0pt}{}{#1}{#2}_{#3}}

\newcommand{\ksubspaces}[2]{\genfrac{[}{]}{0pt}{}{#1}{#2}}

\DeclareMathOperator{\PG}{PG}

\DeclareMathOperator{\Aut}{Aut}

\DeclareMathOperator{\GL}{GL}

\DeclareMathOperator{\Hom}{Hom}
\DeclareMathOperator{\End}{End}

\DeclareMathOperator{\rank}{rk}

\DeclareMathOperator{\kernel}{Ker}
\DeclareMathOperator{\image}{Im}

\newcommand{\imat}{\mathbf{I}}
\newcommand{\zentrum}{\mathrm{Z}}
\newcommand{\trace}{\mathrm{Tr}}

\newcommand{\opp}{\circ}

\newcommand{\wham}{\mathrm{w}_{\mathrm{Ham}}}
\newcommand{\dham}{\mathrm{d}_{\mathrm{Ham}}}

\newcommand{\sdist}{\mathrm{d}_{\mathrm{s}}}

\newcommand{\smax}{\mathrm{A}}

\newcommand{\F}{\mathbb{F}}
\DeclareMathOperator{\GF}{GF}

\newcommand{\Z}{\mathbb{Z}}
\newcommand{\U}{\mathrm{U}}

\newcommand{\dickson}{\delta}
\newcommand{\sickson}{\sigma}
\newcommand{\graph}{\Gamma}
\newcommand{\gabidulin}{\mathcal{G}}
\newcommand{\lgabidulin}{\mathcal{L}}

\newcommand{\spol}{\mathrm{s}}

\newcommand{\adj}{\ast}

\newcommand{\frob}{\varphi}

\newcommand{\littleo}{\mathrm{o}}
\newcommand{\Wtilde}{\widetilde{W}}
\newcommand{\tp}{\mathsf{T}}

\newcommand{\name}[1]{\textsc{#1}}

\title[The Expurgation-Augmentation Method]{The Expurgation-Augmentation
    Method for Constructing Good Plane Subspace Codes}

\author[Jingmei Ai, Thomas Honold, and Haiteng Liu]{}



\email{somnus@zju.edu.cn}
\email{honold@zju.edu.cn}
\email{liuhaiteng@zju.edu.cn}

\thanks{Research supported by the National Natural Science Foundation
  of China under Grants 60872063 and 61571006}

\begin{document}
\maketitle

\centerline{\scshape Jingmei Ai, Thomas Honold, and Haiteng Liu}
\medskip
{\footnotesize
  \centerline{Department of Information Science and Electronics
    Engineering}
  \centerline{Zhejiang University, 38 Zheda Road, 310027 Hangzhou, China}
} 

%

\bigskip


\begin{abstract}
  As shown in \cite{smt:fq11proc}, one of the five isomorphism types
  of optimal binary subspace codes of size $77$ for packet length
  $v=6$, constant dimension $k=3$ and minimum subspace distance $d=4$
  can be constructed by first expurgating and then augmenting the
  corresponding lifted Gabidulin code in a fairly simple way. The
  method was refined in \cite{lt:0328,mt:alb80} to yield an
  essentially computer-free construction of a currently best-known
  plane subspace code of size $329$ for $(v,k,d)=(7,3,4)$. In this paper we
  generalize the expurgation-augmentation approach to arbitrary packet
  length $v$, providing both a detailed theoretical analysis of our
  method and computational results for small parameters. As it turns
  out, our method is capable of producing codes larger than those
  obtained by the echelon-Ferrers construction and its variants. We
  are able to prove this observation rigorously for packet lengths
  $v\equiv 3\pmod{4}$.
\end{abstract}

\section{Introduction}\label{sec:intro}

Let $V$ be a vector space of finite dimension $v$ over the finite
field $\mathbb{F}_q$. The lattice $\PG(V)$ of all subspaces of $V$,
relative to the operations $X\wedge Y=X\cap Y$ (meet) and
$X\vee Y=X+Y$ (join) is called the projective (coordinate) geometry
associated with $V$ and forms the unique (up to isomorphism) model of
$\PG(v-1,\F_q)$, the Desarguesian projective geometry of geometric
dimension $v-1$ and order $q$. Moreover, $\PG(V)$ forms a metric space
with respect to the \emph{subspace distance} defined by
$\sdist(X,Y)=\sdist(X+Y)-\sdist(X\cap Y)=\dim(X)+\dim(Y)-2\dim(X\cap
Y)$.
A code in this metric space is called a \emph{$q$-ary subspace
  code}.\footnote{In the special case $V=\F_q^v$ one also refers to
  $V$, $v$, $q$ as \emph{packet space}, \emph{packet length} and
  \emph{packet alphabet size}, respectively. This is motivated by the
  applications in network coding.}  Such codes are of interest in the
emerging area of error-resilient network coding, where they can be
used as channel codes for the linear operator channel introduced by
Koetter and Kschischang \cite{koetter-kschischang08}; cf.\ also
\cite{kschischang12,guang-zhang14,etzion-storme16}.  The so-called
\emph{Main Problem of Subspace Coding}, motivated by the application
to network coding and modelled after the Main Problem of classical
coding theory, asks for the determination (and, as a secondary goal,
the classification) of subspace codes of maximum size when the
remaining parameters are fixed. In contrast with its classical
counterpart, however, much less is known about the Main Problem of
Subspace Coding, in particular in the general mixed-dimension
case. Fur the current state of knowledge we refer to the online tables
at \texttt{subspacecodes.uni-bayreuth.de}, a recently established
service of the research group at Bayreuth University;
cf.\ also \cite{heinlein-kiermaier-kurz-wassermann16}.  For the
mixed-dimension case, see \cite{smt:alc15proc} and the references
therein.

In this paper we restrict ourselves to (a subcase of) the
constant-dimension case of the Main Problem, which is somewhat more
accessible and, because of its geometric significance, has been
investigated earlier by researchers in Finite Geometry and solved in
special cases. A subspace code in $V$ is said to be a
\emph{constant-dimension code} if all its members have the same
dimension $k$. For the parameters of a constant-dimension code
$\mathcal{C}$ we use the notation $(v,M,d;k)$ or $(v,M,d;k)_q$, where
$q$, $v$, $k$ have the same meaning as above,
$M=\#\mathcal{C}=\abs{\mathcal{C}}$, and $d$ denotes the minimum
(subspace) distance of $\mathcal{C}$. The Main Problem, restricted to
the constant-dimension case, is to determine the maximum sizes
$\smax_q(v,d;k)$ of $(v,M,d;k)_q$ codes.

For $X,Y\in\ksubspaces{V}{k}$, the set of $k$-dimensional subspaces
(``$k$-subspaces'') of $V$, the formula for the subspace
distance reduces to $\sdist(X,Y)=2k-2\dim(X\cap Y)\in 2\Z$, and the
inequality $\sdist(X,Y)\geq d=2\delta$ is equivalent to
$\dim(X\cap Y)<t$ with $t=k-\delta+1$. Hence a $(v,M,d;k)_q$
constant-dimension code $\mathcal{C}$ with ambient vector space $V$
may also be viewed as a set of ($k-1$)-flats in $\PG(V)$ with
$\#\mathcal{C}=M$ and the following property: $t=k-\delta+1$ is the
smallest integer such that every ($t-1$)-flat of $\PG(V)$
($t$-subspace of $V$) is contained in at most one member
of $\mathcal{C}$. Finding a $(v,M,2\delta;k)_q$ code of maximum size thus
translates exactly into the packing problem for the incidence
structure of $t$-subspaces versus $k$-subspaces of $V$,
provided we identify ``blocks'' of this incidence structure (i.e.\
subspaces $X\in\ksubspaces{V}{k}$) with sets of ``points'' (the set of
all $T\in\ksubspaces{V}{t}$ satisfying $T\subseteq X$).

The special case $t=1$ of the restricted Main Problem asks for the
maximum number of $k$-subspaces of $V$ that are pairwise disjoint as
point sets. In Finite Geometry such sets of subspaces are known as
\emph{partial spreads} and have been the subject of extensive research
since the fundamental work of Beutelspacher~\cite{beutelspacher75};
see \cite{eisfeld-storme00} for a survey. However, apart from the
``spread case'' $v\equiv 0\pmod{k}$ and the line case $k=2$ already
solved in~\cite{beutelspacher75}, the numbers $\smax_q(v,2k;k)$ remain
unknown in general. Only for certain specific parameter combinations
they have been determined exactly---in the case $v\equiv 1\pmod k$
\cite{beutelspacher75}, the case $q=2$, $k=3$
\cite{el-zanati-jordon-seelinger-sissokho-spence10}, and recently in
the case $q=2$, $v\equiv 2\pmod k$ \cite{kurz15}.

Predating the work of Koetter-Kschischang \cite{koetter-kschischang08}
on network coding, the Main Problem for general constant-dimension subspace
codes has already been formulated and investigated by
Metsch in the language of Finite Geometry (see \cite{metsch00}, in
particular Section~4) and by Wang, Xing and Safavi-Naini in their work
on linear authentication codes (see
\cite[Th.~4.1]{wang-xing-safavi-naini03} and the discussion following
this theorem). The publication of \cite{koetter-kschischang08} and the
related work
\cite{silva-kschischang-koetter08,silva-kschischang-koetter10} sparked
a lot of research interest in constant-dimension
subspace coding, focusing either on the
derivation of upper bounds for the numbers $\smax_q(v,2\delta;k)$ or
on explicit code constructions, which provide lower bounds for
$\smax_q(v,2\delta;k)$. An non-authoritative, non-exhaustive selection of
additional references is \cite{kohnert-kurz08,xia-fangwei09,etzion-vardy11a,etzion13b,cossidente-pavese16,storme-nakic16}.

The exact determination of $\smax_q(v,2\delta;k)$ seems to be a very difficult
problem even for moderate parameter sizes. To the best of our
knowledge, there are currently only two
parameter sets $(v,2\delta;k)_q$ with $1<t=k-\delta+1<k$,
where $\smax_q(v,2\delta;k)$ is known exactly: $\smax_2(6,4;3)=77$
\cite{smt:fq11proc} and $\smax_2(13,4;3)=1,597,245$
\cite{braun-etal13}. The $(13,1597245,4;3)_2$ code---in
fact there exist many non-isomorphic codes with these parameters---is
particularly remarkable, since it forms the first nontrivial example of a
Steiner system over a finite field (a $2$-analogue of the projective
plane of order $3$).

Our contribution to the restricted Main Problem in this paper pertains
also to the case $q=2$, $k=3$, $d=4$. In geometric terms, we consider
sets of planes in a binary projective geometry $\PG(V)\cong\PG(v-1,\F_2)$
mutually intersecting in at most a point. We refer to these sets as
(binary) \emph{plane subspace codes}, and our interest is in finding the
largest plane subspace code(s) over $\F_2$ with fixed packet length
$v$. Before stating our main result, we shall briefly review previous
work on binary plane subspace codes.

The exact results known in this case are the ``trivial'' (inasmuch as
it reduces to a case with $t=1$) $\smax_2(5,4;3)=\smax_2(5,4;2)=9$,
and the two already mentioned results $\smax_2(6,4;3)=77$,
$\smax_2(13,4;3)=1,597,245$. In the smallest open case $v=7$ we know
$329\leq\smax_2(7,4;3)\leq 381$.  The lower bound stems from the
computer-aided group-invariant construction in \cite{braun-reichelt14}
(cf.\ \cite{lt:0328,mt:alb80} for alternative constructions), and the
upper bound is the size of a putative $2$-analogue of the projective
plane of order $2$, whose existence is still undecided (despite the
known solution in the much larger case $v=13$). In addition, for
$7<v<13$ strong lower bounds for $\smax_2(v,4;3)$ are known from the
computational work in \cite{braun-ostergard-wassermann15}, which
employs dedicated combinatorial optimization techniques for
group-invariant subspace codes.

The best known constructive\footnote{For large $v$ a non-constructive
  lower bound, which asymptotically matches the upper bound stated in
  the next paragraph, has been shown in \cite{blackburn-etzion12}.}
lower bound for general $v$ is provided by the echelon-Ferrers
construction and its variants
\cite{etzion-silberstein09,trautmann-rosenthal10,etzion-silberstein13,silberstein-trautmann15}. It
asserts that
\begin{align*}
  \label{eq:LMRD}
  \smax_2(v,4;3)&\geq 2^{2(v-3)}+\gauss{v-3}{2}{2}
  =2^{2v-6}+\frac{(2^{v-3}-1)(2^{v-4}-1)}{3}\\
  &=2^{2v-6}+2^{2v-9}+2^{2v-11}+\dots+2^1+1-2^{v-4}
\end{align*}
for $v\leq 11$, with a slightly inferior bound for larger $v$. The
quantity $2^{2(v-3)}+\gauss{v-3}{2}{2}$ also provides an upper bound
for any $(v,M,4;3)_2$ code that contains a lifted maximum rank
distance code (LMRD code) as a subcode and hence represents
essentially the optimum achievable by the echelon-Ferrers construction
and its variants \cite{etzion-silberstein13}. Subsequently we will
refer to this upper bound as the \emph{LMRD code bound}.

The best known upper bound for unrestricted $(v,M,4;3)_2$ codes, a
consequence of the known maximum size of partial line spreads in
$\PG(v-1,\F_2)$, is
\begin{equation}
  \label{eq:ub}
   \smax_2(v,4;3)\leq
   \begin{cases}
     \left\lfloor\frac{(2^v-1)(2^{v-1}-1)}{21}\right\rfloor&\text{for $v\equiv
       1\pmod{2}$},\\[1ex]
     \left\lfloor\frac{(2^v-1)(2^{v-1}-5)}{21}\right\rfloor&\text{for $v\equiv
       0\pmod{2}$},
   \end{cases}
\end{equation}
and is substantially larger.\footnote{The maximum size of partial line
  spreads in $\PG(v-1,\F_q)$ is known for all prime powers $q>1$
  \cite{beutelspacher75}, leading to an analogous (best known) upper
  bound for $\smax_q(v,4;3)$.}  It can be verified that the base-$2$
representation of the upper bound has the form
$2^{2v-6}+2^{2v-7}+2^{2v-12}+2^{2v-13}+\dots+2^{2v-6s}+2^{2v-6s-1}+\text{smaller
  terms}$,
where $s=\lfloor(v-3)/6\rfloor$.\footnote{We can also express these
  bounds asymptotically for $v\to\infty$ as
  $2^{2v-6}\left(\frac{7}{6}+\littleo(1)\right)\leq\smax_2(v,4;3) \leq
  2^{2v-6}\left(\frac{32}{21}+\littleo(1)\right)$.}

Our main theorem, stated below, improves upon the echelon-Ferrers
construction and all its variants for $q=2$, $k=3$ and
infinitely many packet lengths
$v$. The codes are constructed using a generalization of the
expurgation-augmentation method introduced in
\cite{smt:fq11proc,lt:0328,mt:alb80}. Moreover, the augmentation step
can be done in such a way that the resulting codes are invariant under
a ($v-3$)-dimensional Singer subgroup of $\GL(v,2)$ acting trivially
on the complementary three coordinates. More precisely, the codes have
ambient space $V=W\times\F_{2^n}$, where $n=v-3$ and $W$ is a certain
$3$-dimensional $\F_2$-subspace of $\F_{2^n}$, and are invariant under
the group $\Sigma_v\leq\GL(V)$ consisting of all maps of the form
$(x,y)\mapsto(x,ry)$ with $r\in\F_{2^n}^\times$.

\begin{main}
  \label{thm:main}
  \begin{enumerate}[(i)]
  \item\label{thm:main:n=4mod8}
    For $v\equiv 7\pmod{8}$, there exists a $\Sigma_v$-invariant
    $(v,M,4;3)_2$ subspace code with
    \begin{equation*}
      \label{eq:main:n=4mod8}
      M\geq 2^{2(v-3)}+\frac{9}{8}\gauss{v-3}{2}{2},
    \end{equation*}
    and consequently we have $\smax_2(v,4;3)\geq
    2^{2(v-3)}+\frac{9}{8}\gauss{v-3}{2}{2}$ in this case.
  \item\label{thm:main:n=0mod8} For $v\equiv 3\pmod{8}$, $v\geq 11$,
    there exists a $\Sigma_v$-invariant $(v,M,4;3)_2$ subspace code
    with
    \begin{equation*}
      \label{eq:main:n=0mod8}
      M\geq 2^{2(v-3)}+\frac{81}{64}\gauss{v-3}{2}{2},
    \end{equation*}
    and consequently we have $\smax_2(v,4;3)\geq
    2^{2(v-3)}+\frac{81}{64}\gauss{v-3}{2}{2}$ in this case.
  \end{enumerate}
\end{main}

The route to our main theorem is long and involved, at least when
following the chronological order in which the various pieces were put
together. We have deliberately left this order untouched, since it
captures best the line of argument and the motivation for each
subsequent step. In order to make the paper essentially
self-contained, we provide an exposition of the basic and refined
expurgation-augmentation method as developed in
\cite{smt:fq11proc,lt:0328,mt:alb80}, including the key examples for
$q=2$. The theoretical analysis is supplemented\footnote{In fact the
  main result would have never been discovered without the
  computational data, which suggested the route pursued in later
  sections, and the need to make our computation more efficient.} by
extensive computations, which were done using the computer algebra
system SageMath (\texttt{www.sagemath.org}).\footnote{SageMath has
  proved to be an extremely versatile tool in our present research.}
In some cases the largest
subspace codes obtained during this optimization process considerably
exceed the bounds stated in the main theorem, and for packet lengths
$v>13$ they probably form the largest subspace codes explicitly known at
present. For details we refer to Table~\ref{tbl:netgain}.

In the remainder of this introduction we will provide a brief overview
of the route to the main theorem and at the same time explain how the rest
of this paper is organized.

The background of the expurgation-augmentation method, its previous
developments, and the adaption to packet lengths $v>7$ (the initial
stage of our research) is described in Sections~\ref{sec:prelim},
\ref{sec:basic}, and~\ref{sec:refined}. A key ingredient, dating back
to \cite{smt:fq11proc}, is a particular choice of the ambient space
$V\cong\F_2^v$, which takes the structure of the optimal
$(6,77,4;3)_2$ subspace codes into account and, similar to the Singer
representation, allows for using the multiplicative structure of a
large finite field: We always take $V$ as $W\times\F_{2^n}$, where
$n=v-3$ and $W$ is some $3$-subspace of $\F_{2^n}$, i.e., a plane in
$\PG(\F_{2^n})$.\footnote{``$\PG(\F_{2^n})$'' will be used as a
  shorthand for $\PG(\F_{2^n}/\F_2)$, the projective geometry derived
  from the field extension $\F_{2^n}/\F_2$.}

At the beginning, with a rough goal of generalizing the results of
\cite{mt:alb80}, we made computational experiments in the case $v=8$;
these were all but encouraging.\footnote{For $v=8$ the best subspace codes
  obtained by expurgation-augmentation have size $1024+93=1117$. Even
  if we were lucky to extend the codes by the theoretical maximum of
  $\gauss{5}{2}{2}=155$ further planes, we would still remain way below
  the best known $(7,1312,4;3)_2$ codes at that time. The currently
  best known $(7,M,4;3)_2$ has size $M=1326$; cf.\
  \cite{braun-ostergard-wassermann15}.}
But in fact, $v=8$ and $v=10$ are the only cases within the range
$v\in\{7,8,\dots,16\}$ now covered by the computational part of our work,
in which the refined (``rotation-invariant'') expurgation-augmentation
method is inferior to the echelon-Ferrers construction and its
variants. Our experiments used results in \cite[Sect.~5]{mt:alb80},
which express the number of planes that can be ``locally'' added to
the expurgated lifted Gabidulin code without decreasing the minimum
distance as the number of distinct values of a certain numerical 
invariant for planes in $\PG(\F_{2^n})$, named
\emph{$\sickson$-invariant} in \cite{mt:alb80}. 

Although the \emph{$\sickson$-invariant} and some of its properties
generalize to arbitrary $v$, the case $v\geq 8$ differs
fundamentally from the case $v=7$ considered in
\cite{mt:alb80} in the following respects:
\begin{itemize}
\item For $v\geq 8$ the Gabidulin code and the $\sickson$-invariant
  generally depend on the plane $W$ used in $V=W\times\F_{2^n}$. As a
  consequence, all planes $W$ in $\PG(\F_{2^n})$ have to be taken into
  account for the subspace code optimization and there will certainly
  be no analogue of the nice explicit formula for the
  $\sickson$-invariant established for $v=7$ in
  \cite[Lemma~7]{mt:alb80}.\footnote{Subsequently we will write
    $\gabidulin_W$, $\sickson_W$ to indicate the dependence of the
    Gabidulin code, respectively, $\sickson$-invariant on $W$.}
\item For $v\geq 8$ there is no longer a canonical choice for the
  expurgated Gabidulin code in the refined
  expurgation-augmentation method. Instead there are $2^{v-6}-1$
  minimal subsets of $\gabidulin_W$, any combination of which can be
  removed to obtain an expurgated Gabidulin code. The subspace code
  optimization algorithm should consider all such combinations and
  select the best one.
\end{itemize}

With these two guidelines at hand, we created a simple prototype of
the subspace code optimization algorithm, which generated planes $W$
randomly, evaluated the associated invariant $\sickson_W$ on all
planes intersecting $W$ in a line,\footnote{The domain of $\sickson_W$
  consists precisely of those planes.} and computed an optimal
solution of the resulting optimization problem---maximize the
difference between the image size of $\sickson_W$ and the (suitably
normalized\footnote{The total number of planes removed from
  $\lgabidulin_W$ is divided by the number of points on the special
  flat $S$ defined in Section~\ref{sec:prelim}, in order to match the
  present ``local'' point-of-view.}) number of planes removed from
$\lgabidulin_W$, subsequently referred to as the \emph{local
  net gain}---in a brute-force manner by exhaustive search through all
combinations of minimal subsets of $\gabidulin_W$. With this algorithm
and the additional observation that the problem setting is invariant
under a fairly large group of collineations of $\PG(\F_{2^n})$ acting
on the set of planes $W$, we were able to solve the cases $v=8,9,10$
completely and do a partial search for length $v=11$. It turned out
that the initial estimate, based on $v=8$, had been too pessimistic;
for $v=9,11$ our algorithm found solutions which exceeded the LMRD
code bound.

The next step was to replace the exhaustive search through all
$2^{2^{v-6}-1}-1$ nonempty combinations of minimal subsets, which is
clearly prohibitive for $v=11$, by something more efficient. For this
we inspected the data structure containing the computed values of
$\sickson_W$, a matrix of size $(2^{n-3}-1)\times(2^n-1)$, $n=v-3$,
indexed with the solids ($4$-subspaces) $T$ in $\PG(\F_{2^n})$
containing $W$ and the elements $a\in\F_{2^n}^\times$, which has as
$(i,j)$-entry the number of planes $E$ such that $E+W=T_i$ and
$\sickson(E)=a_j$.\footnote{The solids correspond to the minimal
  subsets of $\gabidulin_W$ that can be combined. Hence it is not
  necessary to distinguish between different planes in $T_i$ at this
  stage.}  The matrix turned out to have a very special
structure. Obviously it is divided into two parts representing
elements $a\in\F_{2^n}^\times$ with $\leq 1$ and $>1$ preimages $E$
under $\sickson_W$, respectively, and an optimal solution of the
optimization problem must include all planes $E$ of the first
kind.\footnote{That is, there is no plane $E'\neq E$ intersecting $W$
  in a line and such that $\sickson(E)=\sickson(E')$.} Hence attention
can be restricted to the second part, which turned out to be a square
matrix of order $2^{n-3}-1$ with columns indexed by an
($n-3$)-subspace (``collision space'') of $\F_{2^n}$. This submatrix
and subspace will be called \emph{collision matrix}, respectively,
\emph{collision space} of $W$ or $\sickson_W$; cf.\
Theorem~\ref{thm:cspace} and Definition~\ref{dfn:cspace}. As it turned
out, collision matrices have only $3$ different column shapes with
nonzero entry patterns $4^1$, $2^3$, $1^7$ and such that the supports
of each column, viewed as a set of points in $\PG(\F_{2^n}/W)$, forms
a subspace (point, line, or plane) as well. The proof of these
peoperties, which were first noticed through experiments, is given in
Theorem~\ref{thm:C}.

The rather difficult proofs of the preceding (and also subsequent)
observations use properties of so-called subspace polynomials and in
particular the linear (i.e., degree one) coefficients of such
polynomials, which are analogous to the elementary symmetric
polynomials $\sigma_k(X)=X_1X_2\dotsm X_k$. The basic theory of
subspace polynomials and their coefficients, the ``Dickson
invariants'', has been developed long ago by \name{L.E.~Dickson} and
\name{O.~Ore}, cf.\ \cite{dickson11,ore33a}. We provide a brief
account of this theory in Section~\ref{sec:allthat}, tailored to the
case $q=2$ and including some new (or at least less well-known)
results, which are needed in subsequent sections. 

Subspace polynomials have recently been used in other
contexts~\cite{ben-sasson-etal10,ben-sasson-kopparty12,cheng-gao-wan12},
including a direct application to subspace
coding~\cite{ben-sasson-etal14}. We also offer a tiny but curious
result in this direction (Corollary~\ref{cor:allthat}).

Armed with the new theoretical insight we were able to reformulate the
local net gain maximization as a combinatorial optimization problem
with a fairly rich structure provided by the collision matrix
(Theorem~\ref{thm:copt}, Corollary~\ref{cor:copt}),\footnote{The
  problem bears some similarity to the set cover problems studied in
  Theoretical Computer Science; see \cite{garey-johnson79}.} and solve
it completely for packet lengths $v\leq 13$. This and further
non-exhaustive computations for $v\in\{14,15,16\}$ confirmed that the
expurgation-augmentation method produces codes larger than the LMRD
code bound for $v\neq 8,10$; see Section~\ref{sec:comp}.

The last (but not least) steps towards the main theorem were the
following: Along with the maximum net gain computations we had recorded
some statistical data for the collision matrices, with the goal of
understanding which algebraic properties of $W$ are responsible for a
large maximum net gain. From this we noticed that the ``best'' planes
were those with an entry ``4'' in their collision matrix, a property
that can be described algebraically (Theorem~\ref{thm:cij=4}), and
with the largest number of so-called ``missing points'' (this concept
is defined at the beginning of Section~\ref{sec:cont}) in their
collision space. The best example for $v=11$, in which $W$ is equal to
the trace-zero subspace of the subfield $\F_{16}\subset\F_{2^8}$ and
whose $31\times 31$ collision matrix is shown in
Figure~\ref{fig:CWn=8}, helped us to understand that the geometric
configuration formed by the multiset of missing points in the collision
space determines the row-sum spectrum of the collision matrix and hence, via
Corollary~\ref{cor:copt}, to some extent controls the maximum
achievable net gain. The precise relation is described in
Theorem~\ref{thm:HZ}. Moreover, the trace-zero subspace of $\F_{16}$ provides a
natural candidate for $W$ in all cases $v\equiv 3\pmod{4}$, since for
such $v$ the field $\F_{2^n}$ contains $\F_{16}$. The final steps
where to show that the maximum net gain achievable with the trace-zero
subspace satisfies the bounds stated in the main theorem. This is
accomplished in Section~\ref{sec:bingo}, first in the case $v\equiv
7\pmod{8}$, which is considerably easier, (Theorem~\ref{thm:n=4mod8})
and then in the case $v\equiv 3\pmod{8}$ (Theorem~\ref{thm:n=0mod8}).

Theorem~\ref{thm:HZ} encompasses the nice fact that the row-sum
spectrum of a collision matrix can be computed in much the same way as
the weight distribution of a linear code from geometric information
about an associated multiset of points in some projective
geometry. This connection, first described after Theorem~\ref{thm:HZ}
and used in the proofs of Theorems~\ref{thm:n=4mod8}
and~\ref{thm:n=0mod8}, is made more explicit in
Section~\ref{sec:code}. Here we show, by bounding the maximum net gain
in terms of a certain quantity (``code sum'') and estimating this code
sum for all projective binary linear $[\mu,k]$ coes of length
$\mu\leq 7$, that the non-exhaustively computed maximum net gains for
$v=14,15$ in Table~\ref{tbl:netgain} represent the true maximum
(Theorem~\ref{thm:n=11,12}).

Based on the accumulated computational data, we conjecture that the largest
subspace codes obtained by the
expurgation-augmentation method exceed the LMRD code bound for all
sufficiently large packet lengths $v$
(Conjecture~\ref{conj:LMRD}). For odd $v$ this conjecture is strongly supported
by computational data on the distribution of missing points in the
collision space; see the end of Section~\ref{sec:bingo}.

The paper concludes with Section~\ref{sec:conc}, which provides a
discussion of some in a sense ``neglected'' aspects of our work and
gives some suggestions for future research.

Some familiarity with basic concepts and terminology from Finite
Geometry is indispensable for understanding this paper. The relevant
background information can be found in \cite{dembowski68},
\cite{hirschfeld98}, or \cite{nova2011}. Regarding notation,
we only mention at this point the abbreviation
$\trace(x)=\trace_{\F_{2^n}/\F_2}(x)=x+x^2+x^4+\dots+x^{2^{n-1}}$ for
the absolute trace of $\F_{2^n}$, which is used frequently in the
sequel. All other non-standard notation will be explained on its first
occurrence.


\section{Preliminaries on Plane Subspace Codes}\label{sec:prelim}

Throughout this section let $\mathcal{C}$ be a plane subspace code
with parameters $(v,M,4;3)$ and ambient space $V$, where
w.l.o.g.\ $v\geq 6$. Since planes in
$\mathcal{C}$ do not have a line in common, $\mathcal{C}$ covers
(i.e., its members contain) precisely $7M$ lines of
$\PG(V)$. Conversely, if we know the number $l$ of lines covered by
$\mathcal{C}$, we can recover the size of $\mathcal{C}$ as
$M=l/7$. Hence maximizing $M$ and $l$ are equivalent problems.

This point of view is especially useful when looking at lifted maximum
rank distance codes (LMRD codes) with these parameters. Such codes are
obtained from maximum rank-distance-$2$ matrix codes in
$\F_2^{3\times(v-3)}$, e.g.\ Gabidulin codes, by the \emph{lifting
  construction} $\mat{A}\mapsto\langle(\imat_3|\mat{A})\rangle$
(``prepending the $3\times 3$ identity matrix to $\mat{A}$ and then
taking the row space'') and have size $2^{2(v-3)}$. The planes
obtained in this way are disjoint from the special ($v-3$)-dimensional
subspace $S=\{\vek{x}\in\F_2^v;x_1=x_2=x_3=0\}$ and, as remarked
above, cover $7\cdot 2^{2(v-3)}$ lines of $\PG(\F_2^v)$. On the other
hand, standard counting facts in finite projective spaces imply that
the total number of lines in $\PG(\F_2^v)$ disjoint from $S$ is also
$7\cdot 2^{2(v-3)}$. Hence any LMRD code in $\F_2^v$ forms a perfect
cover of the set of lines disjoint of $S$ (and conversely, any such
perfect cover arises from a maximum rank-distance-$2$ matrix code in
the way described).

This leads directly to the LMRD code bound mentioned in
Section~\ref{sec:intro}: If $\mathcal{C}$ contains an LMRD code then
it cannot contain planes meeting $S$ in a point (since these contain lines
disjoint from $S$) and hence contains at most $\gauss{v-3}{2}{2}$
further planes, one for each line contained in $S$.\footnote{Since
  a plane contained in $S$ covers $7$ lines, it is also clear that
  $\mathcal{C}$ meets the bound with equality iff it has a subcode
  forming a perfect cover of the lines in $S$.}

In order to overcome the LMRD code bound, we should therefore start
with a smaller set of planes disjoint from $S$. It is reasonable to
choose this set as a large subcode of a Gabidulin code, and it has
been shown in \cite{smt:fq11proc,lt:0328,mt:alb80} that this idea can
indeed be put to work for $v=6,7$. The method, which we call
\emph{expurgation-augmentation}---remove some ``old'' planes from the
Gabidulin code (``expurgate'' the Gabidulin code) and add in turn some
``new'' planes meeting $S$ in a point (``augment'' the expurgated
Gabidulin code)---, is described in the next section. Crucial for the
success of the method is a particular choice of the ambient space $V$,
which involves a large extension field of $\F_2$ and allows us later
to employ properties of linearized polynomials and the multiplicative
structure of the extension field. This choice of $V$ will be discussed in the
remainder of this section.


The ambient space for the expurgation-augmentation method is taken as
$V=W\times\F_{2^n}$, $n=v-3$, for some $3$-dimensional
$\F_2$-subspace $W$ of $\F_{2^n}$.\footnote{Note that $v-3\leq n$ in view
  of our assumption $v\geq 6$.}  In this model the special
($v-3$)-subspace is $S=\{0\}\times\F_{2^n}$, and $W$ is
represented within $V$ as $\Wtilde=W\times\{0\}$.

The
corresponding Gabidulin code can be defined in a basis-independent
manner as $\gabidulin_W=\{a_0x+a_1x^2;a_0,a_1\in\F_{2^n}\}$, where
$a_0x+a_1x^2$ is used as an abbreviation for the $\F_2$-linear map
$W\to\F_{2^n}$, $x\mapsto a_0x+a_1x^2$. The lifted Gabidulin code in
this model is $\lgabidulin_W=\{\graph_f;f\in\gabidulin_W\}$, where
$\graph_f=\bigl\{(x,f(x));x\in W\bigr\}$ denotes the graph of $f$ (in
the sense of Real Analysis, if you like).

The $7\cdot 2^{2(v-3)}=7\cdot 2^{2n}$ lines covered by $\lgabidulin_W$
are precisely the graphs of the restrictions
$f|_Z\colon Z\to\F_{2^n}$, $x\mapsto f(x)$ of $f\in\gabidulin_W$ to lines
($2$-subspaces) $Z\subset W$. The perfect cover property
of $\lgabidulin_W$ is reflected in the fact that the maps
$\gabidulin_W\to\Hom(Z,\F_{2^n})$, $f\mapsto f|_Z$ are linear
isomorphisms,\footnote{As usual, $\Hom(X,Y)$ denotes the vector space
  of all linear maps from $X$ to $Y$. Here the ground field is $\F_2$,
  and $\Hom(Z,\F_{2^n})$ is considered as a vector space over
  $\F_2$ (although it is also a vector space over $\F_{2^n}$).}  and
hence any line disjoint from $S$, which is the graph of a unique
$\F_2$-linear map $g\colon Z\to\F_{2^n}$ for some $Z$, is covered
precisely once by $\lgabidulin_W$.

More generally, any $\F_2$-subspace $U$ of $V$ can be
parametrized in the form
\begin{equation}
  \label{eq:param}
  U=\bigl\{(x,f(x)+y);x\in Z,y\in T,f\in\Hom(Z,\F_{2^n})\bigr\},
\end{equation}
where $Z\subseteq W$, $T\subseteq\F_{2^n}$ are $\F_2$-subspaces and
$f$ is an $\F_2$-linear map. We write $U=\U(Z,T,f)$ in this case.
The spaces $Z$, $W$ are recovered from
$U$ as $Z=\bigl\{x\in W;\exists y\in\F_{2^n}\text{ such that }(x,y)\in
U\bigr\}$ and $T=\bigl\{y\in\F_{2^n};(0,y)\in U\bigr\}$. The map $f$
can be any element of $\Hom(Z,\F_{2^n})$ satisfying $\graph_f\subseteq
U$ and corresponds to a complement for $U\cap S=\{0\}\times T$
in $U$ via $f\mapsto\graph_f$.\footnote{An explicit map $f$ is obtained by
  choosing a basis $B$ of $Z$ and defining $f(b)$ as any $y$ such that
  $(b,y)\in U$.} Further,
we have $\U(Z,T,f)=\U(Z',T',f')$ if and only if $Z=Z'$, $T=T'$ and
$f-f'\in\Hom(Z,T)$ (i.e., $f(x)-f'(x)\in T$ for all $x\in Z$). The
parametrization $U=\U(Z,T,f)$ thus induces a 1-1 correspondence
between $\F_2$-subspaces of $V$ and triples
$\bigl(Z,T,f+\Hom(Z,T)\bigr)$. Finally,
the incidence relation on subspaces of $V$ translates into the
following conditions for the parameters:
$\U(Z,T,f)\subseteq\U(Z',T',f')$ iff $Z\subseteq Z'$, $T\subseteq T'$
and $f'|_Z-f\in\Hom(Z,T')$.

\section{The Basics of Expurgation-Augmentation}\label{sec:basic}

The underlying geometric idea is to find sets of planes in
$\lgabidulin_W$, whose lines can be rearranged into new planes meeting
$S$ in a point. Removing the planes in such a set from $\lgabidulin_W$
and adding in turn the new planes to the expurgated subspace code
preserves the exact cover property with
respect to lines disjoint from $S$. Moreover, if $t$ planes are
removed then $7t$ lines disjoint from $S$ are involved and, since new
planes contain only $4$ such lines, the subspace code size increases
by $\frac{7t}{4} - t = \frac{3t}{4}$.\footnote{Let us keep in mind
  that for beating the LMRD code bound we should have
  $t\geq\frac{4}{3}\gauss{n}{2}{2}=
  \frac{(2^{n+1}-1)(2^n-1)}{9}\approx\frac{2}{9}\#\gabidulin_W$.}
However, we must be careful to avoid any multiple cover of a line
meeting $S$ in a point.

Planes of $\PG(V)$ meeting $S$ in a point $P=\F_2(0,r)$ are parametrized as
$N=\U(Z,\F_2r,g)$ for some line $Z$ in $W$ and some linear map $g\colon
Z\to\F_{2^n}$. Using the natural isomorphism
$S=\{0\}\times\F_{2^n}\cong\F_{2^n}$, we may view $P$ as a point of
$\PG(\F_{2^n})$ and write the parametrization as
$N=\U(Z,P,g)$. The line $Z$ specifies the
hyperplane $H=N\vee S=\U(Z,\F_{2^n},0)$ above $S$ that contains $N$,
and the $4$ lines in $N$ disjoint from $S$ are the graphs of the
rank-distance-$1$ clique $g+\Hom(Z,P)$.
At this point it comes in handy that $f\mapsto f|_Z$ identifies
$\gabidulin_W$ with $\Hom(Z,\F_{2^n})$. The associated map on planes is
$\graph_f\mapsto\graph_{f|_Z}=\graph_f\cap H$ and can be used to
determine exactly the set of $4$ planes in $\lgabidulin_W$ that
determine the $4$ lines $\graph_h$, $h\in g+\Hom(Z,P)$.

Before stating the criterion for exact rearrangement, it will be
convenient to introduce a few special $\F_2$-subspaces of
$\gabidulin_W$. We define
\begin{align}
  \mathcal{T}&=\{ux^2+u^2x;u\in W\},\\
  \mathcal{R}&=\{ux^2+u^2x;u\in\F_{2^n}\},\\
  \mathcal{D}(Z,P)&=\{f\in\gabidulin_W;f(Z)\subseteq P\},
\end{align}
the latter with $Z$, $P$ having the same meaning as above.
Since the nonzero maps in $\mathcal{R}$ have the factorization
$ux(x+u)$, we have $\kernel(f)=\F_2u$, $\rank(f)=2$ if
$f\in\mathcal{T}$ and $\kernel(f)=\{0\}$, $\rank(f)=3$ if
$f\in\mathcal{R}\setminus\mathcal{T}$.\footnote{In the
  original version in \cite{smt:fq11proc}, which was
  designed for $v=6$, the subspace $W$ is equal to
  $\F_{2^n}=\F_8$ and $\mathcal{R}$ coincides with $\mathcal{T}$.}
Further, since $\mathcal{D}(Z,P)$ is mapped to $\Hom(Z,P)$ by $f\mapsto f|_Z$,
it is clear that $\#\mathcal{D}(Z,P)=4$. Writing $Z=\langle
a,b\rangle=\{0,a,b,a+b\}$, it is easy to verify that
$\mathcal{D}\bigl(Z,\F_2(ab^2+a^2b)\bigr)=\{0,ax^2+a^2x,bx^2+b^2x$,
$(a+b)x^2+(a+b)^2x\}$ and, using the factorized form,
\begin{equation*}
  \mathcal{D}(Z,\F_2r)=\left\{0,\frac{rax(x+a)}{ab(a+b)},
    \frac{rbx(x+b)}{ab(a+b)},\frac{r(a+b)x(x+a+b)}{ab(a+b)}\right\}
\end{equation*}
in general.

The preceding considerations imply the following
\begin{lemma}[{cf.\ \cite[Lemma~10]{smt:fq11proc}
    and \cite[Lemma~4]{mt:alb80}}]\label{lma:remove}
  Let $\mathcal{A}\subseteq\gabidulin_W$ be a subset of size $t$. The
  $7t$ lines contained in the members of
  $\{\graph_f;f\in\mathcal{A}\}\subseteq\lgabidulin_W$ can be
  rearranged into $7t/4$ new planes meeting $S$ in a point if and only if
  $t=4m$ is a multiple of $4$ and for every line $Z\subset W$ there
  exist (not necessarily distinct) points $P_1,\dots,P_m\in S$ and
  linear maps $f_1,\dots,f_m\in\mathcal{A}$ such that
  \begin{equation*}
    \mathcal{A}=\biguplus_{i=1}^m\bigl(f_i+\mathcal{D}(Z,P_i)\bigr).
  \end{equation*}
\end{lemma}
In other words, $\mathcal{A}$ should admit decompositions into
disjoint cosets of spaces $\mathcal{D}(Z,\cdot)$ simultaneously for
each $Z$. The points $P_1,\dots,P_m$ may coincide, in which case the
condition reduces to a representation of $\mathcal{A}$ as a union of
cosets of some space $\mathcal{D}(Z,P)$.

The criterion in Lemma~\ref{lma:remove} seems to be rather complicated to
check and a description of all such rearrangements for any given subset
$\mathcal{A}$ out of reach. However, there is an obvious candidate for
$\mathcal{A}$ that admits a simultaneous decomposition of
the required form, viz.\ the space $\mathcal{T}$, which contains the
$7$ subspaces $\mathcal{D}\bigl(Z,\F_2(ab^2+a^2b)\bigr)$ and hence
decomposes into $2$ cosets of each of them. The corresponding
rearrangement is clearly unique, and by using $\mathcal{T}$ as the
basic building block we obtain a large number of sets $\mathcal{A}$
satisfying the condition in Lemma~\ref{lma:remove} for a specific
decomposition. This will be sufficient for our purposes.

Every binomial $a_0x+a_1x^2\in\gabidulin_W$ is uniquely represented as
$r(ux^2+u^2x)$ with $r,u\in\F_{2^n}^\times$ (i.e., as $rf$ with
$r\in\F_{2^n}^\times$,
$f\in\mathcal{R}\setminus\{0\}$).\footnote{Solving
  $a_1x^2+a_0x=rux^2+ru^2x$ for $R,u$ gives $u=a_0/a_1$, $r=a_1^2/a_0$.} Hence
$\gabidulin_W$ consists of $0$, the $2(2^n-1)$ monomials $rx$,
$rx^2$, which have rank $3$, the $7(2^n-1)$
rank-2 binomials $rf$ with $f\in\mathcal{T}\setminus\{0\}$, and the
$(2^n-8)(2^n-1)$ rank-3 binomials $rf$ with
$f\in\mathcal{R}\setminus\mathcal{T}$. The subset of rank-3 binomials
decomposes into $(2^{n-3}-1)(2^n-1)$ pairwise disjoint
``rotated'' cosets $r(f+\mathcal{T})$, where $r\in\F_{2^n}^\times$ and
$f\in\mathcal{R}\setminus\mathcal{T}$ is determined modulo
$\mathcal{T}$.

Just like $\mathcal{T}$, the set $r(f+\mathcal{T})$ admits a unique
simultaneous decomposition into $2$ cosets of
$\mathcal{D}\bigl(Z,\F_2r(ab^2+a^2b)\bigr)$.  Extending this in the
obvious way to unions of rotated cosets, we see that any such union
admits a simultaneous decomposition into cosets of spaces
$\mathcal{D}\bigl(Z,\cdot)$, as required in
Lemma~\ref{lma:remove}. Hence we have the following

\begin{lemma}
  \label{lma:standard}
  Suppose $\mathcal{A}\subseteq\gabidulin_W$ is the union of some of
  the $(2^{n-3}-1)(2^n-1)$ rotated cosets $r(f+\mathcal{T})$
  ($r\in\F_{2^n}^\times$, $f\in\mathcal{R}\setminus\mathcal{T}$) and
  at most one rotated subspace $r\mathcal{T}$. Then the lines in
  $\{\graph_f;f\in\mathcal{A}\}$ can be exactly rearranged into new
  planes meeting $S$ in a point.
\end{lemma}
The corresponding ``obvious'' exact rearrangement of free lines into
new planes will be called the \emph{standard
  rearrangement}.\footnote{``Free line'' refers to a line covered by
  $\{\graph_f;f\in\mathcal{A}\}$. After removal of this set of planes
  from $\gabidulin_W$, such a line is uncovered, i.e.,
  ``free''. This name was coined in \cite{smt:fq11proc}.}

The previous construction provides us with myriads of subsets
$\mathcal{A}\subseteq\gabidulin_W$ satisfying the conditions
in Lemma~\ref{lma:remove}, but it does not tell us
whether the set $\mathcal{N}$ of
new planes of the corresponding standard rearrangement has
$\sdist(\mathcal{N})\geq 4$. We are interested in the
largest subsets $\mathcal{A}$ having this extra property, since
the size $\#\mathcal{C}=4^{v-3}+3t/4$ of the modified subspace code
\begin{equation*}
  \mathcal{C}=\lgabidulin_W-\{\graph_f;f\in\mathcal{A}\}\cup\mathcal{N},
\end{equation*}
which then has $\sdist(\mathcal{C})\geq 4$ as well,
is an increasing function of $t=\#\mathcal{A}$.

\begin{RP}
  Determine the subsets $\mathcal{A}\subseteq\gabidulin_W$ of maximum
  size that are unions of pairwise disjoint rotated cosets
  $r(f+\mathcal{T})$ (as in Lemma~\ref{lma:standard}) and whose
  standard rearrangement into new planes forms a subspace code
  $\mathcal{N}$ with $\sdist(\mathcal{N})\geq 4$.
\end{RP}
We are not able to solve the rearrangement problem, but we will
exhibit fairly large subsets $\mathcal{A}$ with this property (cf.\
Theorem~\ref{thm:remove} below). The resulting modified subspace
codes, however, are still inferior to those produced by the echelon-Ferrers
construction (although it is conceivable that they can be extended by
$\approx\gauss{v-3}{2}{2}$ further planes meeting $S$ in a line to a
code exceeding the LMRD code bound). That notwithstanding, the
preparations made en-route to Theorem~\ref{thm:remove} will be needed
for the refined approach taken up in Section~\ref{sec:refined}.

Before proceeding, it will be convenient to discuss some properties of
the map $\dickson\colon\F_{2^n}\times\F_{2^n}\to\F_{2^n}$,
$(x,y)\mapsto xy^2+x^2y=xy(x+y)$. The map $\dickson$ is
$\F_2$-bilinear, symmetric, and alternating (i.e., $\dickson(x,x)=0$ for
$x\in\F_{2^n}$). Fixing the second argument, say, we have that
$x\mapsto\dickson(x,y)$, $y\neq 0$, is $\F_2$-linear with kernel
$\F_2y$ and hence induces a collineation from the quotient geometry
$\PG(\F_{2^n})/P$, $P=\F_2y$, onto some hyperplane $H$ in
$\PG(\F_{2^n})$.\footnote{The points and lines of
  $\PG(\F_{2^n})/P\cong\PG(\F_{2^n}/\F_2y)$ are the lines and planes
  of $\PG(\F_{2^n})$ through $P$, respectively, and the incidence
  relation is the induced one.} Using
$xy^2+x^2y=y^3\bigl(x/y+(x/y)^2\bigr)$ and Hilbert's Satz~90, this
hyperplane is easily seen to have equation $\trace(x/y^3)=0$. The
factorized form $\dickson(x,y)=xy(x+y)$ reveals that $\dickson(x,y)$
is equal to the product of the three (nonzero) points on the line
$L=\langle x,y\rangle=\{0,x,y,x+y\}$ and thus provides a second
geometric interpretation of $\dickson(x,y)$. In particular, by setting
$\dickson(L)=xy(x+y)$ we obtain a map from lines to points of
$\PG(\F_{2^n})$. The following property of this map turns out to be
crucial for the subsequent development; cf.\ Section~\ref{sec:allthat}.

\begin{lemma}
  \label{lma:dickson}
  $L\mapsto\dickson(L)$ maps the lines contained in any plane $E$ of
  $\PG(\F_{2^n})$ bijectively onto the points of another plane
  $E'$. Moreover, the induced map from $\PG(E)$ to $\PG(E')$ is a
  correlation (i.e., an incidence reversing bijection mapping lines to
  points and points to lines).
\end{lemma}
\begin{proof}
  Writing $E=\langle a,b,c\rangle$, we must show that
  $E'=\langle ab^2+a^2b,ac^2+a^2c,bc^2+b^2c\rangle
  =\bigl\langle\dickson(a,b),\dickson(a,c),\dickson(b,c)\bigr\rangle$
  has the required property. The lines of $E$ are
  $L_1=\overline{a,b}$, $L_2=\overline{a,c}$, $L_3=\overline{b,c}$,
  $L_4=\overline{a,b+c}$, $L_5=\overline{b,a+c}$,
  $L_6=\overline{c,a+b}$, and $L_7=\overline{a+b,a+c}$. Using
  the stated properties of $(x,y)\mapsto\dickson(x,y)$,
  we obtain $\dickson(L_4)=\dickson(a,b)+\dickson(a,c)$,
   $\dickson(L_5)=\dickson(a,b)+\dickson(b,c)$,
   $\dickson(L_6)=\dickson(a,c)+\dickson(b,c)$, $\dickson(L_7)
   =\dickson(a,b)+\dickson(a,c)+\dickson(b,c)$. Together with
   $\dickson(L_i)\neq 0$ for $1\leq i\leq 7$ this shows that $E'$ is
   indeed a plane (i.e., $\dickson(a,b)$, $\dickson(a,c)$,
   $\dickson(b,c)$ are linearly independent) and contains precisely
   the points $\dickson(L_i)$, $1\leq i\leq 7$, as claimed. Finally,
   $L\mapsto\dickson(L)$ maps the three lines in $E$ through a fixed
   point $P$ onto some line in $E'$ (since it induces a collineation
   $\PG(\F_{2^n})/P\to H$), proving the last assertion.
\end{proof}
\begin{remark}
  \label{rmk:dickson}
  For any line $L=\langle a,b\rangle$ in $\PG(\F_{2^n})$ we may form
  the \emph{line polynomial} $\spol_L(X)=\prod_{u\in
    L}(X-u)=X(X+a)(X+b)(X+a+b)=(X^2+aX)\bigl(X^2+aX+b(a+b)\bigr)
  =X^4+(a^2+ab+b^2)X^2+ab(a+b)X\in\F_{2^n}[X]$. The coefficients of
  $\spol_L(X)$, viewed as polynomials in $\F_2[a,b]$, are
  $\GL(2,\F_2)$-invariants and freely generate the invariant ring
  $R=\F_2[a,b]^{\GL(2,\F_2)}$ in the sense that
  $R=\F_2[a^2+ab+b^2,ab(a+b)]$ is a polynomial ring. An analogous
  result holds for arbitrary subspaces (in place of lines) and prime
  powers $q>1$ (in place of $2$). This $q$-analogue of the fundamental
  theorem for symmetric polynomials is due to Dickson
  \cite{dickson11}, and the coefficients $\delta_i^{(k)}$ of the
  generic $k$-dimensional subspace
  polynomial
  $\spol_U(X)=X^{q^k}-\dickson_1^{(k)}X^{q^{k-1}}\pm\dots+(-1)^k\dickson_k^{(k)}X$
  are accordingly referred to as $q$-ary, $k$-dimensional \emph{Dickson
    invariants}. Thus $\dickson(a,b)=ab^2+a^2b=ab(a+b)$ is equal to
  the second binary $2$-dimensional Dickson invariant (``line
  invariant'') $\dickson_2^{(2)}$.
\end{remark}

Now we resume our analysis of the rearrangement problem. Since
$f\in\mathcal{R}$ has the form $f(x)=ux^2+u^2x=\dickson(u,x)$ and
$f\in\mathcal{T}$ iff $u\in W$, we can write $r(f+\mathcal{T})$ as
$r\bigl(\dickson(u,x)+\dickson(W,x)\bigr)=r\dickson(u+W,x)$.  The $14$
new planes obtained from $r\dickson(u+W,x)$ by the standard
rearrangement are $\U\bigl(Z,r\dickson(Z),r\dickson(u,x)\bigr)$ and
$\U\bigl(Z,r\dickson(Z),r\dickson(u+c,x)\bigr)$ with $Z$ varying
over the $7$ lines in $W$ and $c\in W\setminus Z$ (thus $c$ depends
on $Z$). Two distinct new planes (not necessarily from the same
rotated coset) have a line in common if and only if they pass through
the same point in $S$ and have another point outside $S$ in
common. The $12$ points outside $S$ covered by
$\U\bigl(Z,r\dickson(Z),r\dickson(u,x)\bigr)$ and
$\U\bigl(Z,r\dickson(Z),r\dickson(u+c,x)\bigr)$, respectively, are
\begin{equation}
  \label{eq:cpoints}
  \begin{aligned}
    &\bigl(a,r\dickson(u,a)\bigr)&&\bigl(a,r\dickson(u+c,a)\bigr)\\
    &\bigl(a,r\dickson(u+b,a)\bigr)&&\bigl(a,r\dickson(u+c+b,a)\bigr)\\
    &\bigl(b,r\dickson(u,b)\bigr)&&\bigl(b,r\dickson(u+c,b)\bigr)\\
    &\bigl(b,r\dickson(u+a,b)\bigr)&&\bigl(b,r\dickson(u+c+a,b)\bigr)\\
    &\bigl(a+b,r\dickson(u,a+b)\bigr)&&\bigl(a+b,r\dickson(u+c,a+b)\bigr)\\
    &\bigl(a+b,r\dickson(u+a,a+b)\bigr)&&\bigl(a+b,r\dickson(u+c+a,a+b)\bigr)
  \end{aligned}
\end{equation}
(listed column-wise and writing $Z=\langle a,b\rangle$ as before).
\begin{theorem}
  \label{thm:remove}
  \begin{enumerate}[(i)]
  \item\label{thm:remove:1}
    The standard rearrangement of the $2^n$, $n=v-3$, planes in
    $\{\graph_f;f\in\mathcal{R}\}$ forms a set $\mathcal{N}$ of
    new planes satisfying $\sdist(\mathcal{N})\geq 4$.
  \item\label{thm:remove:2} If $r_1,\dots,r_s\in\F_{2^n}^\times$ are
    such that $r_iW'\cap r_jW'=\emptyset$ for $1\leq i<j\leq s$, then
    the standard rearrangement of the $2^n+(s-1)(2^n-8)$ planes
    $\graph_f$,
    $f\in r_1\mathcal{R}\uplus r_2(\mathcal{R}\setminus\mathcal{T})
    \uplus\dots\uplus r_s(\mathcal{R}\setminus\mathcal{T})$
    satisfies $\sdist(\mathcal{N})\geq 4$.
  \item\label{thm:remove:3} If $v\equiv0\pmod{3}$ and $W$ is chosen as
    the subfield $\F_8\subset\F_{2^n}$, then the standard
    rearrangement of the $2^n+(2^n-8)^2/7$ planes $\graph_f$,
    $f\in\mathcal{R}\uplus r(\mathcal{R}\setminus\mathcal{T})
    \uplus\dots\uplus
    r^{(2^n-1)/7-1}(\mathcal{R}\setminus\mathcal{T})$
    satisfies $\sdist(\mathcal{N})\geq 4$ and yields a modified
    subspace code $\mathcal{C}$ with parameters
    $\bigl(v,4^{v-3}+\frac{3}{7}(4^{v-4}-9\cdot
    2^{v-5}+16),4;3\bigr)_2$.
  \end{enumerate}
\end{theorem}
\begin{proof}
  \eqref{thm:remove:1} Consider the $4$ points in \eqref{eq:cpoints} with first
  coordinate $a$. If $u$ varies over a set of coset representatives for
  $\F_{2^n}/W$, then the second coordinate of the $4$
  points takes precisely the values $\dickson(x,a)$ with $x$ varying
  over a set of coset representatives for $\F_{2^n}/\F_2a$. Since
  $\dickson$ is one-to-one on $\F_{2^n}/\F_2a$, these $2^{n-1}$
  values, and hence also the $2^{n-1}$ points $\bigl(a,\dickson(x,a)\bigr)$,
  are distinct. This reasoning applies to the points with first
  coordinate $b$ or $a+b$ as well and shows that the $2^{n-2}$ planes
  in $\mathcal{N}$ with the same $Z$ pairwise intersect only in
  the point $\bigl(0,\dickson(Z)\bigr)\in S$. But for $Z\neq Z'$ we
  have $\bigl(0,\dickson(Z)\bigr)\neq\bigl(0,\dickson(Z')\bigr)$ by
  Lemma~\ref{lma:dickson}, and hence planes
  in $\mathcal{N}$ with different $Z$ have subspace distance $\geq 4$
  as well. This completes the proof of \eqref{thm:remove:1}.

  \eqref{thm:remove:2} The new planes in the standard rearrangement of
  $\{\graph_f;f\in r_i\mathcal{R}\}$ intersect $S$ in the points of
  the plane $\{0\}\times r_iW'$. For $i\neq j$, since $r_iW'\cap
  r_jW'=\emptyset$ by assumption, new planes obtained from
  $r_i\mathcal{R}$ and $r_j\mathcal{R}$ cannot intersect in $S$ and
  hence have subspace distance $\geq 4$. Together with
  \eqref{thm:remove:1} this proves \eqref{thm:remove:2}.

  \eqref{thm:remove:3} For $W=\F_8$ we have $W'=W$, and the planes
  $W$, $rW$, \ldots, $r^{(2^n-1)/7-1}W$ are pairwise
  disjoint.\footnote{In fact they form the standard example of a plane
    spread in $\PG(\F_{2^n})$.} Hence
  the first assertion in \eqref{thm:remove:3} follows from
  \eqref{thm:remove:2}. Finally, the number of planes in $\mathcal{C}$ is
  $4^n+\frac{3}{4}\bigl(2^n+\frac{1}{7}(2^n-8)^2\bigr)=4^n+\frac{3}{4\cdot
  7}(4^n-9\cdot 2^n+64)=4^n+\frac{3}{7}(4^{n-1}-9\cdot 2^{n-2}+16)$, as claimed.
\end{proof}
\begin{remark}
  \label{rmk:remove}
  The following more geometric view of
  Theorem~\ref{thm:remove}\eqref{thm:remove:1}, which yields an
  alternative proof, may be of interest.

  The ``removed''
  set of planes $\{\graph_f;f\in\mathcal{R}\}$ covers precisely half of
  the points of $\PG(V)$ outside $S$ and forms an exact $2$-cover of
  this set of points. This follows from the fact that $ux^2+u^2x=y$,
  viewed as an equation for $u\in\F_{2^n}$ with parameter $x\neq 0$,
  has $2$ solutions if $\trace(x^{-3}y)=0$ and no solution if
  $\trace(x^{-3}y)=1$. In other words, in each ($n+1$)-dimensional
  space $F\supset S$, whose affine part is of the form $F\setminus S=\{x\}\times
  S$ for some nonzero $x\in W$, exactly the points $\F_2(x,x^3v)$ with
  $\trace(v)=0$ are covered. There are $2^{n-1}$ such points, forming
  the affine part of an $n$-subspace of $F$ intersecting
  $S$ in $\{x^3v;\trace(v)=0\}$.\footnote{Here we use again the
    identification $S=\{0\}\times\F_{2^n}\cong\F_{2^n}$.}

  Since $\rank(f_1-f_2)=2$ iff $f_1-f_2\in\mathcal{T}$, this $2$-cover
  is made up from smaller pieces $\{\graph_g;g\in f+\mathcal{T}\}$
  corresponding to the cosets in $\mathcal{R}/\mathcal{T}$. The $8$
  planes in such a set mutually intersect in a point and hence
  $2$-cover a set of $\binom{8}{2}=28$ points ($4$ points in each
  $F$).

  A point $Q$ covered by $\{\graph_f;f\in\mathcal{R}\}$ is the
  intersection point of unique planes $E_1$ and $E_2$ in
  $\{\graph_f;f\in\mathcal{R}\}$. Each of $E_1,E_2$ contains $3$ lines
  through $Q$, which represent the $3$ hyperplanes above $F=Q\vee
  S$. Hence these lines are matched into $3$ pairs;
  the lines in a pair determine the same hyperplane and generate a new
  plane in the standard rearrangement. It follows that the set
  $\mathcal{N}$ of new planes of the standard rearrangement of
  $\{\graph_f;f\in\mathcal{R}\}$ forms a $3$-cover of the same set of
  points and that new planes $N_1$, $N_2$ meeting in $S$ do not meet
  outside $S$, since $N_1\cap S=N_2\cap S$ implies $N_1\vee S=N_2\vee
  S$, which in turn follows from the injectivity $Z=\langle a,b\rangle\mapsto
  P=\F_2(0,ab^2+a^2b)$. This gives again Part~\eqref{thm:remove:1} of
  Theorem~\ref{thm:remove}.
\end{remark}
\begin{example}[$v=6$]
  \label{ex:v=6}
  This is the smallest case, where Theorem~\ref{thm:remove} applies.
  Here $W=\F_8$, $\mathcal{R}=\mathcal{T}$, and
  Part~\eqref{thm:remove:1} yields a subspace code
  $\mathcal{C}=\lgabidulin\setminus\{\graph_f;f\in\mathcal{T}\}\cup\mathcal{N}$
  of size $\#\mathcal{C}=70$ consisting of the planes
  \begin{align*}
    G(a_0,a_1)&=\bigl\{(x,a_0x+a_1x^2);x\in\F_8\bigr\},\quad
                a_0,a_1\in\F_8,\;a_0\neq a_1^2;\\ 
    N(Z,c)&=\bigl\{(x,cx^2+c^2x+y\dickson(Z));y\in Z,\eta\in\F_2\bigr\},
            \quad Z\subset\F_8\text{ a line},\;c\in\F_8/Z.
  \end{align*}
  This provides the essential step in the construction of an optimal
  $(6,77,4;3)_2$ code of Type~A in \cite{smt:fq11proc}. 

  The construction is completed in the following way: The $28$ points
  $\F_2(x,y)$ covered by the new planes $N(Z,c)$ outside
  $S=\{0\}\times\F_8$ are those covered by $G(a^2,a)$, $a\in\F_8$, and
  satisfy $\trace(x^4y)=0$, as is clear from
  $x^4(ax^2+a^2x)=x^{-3}(ax^2+a^2x)=a/x+(a^2/x^2)$.\footnote{The
    hyperbolic quadric $\mathcal{H}$ in $\PG(V)\cong\PG(5,\F_2)$ with
    equation $\trace(x^4y)=0$ consists of these $28$ points and the
    $7$ points in $S$. Together with $S$ the $14$ planes $N(Z,c)$ form
    one of the two sets of generators of $\mathcal{H}$. This provides
    the link to the alternative construction of a $(6,77,4;3)_2$ code
    of Type~A in \cite{cossidente-pavese16}.}  Now it is possible to
  connect the $7$ lines in $S$ to $7$ points outside $S$ in such a way
  that the resulting planes cover precisely the $28$ points
  $\F_2(x,y)$ outside $S$ satisfying $\trace(x^4y)=1$. For this simply
  connect the point $\F_2(x,x^3)$, which has
  $\trace(x^4x^3)=\trace(1)=1$, to the line
  $\bigl\{y\in\F_8;\trace(x^4y)=0\bigr\}$. The resulting $7$ planes
  can be added to $\mathcal{C}$ to form the desired $(6,77,4;3)_2$
  code.
\end{example}
Theorem~\ref{thm:remove}\eqref{thm:remove:3} is still too weak to
produce codes meeting (let alone exceeding) the LMRD code bound.  More
generally, any choice of $\mathcal{A}$ that avoids new planes with
different $Z$ meeting in a point of $S$ is subject to the bound
$\#\mathcal{A}\leq\frac{1}{7}\#\gabidulin_W=\frac{1}{7}4^{v-3}$ and
yields a net gain relative to $\#\gabidulin_W=4^{v-3}$ of at most
$\frac{3}{7}4^{v-4}<\gauss{v-3}{2}{2}\approx\frac{2}{3}4^{v-4}$
planes.
This remains true even if we relax the condition of exact
rearrangement and augment the expurgated Gabidulin code by the maximum
number of planes meeting $S$ in a point while maintaining subspace
distance $\geq 4$.



\section{The Refined Approach}\label{sec:refined}

In this section we relax the condition of exact rearrangement but
restrict attention to rotation-invariant subsets
$\mathcal{A}\subseteq\gabidulin_W$. In \cite{lt:0328} this was
empirically found as the best approach in the smallest applicable
case $v=7$, and the subsequent algebraic analysis in \cite{mt:alb80}
has largely explained this phenomenon.

The smallest rotation-invariant subsets of $\gabidulin_W$ admitting a
standard rearrangement have size $8(2^n-1)$ and consist of the rotated copies
$r(f+\mathcal{T})$, $r\in\F_{2^n}^\times$, of a single coset
$f+\mathcal{T}$ with $f\in\mathcal{R}\setminus\mathcal{T}$. Since
there are $2^{n-3}-1$ such cosets, the total
number of choices for $\mathcal{A}$, being equal to $2^{2^{n-3}-1}-1$
($\mathcal{A}=\emptyset$ is omitted), is still growing extremely fast
with $n$.

Relaxing the condition of exact rearrangement, we now ask for
maximum-size subsets
$\mathcal{N}'\subseteq\mathcal{N}$
of the standard rearrangement of
$\{\graph_f;f\in\mathcal{A}\}$ into new planes subject to
$\sdist(\mathcal{N}')\geq 4$. Rotation-invariance of $\mathcal{A}$
reduces this, in general, ``global'' problem to a ``local'' problem at
one particular point of $S$, say $P_1=\F_2(0,1)$. For the statement of
the result note that the group $\Sigma$ is isomorphic to
$\F_{2^n}^\times$ via $\bigl((x,y)\mapsto(x,ry)\bigr)\mapsto r$.
\begin{lemma}
  \label{lma:local}
  Let $\mathcal{A}\subseteq\gabidulin_W$ be rotation-invariant and
  such that $\mathcal{A}\cap\mathcal{R}$ forms a union of nontrivial
  cosets of $\mathcal{T}$, $\mathcal{N}$ the corresponding standard
  rearrangement into new planes, and
  $\mathcal{N}_r\subseteq\mathcal{N}$ the set of new planes passing
  through $P_r=\F_2(0,r)\in S$.
  \begin{enumerate}[(i)]
  \item\label{lma:local:1}
    If $\mathcal{N}_1'\subseteq\mathcal{N}_1$ has maximum size $M_1$
    subject to the distance condition $\sdist(\mathcal{N}_1')\geq 4$
    then
    $\mathcal{N}'=\bigcup_{g\in\Sigma}g(\mathcal{N}_1')\subseteq\mathcal{N}$
    has maximum size $M_1(2^n-1)$ subject to the distance condition
    $\sdist(\mathcal{N}')\geq 4$.
  \item\label{lma:local:2} The totality of subsets
    $\mathcal{N}'\subseteq\mathcal{N}$ of maximum size $M_1(2^n-1)$ satisfying
    $\sdist(\mathcal{N}')\geq 4$ is obtained by choosing,
    independently for each $g\in\Sigma$, subsets
    $\mathcal{N}_g\subseteq\mathcal{N}_1$ of size $M_1$ with
    $\sdist(\mathcal{N}_g)\geq 4$ and taking the union
    $\mathcal{N}'=\bigcup_{g\in\Sigma}g(\mathcal{N}_g)$.
  \end{enumerate}
Therefore, if the number of sets $\mathcal{N}'_1$ in \eqref{lma:local:1} is $t$,
the total number of choices for $\mathcal{N}'$ in \eqref{lma:local:2}
is equal to $t^{2^n-1}$.
\end{lemma}
\begin{proof}
  Since $\mathcal{A}$ is rotation-invariant, we have $g(\mathcal{N}_1)
  =\mathcal{N}_r$ for the (unique) element $g\in\Sigma$ that acts as
  $(x,y)\mapsto(x,ry)$. Hence solving the optimization problem for
  $\mathcal{N}_r$ is equivalent to solving it for $\mathcal{N}_1$. The
  proof is completed by the observation that
  $\sdist(\mathcal{N}_r,\mathcal{N}_{r'})\geq 4$ if $r\neq r'$, i.e.,
  new planes meeting $S$ in different points do not conflict.\footnote{As
  part of the standard rearrangement, they cannot have a line disjoint from
  $S$ in common.}
\end{proof}
The planes in $\mathcal{N}_1$ satisfy $r\dickson(Z)=1$ in the
$\U(Z,T,f)$-representation stated earlier, hence have the form
$N=\U\bigl(Z,P_1,\dickson(u,x)/\dickson(Z)\bigr)$ or
$N=\U\bigl(Z,P_1,\dickson(u+c,x)/\dickson(Z)\bigr)$.\footnote{The
  meaning of $c$ and $u$ is the same as in \eqref{eq:cpoints}.}  The
smallest sets $\mathcal{A}$ satisfying the conditions of
Lemma~\ref{lma:local} correspond to a nontrivial coset $u+W$ and hence
to a solid of $\PG(\F_{2^n})$ containing $W$, viz.\
$T=\langle W,u\rangle$. If $u$ is fixed then, since
$\langle Z,u\rangle$ and $\langle Z,u+c\rangle$ account precisely for
the $14$ planes $\neq W$ in $T$, we may view this correspondence as a
parametrization of the new planes in $\mathcal{N}_1$ by those
planes. If $\mathcal{A}$ is chosen as the largest set satisfying the
conditions of Lemma~\ref{lma:local} (i.e., the set of all rank-$3$
binomials in $\gabidulin$), all planes of $\PG(\F_{2^n})$ intersecting
$W$ in a line are used as parameters.

\begin{definition}
  \label{dfn:cgraph}
  The \emph{collision graph} $\graph_W$ has as its vertices the planes
  in $\PG(\F_{2^n})$ meeting $W$ in a line. Two vertices
  $E=\langle Z,u\rangle$ and $E'=\langle Z',u'\rangle$ are adjacent in
  $\graph_W$ if and only if the new planes $N,N'\in\mathcal{N}_1$
  parametrized by $E$, $E'$ have a point outside $S$ (and hence a line
  through $P_1$) in common.
\end{definition}
By Lemma~\ref{lma:local}, the graph $\graph_W$ encapsulates all
information necessary for the determination of the largest sets of new
planes that can be added to the expurgated Gabidulin code without
decreasing the subspace distance to $2$, for all ``starter'' sets
$\mathcal{A}\cap\mathcal{R}\subset\gabidulin_W$ satisfying the
assumptions of the lemma: A specific set $\mathcal{A}$ corresponds,
via the parametrization $E=\langle Z,u\rangle$, to a certain vertex
subgraph of $\graph_W$, and the maximum-size cocliques of this
subgraph yields precisely the largest sets
$\mathcal{N}_1'\subseteq\mathcal{N}_1$ of new planes that can be added
in $P_1$; in particular, the number $M_1$ in Part~\eqref{lma:local:1}
of the lemma equals the independence number of the subgraph.

If $\mathcal{A}\cap\mathcal{R}$ consists of $t$ cosets of
$\mathcal{T}$, the size of the largest $(v,M,4;3)_2$ code
$\mathcal{C}$ that can be obtained by this method equals
\begin{equation}
  \label{eq:netgain}
  \#\mathcal{C}=4^n-8t(2^n-1)+M_1(2^n-1)=4^n+(M_1-8t)(2^n-1).
\end{equation}
We call the quantity $(M_1-8t)(2^n-1)$ the \emph{net gain} of
$\mathcal{C}$ (relative to an LMRD code) and the quantity $M_1-8t$ the
local net gain of $\mathcal{C}$. The present optimization problem may then be
stated as follows.
\begin{RRP}
  Among all $\gauss{n}{3}{2}$, $n=v-3$, choices for the plane $W$ in
  $\PG(\F_{2^n})$ and all $2^{2^{n-3}-1}-1$ choices for a (non-empty)
  subset $\mathcal{A}\subset\gabidulin_W$ satisfying the conditions
  of Lemma~\ref{lma:local}, determine those which result in the
  largest (local) net gain for the augmented expurgated Gabidulin code.
\end{RRP}
The LRMD code bound corresponds to a local net gain of
$\gauss{n}{2}{2}/(2^n-1)=\frac{2^{n-1}-1}{3}$.

At the first glance, this new rearrangement problem seems just as
difficult as the original one, but this is not true. Collisions
between new planes through $P_1$ can be characterized algebraically be
a certain invariant of the parametrizing plane $E$, as observed in
\cite{mt:alb80}. This forces the collision graph $\graph_W$ to have a
very special structure, which greatly simplifies the computation of
the independence numbers of the relevant subgraphs.

\begin{definition}
  \label{dfn:sickson}
  The \emph{$\sigma$-invariant} of a plane $E$ in $\PG(\F_{2^n})$
  intersecting $W$ in a line $Z$ is defined as
  \begin{equation*}
    \sickson_W(E)=\frac{\dickson(E)}{\dickson(Z)^3},
  \end{equation*}
  where $\dickson(E)$ denotes the product of all points in $E$.
\end{definition}
The plane invariant
$\dickson(E)$ is the $3$-dimensional analogue of the line invariant
$\dickson(L)$ and another instance of the Dickson invariants mentioned
in Remark~\ref{rmk:dickson}.

Now we can state and prove the key result
on the algebraic characterization of collisions between new planes.

\begin{theorem}
  \label{thm:sickson}
  Two distinct planes $E$, $E'$ in $\PG(\F_{2^n})$
  intersecting $W$ in a line
  form an edge of the collision graph $\graph_W$ if and only if
  $\sickson_W(E)=\sickson_W(E')$.
\end{theorem}
\begin{proof}
  Let $E=\langle Z,u\rangle$, $E'=\langle Z',u'\rangle$ with
  $Z=E\cap W$, $Z'=E'\cap W$, and let $N$, $N'$ the new planes
  corresponding to $E$, $E'$.  Inspecting \eqref{eq:cpoints} and using
  $r\dickson(Z)=1$, we find that the $6$ points on $N$ outside $S$ have the form
  \begin{equation*}
    \left(z,\frac{\dickson(L)}{\dickson(Z)}\right),\quad
    \left(z,\frac{\dickson(L)}{\dickson(Z)}+1\right)\quad\text{with
    $z\in Z$ and $L=\langle z,u\rangle$},
  \end{equation*}
  and similarly for $N'$. Hence $N$, $N'$ have a point outside $S$ in
  common if and only if there exists $z\in Z\cap Z'$ such that, with
  $L=\langle z,u\rangle$ and $L'=\langle z,u'\rangle$,
  \begin{equation*}
    \left(\frac{\dickson(L)}{\dickson(Z)}\right)^2
    +\frac{\dickson(L)}{\dickson(Z)}=
    \left(\frac{\dickson(L')}{\dickson(Z')}\right)^2
    +\frac{\dickson(L')}{\dickson(Z')}
  \end{equation*}
  or, equivalently,
  \begin{equation*}
    \frac{\dickson(Z)\dickson(L)\bigl(\dickson(Z)+\dickson(L)\bigr)}
    {\dickson(Z)^3}=
    \frac{\dickson(Z')\dickson(L')\bigl(\dickson(Z')+\dickson(L')\bigr)}
    {\dickson(Z')^3}.
  \end{equation*}
  But, since $\dickson(Z)+\dickson(L)$ is the line invariant of the
  third line in $E$ through $z$, and similarly for
  $\dickson(Z')+\dickson(L')$ (cf.\ Lemma~\ref{lma:dickson} and the
  remarks preceding it), the latter identity reduces to
  $z^2\dickson(E)/\dickson(Z)^3=z^2\dickson(E')/\dickson(Z')^3$ and
  hence to $\sickson(E)=\sickson(E')$.

  Conversely, $\sickson(E)=\sickson(E')$ implies that $N$, $N'$ have a
  point of the form $(z,y)$ in common for every $z\in Z\cap Z'$. Since
  $Z$ and $Z'$ intersect (as lines of a projective plane), there
  exists at least one such $z$.
\end{proof}
\begin{remark}
  \label{rmk:sickson}
  Using the notation of Theorem~\ref{thm:sickson} and its proof, the
  planes $N$, $N'$ can form a collision only if $Z\neq Z'$. This fact
  follows, e.g., from the last part of the proof: $Z=Z'$ implies that $N$,
  $N'$ have at least $3$ points outside $S$ (one for every $z\in Z$)
  in common; hence $N=N'$. Alternatively, the fact is a consequence of
  Theorem~\ref{thm:remove}\eqref{thm:remove:1}, since $Z=Z'$ implies
  that $N$, $N'$ correspond to the same
  $r=\dickson(Z)^{-1}=\dickson(Z')^{-1}$.

  Thus Theorem~\ref{thm:sickson} says in particular that the map
  $E\mapsto\dickson(E)$ is one-to-one on the set of planes $E\neq W$ of
  $\PG(\F_{2^n})$ containing a fixed line $Z\subset W$. This remains
  true when $E=W$ is included and is a special case of a more general
  fact; cf.\ Theorem~\ref{thm:allthat} in Section~\ref{sec:allthat}.
\end{remark}
Theorem~\ref{thm:sickson} has the following immediate corollary.
\begin{corollary}
  \label{cor:sickson}
  The independence number of $\graph_W$ is equal to the number of
  different values taken by $\sigma_W$ (i.e., the size of
  $\image(\sigma_W)$). Likewise, the independence number of the
  subgraph of $\graph_W$ corresponding to
  $\mathcal{A}\subset\gabidulin_W$ as in Lemma~\ref{lma:local} equals
  the number of different values taken by $\sigma_W$ on the set of all
  planes $E\neq W$ that are contained in one of the solids
  $T\supset W$ corresponding to $\mathcal{A}$.
\end{corollary}
\begin{proof}
  By Theorem~\ref{thm:sickson}, $\graph_W$ is a disjoint union of
  cliques, and the corollary follows.
\end{proof}
Although not explicitly stated in the corollary, it is clear
that all maximum-size cocliques of $\graph_W$ are obtained by selecting for
each $y\in\image(\sigma_W)$ precisely one plane $E$ with
$\sickson_W(E)=y$ and that the number of maximum-size cocliques is
equal to the product of the multiplicities of all
$y\in\image(\sigma_W)$, and similarly for the
subgraphs of $\graph_W$ corresponding to
$\mathcal{A}\subset\gabidulin_W$.

Another pleasant consequence of Theorem~\ref{thm:sickson} is the
invariance of the present optimization problem under a fairly large
collineation group of $\PG(\F_{2^n})$ acting on the set of all
$\gauss{n}{3}{2}$ planes $W$ in $\PG(\F_{2^n})$, which can serve as
the first factor of the ambient space $V$.

Let $G$ be the subgroup of $\GL(\F_{2^n})$ generated by the
multiplication maps $x\mapsto rx$, $r\in\F_{2^n}^\times$, and the
Frobenius automorphism $\frob\colon x\mapsto x^2$. The group $G$ is a
Frobenius group with kernel $H=\{x\mapsto
rx;r\in\F_{2^n}^\times\}\cong\F_{2^n}^\times$ and complement
$K=\langle\frob\rangle=\Aut(\F_{2^n}/\F_2)$ and has
order $\#G=n(2^n-1)$. It can also be seen as the normalizer of $H$ in
$\GL(\F_{2^n})$.
\begin{corollary}
  \label{cor:G}
  If $W_1$ and $W_2$ are in the same orbit of $G$ on planes in
  $\PG(\F_{2^n})$ then the collision graphs $\graph_{W_1}$ and
  $\graph_{W_2}$ are isomorphic, and the corresponding RRP's are
  equivalent in the sense that a solution of one problem immediately
  gives a corresponding solution of the other problem.
\end{corollary}
\begin{proof}
  If $W_2=rW_1$ for some $r\in\F_{2^n}^\times$ then $E\mapsto rE$
  sends the planes meeting $W_1$ in a line $Z$ to those meeting $W_2$
  in $rZ$. Clearly we have $\dickson(rZ)=r^3\dickson(Z)$,
  $\dickson(rE)=r^7\dickson(E)$, and hence
  \begin{equation*}
    \sickson_{W_2}(rE)=\frac{\dickson(rE)}{\dickson(rZ)}
    =\frac{r^7\dickson(E)}{r^3\dickson(Z)}=r^4\sickson_{W_1}(E).
  \end{equation*}
  Together with Theorem~\ref{thm:sickson} it
  follows that $E\mapsto rE$ represents a graph isomorphism
  $\graph_{W_1}\to\graph_{W_2}$. Similarly, if $W_2=\frob(W_1)=W_1^2$ then
  $\sickson_{W_2}(E^2)=\sickson_{W_1}(E)^2$ and $E\mapsto E^2$
  represents a graph isomorphism $\graph_{W_1}\to\graph_{W_2}$.
  Since $G$ is generated by the maps $x\mapsto rx$ and $\frob$, the
  first assertion follows. The second assertion is clear.
\end{proof}
We close this section with two examples. The first example recalls the
construction of a $(7,301,4;3)_2$ code in \cite{lt:0328,mt:alb80},
which was used as the intermediate step in an alternative construction of a
currently best known $(7,329,4;3)_2$ code in \cite{lt:0328,mt:alb80},
following the original discovery of such a code in \cite{braun-reichelt14}.
\begin{example}[$v=7$]
  \label{ex:v=7}
  Here $n=4$ and we represent $\F_{16}$ as $\F_2(\xi)$ with
  $\xi^4+\xi+1=0$. Since all planes in $\PG(\F_{16})$ are conjugate
  under multiplication (by Singer's Theorem), Corollary~\ref{cor:G}
  implies that we can choose $W$ freely; for convenience, we take
  $W=\bigl\{x\in\F_{16};\trace(x)=0\bigr\}
  =\{1,\xi,\xi^2,\xi^4,\xi^5,\xi^8,\xi^{10}\}$.
  Since
  $\mathcal{A}=\mathcal{R}\setminus\mathcal{T}
  =\bigl\{ux^2+u^2x;\trace(u)=1\bigr\}$
  is a single coset in this case, there is only one choice for the
  rotation-invariant set $\mathcal{A}$ in the RRP; $\mathcal{A}$
  consists of the $120$ rank-$3$ binomials $r(ux^2+u^2x)$ with
  $u\in\F_{16}\setminus W$, $r\in\F_{16}^\times$.

  The next step is to compute the $\sickson$-invariants of the $14$
  planes $E\neq W$ in $\PG(\F_{16})$. If $E=aW$,
  $a\in\F_{16}\setminus\F_2$, then
  $\dickson(E)=a^7\dickson(W)=a^7$ and $Z=E\cap W=W\cap
  aW=\{x\in\F_{16};x^8+x^4+x^2+x=a^{-8}x^8+a^{-4}x^4+a^{-2}x^2+a^{-1}x=0\}$. 
  Eliminating, we find $(1+a^4)x^4+(1+x^6)x^2+(1+a^7)x=0$,
  $\dickson(Z)=\frac{1+a^7}{1+a^4}$,
  \begin{align*}
    \sickson(aW)=\frac{a^7(1+a^4)^3}{(1+a^7)^3}
    =\frac{a^7(1+a^4)(1+a^8)}{(1+a^7)(1+a^{14})}
    =\frac{1+a^4}{1+a^{14}}=a+a^2+a^3+a^4,
  \end{align*}
  a special case of \cite[Lemma~7]{mt:alb80}. This tells us that
  $\sickson(aW)=1$ for $a\in\{\xi^3,\xi^6,\xi^9,\xi^{12}\}$ (the 5-th
  primitive roots of unity in $\F_{16}^\times$) and, not difficult to
  verify from the representation $\sickson(aW)=a(1+a)^3$, that
  $E\mapsto\sickson(E)$ maps the remaining $10$ planes $E\neq W$
  bijectively onto
  $\F_{16}^\times\setminus\{1,xi^3,\xi^6,\xi^9,\xi^{12}\}$.

  It follows that the collision graph $\graph$ consists of a complete
  graph $\mathrm{K}_4$ (formed by $\xi^3W$, $\xi^6W$, $\xi^9W$,
  $\xi^{12}W$) and $10$ isolated vertices. The independence number of
  $\graph$ is $11$, and we can add locally at $P_1$ the $10$ new
  planes parametrized by $aW$, $a$ not a 3rd power in
  $\F_{16}^\times$, to the expurgated Gabidulin code, and exactly one
  of the four new planes parametrized by $\xi^{3i}W$
  ($1\leq i\leq 4$). Finally, rotating through the $15$ points of $S$,
  we obtain $4^{15}$ different extensions of the expurgated Gabidulin
  code to a plane subspace code of size $256-120+15\times 11=301$.
  Exactly $4$ of these are rotation-invariant. Explicit
  representations of the new codewords through $P_1$ may be obtained
  by writing $aW=\langle Z,u\rangle$ with $Z=W\cap aW$ and evaluating
  $N=\U\bigl(Z,P_1,\dickson(u,x)/\dickson(Z)\bigr)$
  explicitly.\footnote{Since $N=\langle\graph_f,P_1\rangle$ for
    $f\colon Z\to\F_{16}$, $x\mapsto\dickson(u,x)/\dickson(Z)$, it
    suffices to determine the graphs of these linear maps.}
\end{example}
Our second example provides a solution of the RRP in the smallest open case.
\begin{example}[$v=8$]
\label{ex:v=8}
  Here we have $n=5$ and the relevant extension field of $\F_2$ is
  $\F_{32}=\F_2[\alpha]$ with $\alpha^5+\alpha^2+1=0$. In this case
  the group $G$, of order $5\cdot 31=155=\gauss{5}{3}{2}$, acts
  (sharply) transitive of the set of all planes in
  $\PG(\F_{32})\cong\PG(4,\F_2)$. Hence, by Corollary~\ref{cor:G} it
  suffices to consider one particular plane $W$, which we can take as
  $W=\langle 1,\alpha,\alpha^2\rangle$, determining the ambient space
  $V=W\times\F_{32}$.

  The number of nontrivial cosets in
  $\mathcal{R}/\mathcal{T}\cong\F_{32}/W$ is $3$, so that $2^3-1=7$
  different coset combinations need to be considered for $\mathcal{A}$. The
  $14$ new planes in $\mathcal{N}_1$ corresponding to a minimal choice
  of $\mathcal{A}$ ($1$ coset) are represented by the $14$ planes
  $\neq W$ in one of the solids $T_1,T_2,T_3\supset W$ in
  $\PG(\F_{32})$. Using the computer algebra system SageMath
  (\texttt{www.sagemath.org}), we have found that $\sickson_W$ takes
  $11$ distinct values on each of the three $14$-sets of planes $E$
  contained in a fixed $T_i$ and that the values of multiplicity $>1$
  in each case are the same, viz.\ $\alpha^{23},\alpha^{25},\alpha^{28}$, all of
  multiplicity $2$.\footnote{Since
    $\alpha^{23}+\alpha^{25}+\alpha^{28}=0$, these ``collision
    values'' form a line in $\PG(\F_{32})$.}  This implies
  $\#\image(\sickson_W)=27$ (the $4$ ``missing'' values are
  $\alpha^4$, $\alpha^5$, $\alpha^{21}$, $\alpha^{30}$) and that the
  independence number of any subgraph of $\graph_W$ involving $1$,
  $2$, $3$ cosets is $11$, $19$, and $27$ respectively.\footnote{Thus
    the whole graph $\graph_W$, corresponding to all $3$ cosets, has
    independence number $27$ and consists of $24$ isolated vertices
    and $3$ cliques of size $6$, which intersect the three $14$-sets
    in a $2$-set.}  Since $11-8=19-2\cdot 8=27-3\cdot 8=3$, the local
  net gain when using $t\in\{1,2,3\}$ cosets is always $3$, a constant
  independent of $t$, and the global net gain is $3\cdot 31=93$.  Thus
  the largest subspace codes obtained by rotation-invariant
  rearrangement from the Gabidulin code have size
  $1024+93=1117$. Since the largest known $(8,M,4;3)_2$ code has size
  $1326$ (at the time of writing this article, cf.\
  \cite{braun-ostergard-wassermann15}), these codes are not
  particularly good. However, they can be further extended by planes
  meeting $S=\{0\}\times\F_{32}$ in a line; cf.\ Section~\ref{sec:comp}.
\end{example}
Although the machinery developed so far is sufficient for a complete
solution of the RRP for $v=8$ (with the aid of a computer), the
computational complexity of the presently used naive method for
determining the best coset combination (``exhaustive search'') is
prohibitive for only slightly larger values of $v$.\footnote{For
  example, there is absolutely no way to settle the case $v=11$ in
  this manner in a reasonable time, since the number of coset
  combinations that must be explored is
  $2^{2^{v-6}-1}-1=2^{31}-1=2147483647$.}

\section{Dickson Invariants, Subspace Polynomials and All
  That}\label{sec:allthat}

In this section we develop the machinery that is needed to understand
the subsequent analysis of the collision graphs $\graph_W$ and their
condensed variants, called \emph{collision matrices}, which
provide all essential information about the clique sizes in
$\graph_W$. The relevant background can be found in the seminal work
of \name{Ore} on linearized polynomials \cite{ore33a} and to some
extent in \name{Berlekamp}'s book \cite[Ch.~11]{berlekamp68}. As usual
we will restrict ourselves to the ground field $\F_2$, although
everything can be generalized with only little more effort to $\F_q$.

A convenient starting point is the following $2$-analogue of the
well-known Vandermonde determinant evaluation due to
\name{E.~H.~Moore} \cite{moore96},
which holds as an identity in the polynomial ring
$\F_2[X_1,\dots,X_k]$:
\begin{equation}
  \label{eq:dickson}
    \dickson(X_1,\dots,X_k)=
    \begin{vmatrix}
      X_1&X_2&\hdots&X_k\\
      X_1^2&X_2^2&\hdots&X_k^2\\
      X_1^{2^2}&X_2^{2^2}&\hdots&X_k^{2^2}\\
      \vdots&\vdots&&\vdots\\
      X_1^{2^{k-1}}&X_2^{2^{k-1}}&\hdots&X_k^{2^{k-1}}
    \end{vmatrix}
    =\prod_{\lambda\in\F_2^k\setminus\{\vek{0}\}}
    (\lambda_1X_1+\dots+\lambda_kX_k).
\end{equation}
This identity can be proved using induction on $k$ and
\begin{equation}
  \label{eq:ickson}
  \dickson(X_1,\dots,X_k)=\dickson(X_1,\dots,X_{k-1})
  \prod_{\lambda\in\F_2^{k-1}}(X_k+\lambda_1X_1+\dots+\lambda_{k-1}X_{k-1}).
\end{equation}
The latter identity is obtained in the same way as for the ordinary
Vandermonde determinant by viewing the determinant in
\eqref{eq:dickson} as a polynomial in $X_k$ over the rational function
field $\F_2(X_1,\dots,X_{k-1})$ and determining its zeros.

Now let $U$ be a $k$-dimensional $\F_2$-subspace of $\F_{2^n}$ with
basis $\beta_1,\dots,\beta_k$. Replacing $k$ by $k+1$ in
\eqref{eq:ickson} and making appropriate substitutions, we obtain the
identity
\begin{equation}
  \label{eq:spol}
  \begin{aligned}
    \prod_{u\in U}(X+u)
    &=\prod_{\lambda\in\F_2^{k}}(X+\lambda_1\beta_1+\dots+\lambda_k\beta_k)\\
    &=\frac{\dickson(\beta_1,\dots,\beta_k,X)}{\dickson(\beta_1,\dots,\beta_k)}
    =\sum_{i=0}^ka_iX^{2^i}\in\F_{2^n}[X].
  \end{aligned}
\end{equation}
The last step uses Laplace expansion of the
determinant in \eqref{eq:dickson} along the last column and shows that
$a_i$ is equal to the quotient of a certain $k\times k$ determinant
involving the powers $\beta_j^{2^t}$,
$t\in\{0,\dots,k\}\setminus\{i\}$, and
$\dickson(\beta_1,\dots,\beta_k)$.

The polynomial $\spol_U(X)=\prod_{u\in U}(X+u)$, which is monic and
has the elements of $U$ as roots of multiplicity $1$, is known as the
\emph{subspace polynomial} associated with $U$. From the previous
computation we have that $\spol_U(X)$ is a monic linearized polynomial
($2$-polynomial) of symbolic degree $k$. Conversely, a
monic $2$-polynomial in $\F_{2^n}[X]$ is a subspace polynomial of some
$\F_2$-subspace of $\F_{2^n}$ if it
splits into linear factors and the coefficient $a_0$ of $X$ is $\neq
0$.\footnote{Thus every monic $2$-polynomial with $a_0\neq 0$ becomes
  a subspace polynomial when considered over its splitting field.}

The coefficients of $\spol_U(X)$ will be called \emph{Dickson
  invariants} of $U$ and denoted by
$\dickson_i(U)=a_{k-i}$.\footnote{The usual Dickson invariants studied
  in Modular Invariant Theory (and here specialized to the case $q=2$)
  are the polynomial counterparts
$\dickson_i^{(k)}(X_1,\dots,X_k)\in\F_2[X_1,\dots,X_k]$, which can be
obtained in the same way as the coefficients of the ``generic''
subspace polynomial
$\prod_{\lambda\in\F_2^k}(X+\lambda_1X_1+\dots+\lambda_kX_k)
\in\F_2(X_1,\dots,X_k)[X]$.
see \cite{campbell-wehlau11,derksen-kemper02,smith97,wilkerson83}.
The indexing of
$\dickson_i$, $\dickson_i^{(k)}$ follows the convention used for the
elementary symmetric polynomials; note, however, that the degree of
$\dickson_i^{(k)}$ is not $i$ but $2^k-2^{k-i}$.} For the
last Dickson invariant $\dickson_k(U)=\prod_{u\in
  U\setminus\{0\}}u=\dickson(\beta_1,\dots,\beta_k)$, the coefficient
of $X$ in $\spol_U(X)$, we usually write simply
$\dickson(U)$.\footnote{This is compatible
with the notation used in the cases $k=2,3$, which have already been
considered.}

The set $L_n$ of $2$-polynomials in $\F_{2^n}[X]$ is closed with
respect to addition and composition of polynomials (also called
``symbolic multiplication''), defined by
$f(X)\circ g(X)=f\bigl(g(X)\bigr)$, and forms a ring
$(L_n,+,\circ)$. The ring $L_n$ is non-commutative (except
for $n=1$) and isomorphic to the skew polynomial
ring $\F_{2^n}[Y;\frob]$ via
$\sum a_iX^{2^i}\mapsto\sum a_iY^i$. It is this thus quite
easy to work with.\footnote{The ring
  $\F_{2^n}[Y;\frob]$ differs from the ordinary polynomial ring
  $\F_{2^n}[Y]$ by the law $Ya=\frob(a)Y=a^2Y$, which leads to a formula
  for the coefficients of $f(X)\circ g(X)$ similar to ordinary
  polynomial multiplication except that the coefficients of $g(X)$ are
  ``twisted'' by powers of the Frobenius automorphism. In the special case
  $n=1$ (which is not of interest to us here) this ring, and hence
  $L_1$ as well, is commutative and isomorphic to $\F_2[Y]$.}
One can show that $L_n$ has no zero divisors and admits one-sided
analogues of the Euclidean Algorithm for symbolic division. The center
of $L_n$ consists of all polynomials of the form
$c_0X+c_1X^{2^n}+c_2X^{4^n}+\dotsb$ (``$2^n$-polynomials''). In
particular, $X^{2^n}+X$ is central and
$(X^{2^n}+X)=L_n(X^{2^n}+X)=(X^{2^n}+X)L_n$ is a two-sided ideal in
$L_n$.\footnote{This fact is needed below.}

The computation of subspace polynomials is facilitated by the
following symbolic factorization into linear factors:
\begin{equation}
  \label{eq:linfact}
  \spol_U(X)=(X^2+s_{k-1}(\beta_k)X)\circ
  \dots\circ(X^2+s_1(\beta_2)X)\circ(X^2+\beta_1 X),
\end{equation}
where $s_i(X)=\spol_{\langle\beta_1,\dots,\beta_i\rangle}(X)$.
This identity follows by induction on $k$ from
$\spol_U(X)=\bigl(X^2+\spol_{U'}(\beta)X\bigr)\circ\spol_{U'}(X)$,
valid for any incident pair $U'\subset U$ of subspaces of $\F_{2^n}$
with $\dim(U')=k-1$, $\dim(U)=k$, and for any $\beta\in U\setminus
U'$. The latter can be proved as follows: Since
$U=U'\uplus(\beta+U')$, we have
\begin{equation}
  \label{eq:ilinfact}
  \begin{aligned}
    \spol_U(X)
    &=\spol_{U'}(X)\spol_{U'}(X+\beta)\\
    &=\spol_{U'}(X)\bigl(\spol_{U'}(X)+\spol_{U'}(\beta)\bigr)\\
    &=\spol_{U'}(X)^2+\spol_{U'}(\beta)\spol_{U'}(X)\\
    &=(X^2+\spol_{U'}(\beta)X)\circ\spol_{U'}(X),
  \end{aligned}
\end{equation}
as desired. In the special case $U=\F_{2^n}$ we obtain a symbolic
linear factorization of $X^{2^n}+X$, which can be seen as a
noncommutative analogue of the ordinary factorization of
$X^n+1$. Since such factorizations are in 1-1 correspondence with
ordered bases of $U$, they are highly
non-unique.\footnote{For example, the number of different symbolic
  linear factorizations of $X^{2^n}+X$ in $L_n$ is equal to
  $(2^n-1)(2^n-2)\dotsm(2^n-2^{n-1})=\#\GL(n,\F_2)$.}

An important aspect of the theory is the interplay between $2$-polynomials and
$\F_2$-linear endomorphisms of extension fields $\F_{2^n}$. Every such
endomorphism is represented by a unique $2$-polynomial in $L_n$ of
symbolic degree $<n$, and composition of endomorphisms corresponds to
symbolic multiplication of $2$-polynomials. Using these facts it is
not hard to see that $\End(\F_{2^n}/\F_2)\cong
L_n/(X^{2^n}+X)$.\footnote{Using again skew polynomial rings, this
  can be extended to $\End(\F_{2^n}/\F_2)\cong
  L_n/(X^{2^n}+X)\cong\F_{2^n}[Y;\frob]/(Y^n+1)$.}
A subspace polynomial represents a particular $\F_2$-linear
map $\F_{2^n}\to\F_{2^n}$, $x\mapsto\spol_U(x)$ with kernel $U$.
Its image will be called the \emph{opposite subspace} of $U$ and denoted
by $U^\opp$.


The opposite subspace $U^\opp$ is characterized by the identity
$\spol_{U^\opp}(X)\circ\spol_U(X)=X^{2^n}+X$, which follows from the
observation that both sides represent the zero map in
$\End(\F_{2^n}/\F_2)$ (and considering symbolic
degrees).\footnote{Berlekamp \cite[Th.~11.35]{berlekamp68} denotes
  $\spol_{U^\opp}(X)$ by $\spol_U(X)^\ast$, which conflicts with our
  (\name{Ore's}) notation for the adjoint subspace or polynomial; see
  below.} Since
$X^{2^n}+X$ is in the center of the ring $L_n$, we also have
$\spol_U(X)\circ\spol_{U^\opp}(X)=X^{2^n}+X$,\footnote{This is an
  instance of the following general fact about noncommutative integral
  domains $R$: $ab\in\zentrum(R)$ implies $ab=ba$, which is immediate
  from $a(ab)=(ab)a=a(ba)$.} and hence $U^{\opp\opp}=U$.  Further, we see
that the polynomials in $\F_{2^n}[X]$ that represent subspace
polynomials of $\F_2$-subspaces of $\F_{2^n}$ are precisely the monic
symbolic divisors of $X^{2^n}+X$ in $L_n$ (on either
side),\footnote{If $X^{2^n}+X=a(X)\circ b(X)$ then
  $x\mapsto a\bigl(b(x)\bigr)$ is the zero map in
  $\End(\F_{2^n}/F_2)$, and hence the dimensions of the kernels of
  $x\mapsto a(x)$ and $x\mapsto b(x)$ must be equal to their symbolic
  degrees, i.e., $a(X)$ and $b(X)$ must be subspace polynomials
  (provided they are monic).}  and symbolic factors of subspace
polynomials (on either side) are again subspace polynomials. But
symbolic products of subspace polynomials are not necessarily subspace
polynomials, in view of the extra conditions imposed on the factors
in~\eqref{eq:linfact}.\footnote{More precisely, $\spol_U(X)\circ\spol_V(X)$
  is a subspace polynomial iff $U\subseteq V^\opp$.}

Subspace polynomials are thus analogous to generator polynomials of
cyclic codes, and the subspace polynomial of $U^\opp$ is the
$2$-analogue, or non-commutative analogue, of the check polynomial of
a cyclic code. But the analogy goes still further, as we will see in a
moment.

Apart from the opposite subspace $U^\opp$, the following subspaces
associated with $U$ will be needed later: The \emph{orthogonal subspace}
of $U$ is $U^\perp=\{y\in\F_{2^n};\trace(xy)=0\text{ for all }x\in
U\}$, and the \emph{adjoint subspace} of $U$ is the subspace $U^\adj$
generated by
\begin{equation}
  \label{eq:adj}
  \frac{\dickson(\beta_2,\dots,\beta_k)}{\dickson(\beta_1,\dots,\beta_k)},\;
  \frac{\dickson(\beta_1,\beta_3,\dots,\beta_k)}
  {\dickson(\beta_1,\dots,\beta_k)},\;
  \dots,\;\frac{\dickson(\beta_1,\dots,\beta_{k-1})}
  {\dickson(\beta_1,\dots,\beta_k)}.
\end{equation}
The definition of $U^\adj$ does not depend on the chosen basis
$\beta_1,\dots,\beta_k$ of $U$, as is easily shown by multilinear
expansion. Moreover, one can show that the elements in \eqref{eq:adj}
are linearly independent (cf.\ the proof of Theorem~\ref{thm:allthat}
below) and hence $\dim(U^\adj)=\dim(U)$.

\name{Ore}~\cite{ore33a} has defined $U^\adj$ in a different way as
the space whose square $(U^\adj)^2=\frob(U^\adj)$ is the set of roots
of the $2$-polynomial $\spol_U(X)^\adj=\sum_{i=0}^k(a_iX)^{2^{k-i}}$,
a $2$-analogue of the reciprocal polynomial
$X^{\deg p}p(X^{-1})$ associated with an ordinary polynomial $p(X)$,
and derived the basis of $U$ listed in
\eqref{eq:adj}. \name{Ore}~\cite{ore33a} has also shown that the
nonzero elements in $U^\adj$ are precisely the elements
$A^{-1}\in\F_{2^n}$ for which $X^2+AX$ is a symbolic left factor of
$\spol_U(X)$.\footnote{Compare this with the obvious fact that the
  $X^2+aX$ is a symbolic right factor of $U$ iff $a\in U$. More
  generally, we have $V\subseteq U$ ($V^\adj\subseteq U^\adj)$
  iff $\spol_V(X)$ is a symbolic
  right (respectively, left) factor of $\spol_U(X)$.} This
follows from
\begin{equation}
  \label{eq:leftfact}
  \spol_U(X)=(X^2+AX)\circ\spol_V(X)=\spol_V(X)^2+A\spol_V(X),
\end{equation}
which implies $A^{-1}=\frac{\dickson(V)}{\dickson(U)}$, together with
the observation that the nonzero elements of $U^\adj$ have the form
$\frac{\dickson(V)}{\dickson(U)}$ for some subspace $V\subset U$ of
codimension $1$; cf.\ Theorem~\ref{thm:allthat} below.

The following lemma relates the
three subspaces associated with $U$ and will be needed in Section
~\ref{sec:cont}.
\begin{lemma}
  \label{lma:adjfroboppperp}
  For any subspace of $\F_{2^n}$ we have $(U^\adj)^2=(U^\opp)^\perp$.
\end{lemma}
\begin{proof}
  Dividing \eqref{eq:leftfact} by $A^2$ gives
  $\spol_U(X)/A^2=\bigl(\spol_V(X)/A\bigr)^2+\spol_V(X)/A$. Hence, by
  Hilbert's Satz~90, $\trace\bigl(\spol_U(x)/A^2\bigr)=0$ for all
  $x\in\F_{2^n}$. Since $U^\opp$ is the image of $x\mapsto\spol_U(x)$,
  this says $(U^\adj)^2\subseteq(U^\opp)^\perp$. Since
  $\dim(U^\adj)=\dim(U)=\dim\bigl((U^\opp)^\perp\bigr)$, the result
  follows.
\end{proof}
Since $(U^\adj)^2$ is the subspace associated with the ``reciprocal''
subspace polynomial $a_0^{-2^k}\spol_U(X)^\adj$,
Lemma~\ref{lma:adjfroboppperp} provides a nice $2$-analogue
of the well-known fact that the dual code $C^\perp$ of a cyclic code
$C$ is generated by the reciprocal of the check polynomial of $C$.

The following properties of the map $U\mapsto\dickson(U)$ will play a
crucial role in Section~\ref{sec:cont}, and we state them as a theorem.

\begin{theorem}
  \label{thm:allthat}
  Let $U$ be a $k$-subspace of $\F_{2^n}$,
  \begin{enumerate}[(i)]
  \item\label{thm:allthat:1} $V\mapsto\dickson(V)$ maps the
    ($k+1$)-subspaces of $\F_{2^n}$ containing $U$
    bijectively onto the $1$-subspaces of the space
    $\dickson(U)U^\opp$. The induced map from $\PG(\F_{2^n})/U$ to
    $\PG\bigl(\dickson(U)U^\opp)$ is a collineation.
  \item\label{thm:allthat:2} $V\mapsto\dickson(V)$ maps the
    ($k-1$)-subspaces of $\F_{2^n}$ contained in $U$
    bijectively onto the $1$-subspaces of
    $\dickson(U)U^\adj$. The induced map from $\PG(U)$ to
    $\PG\bigl(\dickson(U)U^\adj\bigr)$ is a correlation.
  \end{enumerate}
\end{theorem}
Note that our earlier Lemma~\ref{lma:dickson} is precisely the case
$k=3$ of Part~\eqref{thm:allthat:2}.
\begin{proof}
  \eqref{thm:allthat:1} From either \eqref{eq:spol} or
  \eqref{eq:ilinfact} we have $\dickson(V)=\spol_U(\beta)\dickson(U)$
  for any ($k+1$)-subspace $V\supset U$ and any
  $\beta\in V\setminus U$. As $\beta$ varies over $V\setminus U$,
  $\spol_U(\beta)$ varies over $U^\opp\setminus\{0\}$. This proves the
  first assertion. The second assertion follows from linearity of
  $\beta\mapsto\spol_U(\beta)$.

  \eqref{thm:allthat:2} Let $V$ be a ($k-1$)-subspace of
  $U$ with basis $v_1,\dots,v_{k-1}$. Using multilinear expansion,
  $\dickson(V)/\dickson(U)=\dickson(v_1,\dots,v_{k-1})/\dickson(U)$
  can be expressed as an $\F_2$-linear combination of the elements in
  \eqref{eq:adj}, showing that
  $\dickson(V)\in\dickson(U)U^\adj$. The coefficients $\mu_i$ of this linear
  combination are easily seen to be the minors of order $k-1$
  of the matrix $(\lambda_{ij})\in\F_2^{k\times(k-1)}$
  determined by $v_j=\sum_{i=1}^k\lambda_{ij}\beta_i$. Now it is
  well-known that for any
  $(\mu_1,\dots,\mu_k)\in\F_2^k\setminus\{\vek{0}\}$ there exists a
  corresponding matrix $(\lambda_{ij})$ of rank $k-1$ having $\mu_i$
  as their order $k-1$ minors.\footnote{Essentially this amounts to
    the fact that any nonzero vector in $\F_2^k$ can be completed to
    an invertible $k\times k$ matrix.} Hence any element of $U^\adj$
  has the form $\dickson(V)/\dickson(U)$ and $V\mapsto\dickson(V)$
  maps onto $\dickson(U)U^\adj$. The
  elements in \eqref{eq:adj} are linearly independent. (A dependency
  relation would yield a subspace $V$ with $\dickson(V)=0$ by what we
  have just shown; this is impossible.) Hence $\dim(U^\adj)=\dim(U)$
  and $V\mapsto\dickson(V)$ maps the ($k-1$)-subspaces of
  $U$ bijectively onto $\dickson(U)U^\adj$. Finally, the correlation property
  follows from Part~\eqref{thm:allthat:1}: If $V_0\subset U$ has
  dimension $k-2$ then the ($k-1$)-subspaces between $V_0$ and
  $U$ are mapped to a $2$-subspace of
  $\dickson(V_0)V_0^\opp$, which must be contained in $\dickson(U)U^\adj$.
\end{proof}
Theorem~\ref{thm:allthat} has the following rather
curious corollary.
\begin{corollary}
  \label{cor:allthat}
  The $k$-subspaces $U\subseteq\F_{2^v}$ with fixed last Dickson
  invariant $\dickson(U)=a$, $a\in\F_{2^v}^\times$, form a subspace
  code $\mathcal{C}(a)$ with minimum distance at least $4$.
\end{corollary}
\begin{proof}
  If $U,V\in\mathcal{C}$ satisfy
$\sdist(U,V)=2$ then $\dim(U\cap V)=k-1$, $\dim(U+V)=k+1$. Now either
Part~\eqref{thm:allthat:1} of the theorem applied to $W=U\cap V$, or
Part~\eqref{thm:allthat:2} applied to $W=U+V$ yields a contradiction.
\end{proof}
By Corollary~\ref{cor:allthat}, the set of $k$-subspaces of $\F_2^v$
is partitioned into $2^v-1$ (possibly empty) subspace codes of minimum
distance $\geq 4$. Viewed as single codes, these are not very
interesting, since they are too small. In the case $k=3$ the largest
of these codes has guaranteed size
\begin{equation*}
  \#\mathcal{C}(a)
  \geq\frac{1}{2^v-1}\gauss{v}{3}{2}=\frac{(2^{v-1}-1)(2^{v-2}-1)}{21}
  \approx\frac{8}{21}2^{2(v-3)},
\end{equation*}
which is considerably
smaller than the size of the corresponding Gabidulin codes.
A computational study of combinations of several
codes $\mathcal{C}(a)$ has not produced anything of
value. Nevertheless, the corollary and its ramifications deserve
further research. Links between subspace polynomials and subspace
codes are also studied in \cite{ben-sasson-etal14}, with emphasis on
the case of cyclic subspace codes. The codes of
Corollary~\ref{cor:allthat} tend to be transversal to the
corresponding Singer orbits and are decidedly noncyclic. On the other
hand, the gap theorem in \cite[Cor.~2]{ben-sasson-etal14} and our
corollary are similar---in both cases certain Dickson invariants are
prescribed.

We close this section with some simple examples of subspace
polynomial computations. About point polynomials $\spol_P(X)=X^2+aX$,
$P=\F_2a$, there is not much to say. Line polynomials $\spol_L(X)$,
$L=\langle a,b\rangle$ are easily computed using \eqref{eq:linfact} and
have the form $\spol_L(X)=\bigl(X^2+(b^2+ab)X\bigr)\circ(X^2+aX)
=(X^2+aX)^2+(b^2+ab)(X^2+aX)=X^4+(a^2+ab+b^2)X^2+(ab^2+a^2b)X$, recovering
$\dickson_2(L)=\dickson(L)=ab^2+a^2b=ab(a+b)$ and showing
that $\dickson_1(L)=a^2+ab+b^2$. We have $\dickson_1(L)=0$ iff
$(a/b)^3=1$, a property characterizing the lines of the standard line
spread in $\PG(\F_{2^n})$, $n$ even (since these lines have the form
$a\F_4^\times$). The most prominent example is the well-known
$\spol_{\F_4}(X)=X^4+X$. Sometimes it is useful to express
$\dickson_1(L)$, $\dickson_2(L)$ in terms of each other, for which we
note that $\dickson_1(L)\dickson_2(L)=ab(a^3+b^3)=ab^4+a^4b$. Often
subspace polynomial computations can be simplified by taking the
action of the multiplicative group, $U\mapsto rU$ for
$r\in\F_{2^n}^\times$ and of the Frobenius automorphism ($U\mapsto
U^2=\{u^2,u\in U\}$) into account. This is illustrated in the final
example of this section.
\begin{example}
  \label{ex:spol}
  We compute the subspace polynomials for all lines and planes in
  $\F_{16}$. One example of a plane polynomial is the well-known polynomial
  $\spol_{W}(X)=X^8+X^4+X^2+X$ of the trace-zero plane
  $W=\{x\in\F_{16};\trace(x)=0\}$. Any further plane has the form
  $rW$ for a unique $r\in\F_{16}^\times$, and the corresponding polynomial is
  \begin{equation*}
    \spol_{rW}(X)=r^8\spol_{W}(r^{-1}X)=X^8+r^4X^4+r^6X^2+r^7X.
  \end{equation*}
  For line polynomials $\spol_L(X)=X^4+a_1X^2+a_0X$ we use the formulas
  \begin{align*}
    \spol_{rL}(X)&=r^4\spol_L(r^{-1}X)=X^4+r^2a_1X^2+r^3a_0X,\\
    \spol_{L^2}(X)&=\prod_{u\in L}(X+u^2)=X^4+a_1^2X^2+a_0^2X,
  \end{align*}
  i.e.\ $L\mapsto rL$ and $L\mapsto L^2$ correspond to
  $(a_0,a_1)\mapsto(r^3a_0,r^2a_1)$ and
  $(a_0,a_1)\mapsto(a_0^2,a_1^2)$. Using $\F_{16}=\F_2(\xi)$ with
  $\xi^4+\xi+1=0$ and $\omega=\xi^5$, we have
  $\F_{16}^\times=\{\xi^i;0\leq i\leq 14\}$ and
  $\F_2(\omega)=\F_4$ inside $\F_{16}$. The lines of the
  standard spread of $\PG(\F_{16})$ have polynomials
  $\spol_{\xi^i\F_4}(X)=X^4+\xi^{3i}X$ ($0\leq i\leq 4$).
  The remaining $30$ lines in $\PG(\F_{16})$ can be obtained from a
  single line $L$ 
  as $rL$ or $rL^2$. Since with $rL$ the coefficient $a_1$ in
  $\spol_{rL}(X)$ ``rotates'' through all of $\F_{16}^\times$, there
  exist exactly two line polynomials of the form $X^4+X^2+a_0$, which
  must be $X^4+X^2+\omega X$ and $X^4+X^2+\omega^2X$ (since their
  conjugates under $x\mapsto x^2$ are line polynomials as well).
  Hence the remaining $30$ line polynomials are
  $X^4+r^2X^2+r^3\omega^i X$ with $i\in\{1,2\}$ and
  $r\in\F_{16}^\times$, and it only remains to identify the line $L$
  behind $X^4+X^2+\omega X=X(X^3+X+\omega)$ by factoring this
  polynomial. It turns out that
  $L=\{\xi^{10},\xi^{11},\xi^{14}\}=\xi^{10}\langle1,\xi\rangle$.\footnote{Perhaps
    the easiest way to find $L$ is to compute the line polynomial of
    $\langle 1,\xi\rangle=\{1,\xi,\xi^4\}$, which is
    $X^4+\xi^{10}X^2+\xi^5X$, and rotate by
    $r=(\xi^{10})^{-1/2}=\xi^{10}$.}
\end{example}


\section{Continuation of the Analysis}\label{sec:cont}

Our first goal in this section is to determine the set of multiple
values of $\sickson_W$. The planes $E$ in $\PG(\F_{2^n})$ meeting $W$
in a line fall into $7$ classes according to their intersection
$Z=E\cap W$. By Theorem~\ref{thm:allthat}\eqref{thm:allthat:1}, the
restriction of $\sickson_W\colon E\mapsto\dickson(E)/\dickson(Z)^3$ to
such a class is one-to-one with image $\dickson(Z)^{-2}Z^\opp$,
provided we include $\dickson(W)/\dickson(Z)^3$ in the image. We will
refer to the $7$ values $\dickson(W)/\dickson(Z)^3$, $Z\subset W$ a
line, as the ``missing values'' (or ``missing points'') of
$\sickson_W$.\footnote{The missing values are counted with their
  multiplicities. They are not necessarily all different, and they can
  still be in the image of $\sickson_W$; cf.\ Theorem~\ref{thm:cspace}
  for the details.}

For $x\in\F_{2^n}\setminus Z$ we have $\dickson\bigl(\langle
Z,x\rangle)=\spol_Z(x)\dickson(Z)$ and hence
\begin{equation*}
  \sickson_W\bigl(\langle Z,x\rangle\bigr)=\frac{\spol_Z(x)}{\dickson(Z)^2}.
\end{equation*}
For $x,y\in\F_{2^n}$ we set
\begin{equation}
  \label{eq:innpro}
  \innpro{x}{y}_Z=\trace\left(\frac{\spol_Z(x)y^2}{\dickson(Z)^2}\right).
\end{equation}

\begin{lemma}
  \label{lma:innpro}
  For any line $Z$ in $\PG(\F_{2^n})$ the map
  $\F_{2^n}\times\F_{2^n}\to\F_2$, $(x,y)\mapsto\innpro{x}{y}_Z$ is a
  symmetric $\F_2$-bilinear form with radical $Z$ and associated quadratic form
  $x\mapsto\innpro{x}{x}_Z=\trace(a_1x^4/a_0^2)$, where
  $a_0=\dickson(Z)$ and $a_1=\dickson_1(Z)$.
\end{lemma}
\begin{proof}
  Bilinearity is clear. Further, we have
  \begin{equation}
    \label{eq:innprosum}
    \frac{\spol_Z(x)y^2}{\dickson(Z)^2}=\frac{(x^4+a_1x^2+a_0x)y^2}{a_0^2}
    =\frac{x^4y^2}{a_0^2}+\frac{a_1x^2y^2}{a_0^2}+\frac{xy^2}{a_0}.
  \end{equation}
  Now note that the trace of the sum in \eqref{eq:innprosum} does not
  change if we conjugate the three summands
    individually (!) with powers of the Frobenius
    automorphism.\footnote{We have found this ``trace trick''
    useful on several occasions.}
  Expressing everything in terms of $x^2$, we get
  \begin{align*}
    \innpro{x}{y}_Z
    &=\trace\left(\frac{x^2y}{a_0}+\frac{a_1x^2y^2}{a_0^2}
      +\frac{x^2y^4}{a_0^2}\right)\\
    &=\trace\left(\frac{x^2(a_0y+a_1y^2+y^4)}{a_0^2}\right)=\innpro{y}{x}_Z.
  \end{align*}
  Since the trace is nondegenerate and $\spol_Z(x)=0$ iff $x\in Z$,
  the radical of $\innpro{x}{y}_Z$ must be $Z$.
  Finally, substituting $y=x$ into \eqref{eq:innprosum} turns the third
  summand into a conjugate of the first, giving
  $\innpro{x}{x}_Z=\trace(a_1x^4/a_0^2)$.
\end{proof}
\begin{theorem}
  \label{thm:cspace}
  \begin{enumerate}[(i)]
  \item\label{thm:cspace:geom} The spaces $\dickson(Z)^{-2}Z^\opp$,
    $Z\subset W$ a line, mutually intersect in $(W^2)^\perp$, and
    hence account for all ($n-2$)-subspaces of $\F_{2^n}$
    containing $(W^2)^\perp$.
  \item\label{thm:cspace:w2perp}
    The set of multiple values of $\sickson_W$ is precisely the subspace
    $(W^2)^\perp$.\footnote{Orthogonality is taken with respect to the
      trace bilinear form.}
  \end{enumerate}
\end{theorem}
Before proving the theorem we note the following consequence
of Part~\eqref{thm:cspace:geom}:
$y\in\F_{2^n}^\times$ is not in the image of $\sickson_W$ iff $y$ is a
missing point of $\sickson_W$ and $y\notin(W^2)^\perp$. This shows
$\#\image(\sickson_W)=2^n-1-(7-\mu)$, where $\mu$ denotes the number
of missing points contained in $(W^2)^\perp$. Since $\mu\leq 7$, $\sickson_W$ is
``almost'' surjective for large $n$.
\begin{proof}
  \eqref{thm:cspace:geom}
  For $y\in\F_{2^n}^\times$ consider the seven equations
  \begin{equation}
    \label{eq:y}
    \frac{\spol_Z(x)}{\dickson(Z)^2}=y,\quad\text{$Z$ a line in $W$}.
  \end{equation}

  For any particular $Z$, \eqref{eq:y} is solvable iff
  $y\in\dickson(Z)^{-2}Z^\opp$. On the other hand, we will show
  that \eqref{eq:y} is solvable iff $y\in(Z^2)^\perp$, thereby
  establishing $\dickson(Z)^{-2}Z^\opp=(Z^2)^\perp$. Since
  $\bigcap\{(Z^2)^\perp;Z\subset W\text{ a line}\}=(W^2)^\perp$ and
  $\dim(W^2)^\perp=n-3$,
  \eqref{thm:cspace:geom} then follows.

  Using Lemma~\ref{lma:innpro}, \eqref{eq:y}
  implies $\trace(yz^2)=\innpro{x}{z}_Z=\innpro{z}{x}_Z=0$ for all
  $z\in Z$ and thus $y\in(Z^2)^\perp$. Conversely, suppose
  $y\in(Z^2)^\perp$ and $c\in\F_{2^n}$ is such that
  $\trace\left(\frac{\spol_Z(x)c^2}{\dickson(Z)^2}\right)=\innpro{x}{c}_Z=0$
  for all $x\in\F_{2^n}$. Then $c\in Z$, the radical of $\innpro{\ }{\
  }_Z$, and hence $\trace(yc^2)=0$. The non-degeneracy of the trace bilinear
  form now implies that \eqref{eq:y} has a solution, completing the
  proof of \eqref{thm:cspace:geom}.

  Let us remark that the identity
  $\dickson(Z)^{-2}Z^\opp=(Z^2)^\perp$, which also shows that
  $Z\mapsto\dickson(Z)^{-2}Z^\opp$ defines a correlation from $\PG(W)$
  to $\PG\bigl(\F_{2^n}/(W^2)^\perp\bigr)$,
  can alternatively be derived from $(Z^\adj)^2=(Z^\opp)^\perp$ (true in
  general, cf.\ Lemma~\ref{lma:adjfroboppperp}) and
  $Z^\adj=\dickson(Z)^{-1}Z$ for any line $Z$ (a speciality in
  dimension $2$).

  \eqref{thm:cspace:w2perp} By \eqref{thm:cspace:geom},
  if $y\in(W^2)^\perp$ then \eqref{eq:y}
  has $7$ solutions (one for each $Z$). Since $Z\mapsto\dickson(Z)$ is
  one-to-one, at most $3$ solutions can be in $W$, i.e., correspond to
  $\dickson(W)/\dickson(Z)^3=y$. Hence $y$ is a value of multiplicity
  $\geq 4$ in this case. On the other hand, if $y\notin(W^2)^\perp$
  then \eqref{eq:y} has a solution for one particular $Z$, which may
  be in $W$, and hence $y$ has multiplicity $0$ or $1$.
\end{proof}

\begin{definition}
  \label{dfn:cspace}
  The space $C=(W^2)^\perp\subset\F_{2^n}$ is called the \emph{collision
    space} (of the RRP) relative to $W$.
  The matrix $\mat{C}_W=(c_{ij})$ whose rows are labelled with the
  solids of $\PG(\F_{2^n})$ containing $W$, whose columns are
  labelled with the elements of $(W^2)^\perp$ (relative to some
  orderings of these sets), and whose $(i,j)$ entry $c_{ij}$ is
  defined as the number of planes $E\neq W$ in the solid $T_i$
  satisfying $\sickson_W(E)=y_j$, is called a \emph{collision matrix}
  relative to $W$.\footnote{In what follows, we will often say ``the
    collision matrix $\mat{C}_W$. This is slightly inaccurate but
    forgivable in our case, since collision matrices for the same $W$
    differ only by row and column permutations and all properties
    discussed will be invariant under these.}
\end{definition}

Since both $\F_{2^n}/W$ and $(W^2)^\perp$ have dimension $n-3$,
$\mat{C}_W$ is a square matrix of order $2^{n-3}-1=2^{v-6}-1$. From
the preceding development it should be clear that $\mat{C}_W$ contains
the necessary information to determine the maximum net gain for all
subsets $\mathcal{A}\subset\gabidulin_W$ satisfying the conditions of
Lemma~\ref{lma:local} and thus essentially solve the RRP. This is made
explicit in our next theorem. The actual solution also requires
selecting planes $E$ with $\sickson_W(E)=y$ in the solids $T$
corresponding to $\mathcal{A}$ and finding the corresponding new
planes in $\mathcal{N}_1$, but this is a straightforward computational
task and will not be discussed further.

\begin{theorem}
  \label{thm:copt}
  Let $W$ be a plane in $\PG(\F_{2^n})$, $m=2^{n-3}-1$ the order of the
  corresponding collision matrix $\mat{C}_W$ and $r_i$ the $i$-th row
  sum of $\mat{C}_W$ ($1\leq i\leq m$).
  The maximum local net gain achievable in the RRP specialized to $W$ is
  the solution of the following combinatorial optimization problem:
  \begin{equation}
    \label{eq:copt}
    \begin{array}{rl}
      \text{Maximize}&\sum_{i=1}^m(6-r_i)x_i+\wham(\vek{x}\mat{C}_W)\\
      \text{subject to}&\vek{x}\in\{0,1\}^m
    \end{array}
  \end{equation}
\end{theorem}
As a consequence of this theorem and Corollary~\ref{cor:G}, the maximum
local net gain achievable in the general RRP (which depends only on
$n=v-3$) is equal to the largest optimal solution of the family of
optimization problems \eqref{eq:copt}, with $W$ running through a system
of representatives for the $G$-orbits on planes in $\PG(\F_{2^n})$.
\begin{proof}[Proof of Theorem~\ref{thm:copt}]
  An admissible selection of $\mathcal{A}\subset\gabidulin_W$
  corresponds to a set $T_i$, $i\in I\subseteq\{1,\dots,m\}$, of solids
  containing $W$ (the indexing is the same as in the definition of
  $\mat{C}_W$) and hence to a unique vector $\vek{x}\in\{0,1\}^m$ (the
  characteristic vector of $I$). Let $\mat{C}_W(I)$ be the
  submatrix of $\mat{C}_W$ with rows indexed by $i\in I$, and $t=\#I$.

  The number of planes
  $E$ involved in the rearrangement (equal to $\#\mathcal{N}_1$) is
  $14t$, of which $\sum_{i\in I,1\leq j\leq m}c_{ij}$ have
  $\sickson_W(E)\in(W^2)^\perp$. The
  corresponding local net gain is obtained by selecting from the $14t$
  planes all those which have $\sickson_W(E)\notin(W^2)^\perp$,
  and for each value $y\in(W^2)^\perp$ that corresponds to a nonzero
  column of $\mat{C}_W(I)$ one further plane with
  $\sickson(E)=y$. Hence the local net gain equals
  \begin{equation*}
    14t-\sum_{i\in I,1\leq j\leq m}c_{ij}+\wham(\vek{x}\mat{C}_W)-8t
    =\sum_{i\in I}(6-r_i)+\wham(\vek{x}\mat{C}_W),
  \end{equation*}
  as claimed.
\end{proof}
Before deriving further properties of collision matrices, we
illustrate the newly developed concepts with examples, including a
solution of the RRP in the case $v=9$.

\begin{example}[$v=7$, continuation of Example~\ref{ex:v=7}]
  \label{ex:v=7cont}
  This case is now rather trivial. Recalling that $W\subset\F_{16}$ has been
  chosen as the trace-zero subspace, we only observe the following: The
  collision space in this case is $(W^2)^\perp=W^\perp=\F_2$; i.e.,
  $\sickson(E)=1$ is the only multiple value of the
  $\sickson$-invariant and at the same time a missing point of
  multiplicity $3$. The corresponding collision matrix is the
  $1\times 1$ matrix $\mat{C}_W=(4)$, and \eqref{eq:copt} 
  has objective value $3$ (attained at $x=1$). The $4$ missing points
  outside the collision space are the primitive 5th roots of unity in
  $\F_{16}$.
\end{example}

\begin{example}[$v=8$, continuation of Example~\ref{ex:v=8}]
  \label{ex:v=8cont}
  In the case $v=8$, $W=\{1,\alpha,\alpha^2\}$, the collision space
  is the line $\{\alpha^{23},\alpha^{25},\alpha^{28}\}$. The points of
  the trace zero subspace of $\F_{32}$ are $\alpha^i$,
  $i\in\{0, 1, 2, 4, 7, 8, 14, 15, 16, 19, 23, 25, 27, 28, 29, 30\}$,
  and $\trace(\alpha^{2s+t})=0$ for $s=0,1,2$ and $t=23,25,28$. The
  collision matrix in this case is
  \begin{equation*}
    \mat{C}_W=
    \begin{pmatrix}
      2&2&2\\
      2&2&2\\
      2&2&2
    \end{pmatrix},
  \end{equation*}
  and \eqref{eq:copt} reduces to a trivial
  optimization problem with maximum objective value $3$ attained at all
  nonzero vectors $\vek{x}\in\{0,1\}^3$.

  Further, using $1+\alpha^2=\alpha^5$, $1+\alpha=\alpha^{18}$,
  $\alpha+\alpha^2=\alpha^{19}$, $1+\alpha+\alpha^2=\alpha^{11}$, we
  have
  $W=\{1,\alpha,\alpha^2,\alpha^5,\alpha^{11},\alpha^{18},\alpha^{19}\}$,
  $\dickson(W)=\alpha^{25}$; the lines in $W$ are
  $L_1=\{1,\alpha,\alpha^{18}\}$,
  $L_2=\{1,\alpha^2,\alpha^5\}$,
  $L_3=\{1,\alpha^{11},\alpha^{19}\}$,
  $L_4=\{\alpha,\alpha^2,\alpha^{19}\}$,
  $L_5=\{\alpha,\alpha^5,\alpha^{11}\}$,
  $L_6=\{\alpha^{18},\alpha^2,\alpha^{11}\}$,
  $L_7=\{\alpha^{18},\alpha^5,\alpha^{19}\}$,
  and $W'=\{\dickson(Z);Z\subset
  W\}=\{1,\alpha^7,\alpha^{11},\alpha^{17},\alpha^{19},\alpha^{22},\alpha^{30}\}$,
  Hence the missing points are $\{\dickson(W)/\dickson(Z)^3;Z\subset
  W\}
  =\{\alpha^4,\alpha^5,\alpha^{21},\alpha^{23},\alpha^{25},
  \alpha^{28},\alpha^{30}\}$. Note that all points on the collision line are
  missing points of multiplicity $1$. We will see later (cf.\
  Theorem~\ref{thm:C}\eqref{thm:C:types}) that this fact
  is responsible for the three $2$'s in each column of $\mat{C}_W$.
\end{example}

\begin{example}[$v=9$]
  \label{ex:v=9}
  Here $n=6$ and the corresponding extension field is
  $\F_{64}=\F_2[\alpha]$ with $\alpha^6+\alpha^4+\alpha^3+\alpha+1=0$.
  The $\gauss{6}{3}{2}=1395$ planes in $\PG(\F_{64})$ fall into $7$
  $G$-orbits with representatives $W_1=\langle
  1,\alpha,\alpha^2\rangle$, $W_2=\langle
  1,\alpha,\alpha^3\rangle$, $W_3=\langle
  1,\alpha,\alpha^4\rangle$, $W_4=\langle
  1,\alpha,\alpha^5\rangle$, $W_5=\langle
  1,\alpha,\alpha^{22}\rangle$, $W_6=\langle
  1,\alpha^3,\alpha^{18}\rangle$, $W_7=\langle
  1,\alpha^9,\alpha^{18}\rangle$ and orbit lengths $189$, $378$,
  $126$, $378$, $189$, $126$, $9$, respectively. The corresponding
  collision matrices are, in order,
\begin{align*}
  &\setlength{\arraycolsep}{1.5pt}
  \renewcommand{\arraystretch}{0.75}
  \left(\begin{array}{rrrrrrr}
1 & 1 & 2 & 1 & 2 & 0 & 1 \\
1 & 1 & 0 & 1 & 0 & 0 & 1 \\
1 & 1 & 0 & 1 & 0 & 0 & 1 \\
1 & 1 & 0 & 1 & 0 & 0 & 1 \\
1 & 1 & 0 & 1 & 0 & 0 & 1 \\
1 & 1 & 2 & 1 & 2 & 0 & 1 \\
1 & 1 & 2 & 1 & 2 & 4 & 1
\end{array}\right),\;
\left(\begin{array}{rrrrrrr}
0 & 0 & 1 & 1 & 1 & 1 & 2 \\
2 & 2 & 1 & 1 & 1 & 1 & 0 \\
0 & 0 & 1 & 1 & 1 & 1 & 0 \\
2 & 0 & 1 & 1 & 1 & 1 & 0 \\
0 & 2 & 1 & 1 & 1 & 1 & 0 \\
2 & 0 & 1 & 1 & 1 & 1 & 2 \\
0 & 2 & 1 & 1 & 1 & 1 & 2
\end{array}\right),\;
\left(\begin{array}{rrrrrrr}
2 & 1 & 1 & 1 & 1 & 1 & 1 \\
0 & 1 & 1 & 1 & 1 & 1 & 1 \\
0 & 1 & 1 & 1 & 1 & 1 & 1 \\
2 & 1 & 1 & 1 & 1 & 1 & 1 \\
2 & 1 & 1 & 1 & 1 & 1 & 1 \\
0 & 1 & 1 & 1 & 1 & 1 & 1 \\
0 & 1 & 1 & 1 & 1 & 1 & 1
\end{array}\right),\;
\left(\begin{array}{rrrrrrr}
1 & 2 & 1 & 1 & 0 & 1 & 1 \\
1 & 0 & 1 & 1 & 2 & 1 & 1 \\
1 & 0 & 1 & 1 & 0 & 1 & 1 \\
1 & 0 & 1 & 1 & 2 & 1 & 1 \\
1 & 0 & 1 & 1 & 0 & 1 & 1 \\
1 & 2 & 1 & 1 & 2 & 1 & 1 \\
1 & 2 & 1 & 1 & 0 & 1 & 1
\end{array}\right),\\
&\setlength{\arraycolsep}{1.5pt}
\renewcommand{\arraystretch}{0.75}
\left(\begin{array}{rrrrrrr}
2 & 0 & 0 & 1 & 2 & 1 & 0 \\
2 & 0 & 2 & 1 & 0 & 1 & 2 \\
2 & 2 & 0 & 1 & 0 & 1 & 0 \\
0 & 0 & 0 & 1 & 0 & 1 & 2 \\
0 & 2 & 2 & 1 & 0 & 1 & 0 \\
0 & 2 & 0 & 1 & 2 & 1 & 2 \\
0 & 0 & 2 & 1 & 2 & 1 & 0
\end{array}\right),\;
\left(\begin{array}{rrrrrrr}
1 & 1 & 0 & 0 & 1 & 0 & 0 \\
1 & 1 & 2 & 2 & 1 & 0 & 2 \\
1 & 1 & 0 & 0 & 1 & 0 & 0 \\
1 & 1 & 2 & 2 & 1 & 0 & 2 \\
1 & 1 & 0 & 0 & 1 & 0 & 0 \\
1 & 1 & 2 & 2 & 1 & 0 & 2 \\
1 & 1 & 0 & 0 & 1 & 4 & 0
\end{array}\right),\;
\left(\begin{array}{rrrrrrr}
0 & 2 & 0 & 2 & 0 & 2 & 0 \\
2 & 0 & 0 & 2 & 2 & 0 & 0 \\
0 & 0 & 2 & 2 & 0 & 0 & 2 \\
2 & 0 & 0 & 0 & 0 & 2 & 2 \\
0 & 0 & 2 & 0 & 2 & 2 & 0 \\
2 & 2 & 2 & 0 & 0 & 0 & 0 \\
0 & 2 & 0 & 0 & 2 & 0 & 2
\end{array}\right).
  \end{align*}
  The computations were done with SageMath.

  From this point onward it is fairly easy to solve the RRP by hand.
  A closer look at \eqref{eq:copt} reveals that rows $i$ with $r_i<6$
  must be part of any optimal solution (since they strictly increase
  the net gain) and those with $r_i=6$ can be included w.l.o.g.\ in any
  optimal solution (since they cannot decrease the
  net gain). Moreover, since $\wham(\vek{x}\mat{C}_W)\leq m$, the optimal local
  net gain is upper bounded by $\sum_{i;r_i<6}(6-r_i)+m$.

  These observations give that the 1st and 6th collision matrix have a
  maximum local net gain $\geq 12$ (corresponding to $I=\{2,3,4,5\}$
  and $I=\{1,3,5\}$, respectively) and allows us to discard the other
  $5$ collision matrices, whose optimal values are bounded by $9$,
  $7$, $9$, $9$, and $7$ (in that order). Then it is easy to complete
  the solution: The overall maximum local net gain is $12$, and is attained
  precisely for the following $\mat{C}_W$ and $I$: $I=\{2,3,4,5\}$,
  $\{1,2,3,4,5\}$, $\{2,3,4,5,6\}$ for the 1st collision matrix
  (corresponding to $W=\langle 1,\alpha,\alpha^2\rangle$) and $I=\{1,3,5\}$,
  $\{1,3,5,7\}$, $\{1,2,3,5\}$, $\{1,3,4,5\}$, $\{1,3,5,6\}$,
  $\{1,2,3,5,7\}$, $\{1,3,4,5,7\}$, $\{1,3,5,6,7\}$, for the 6th
  collision matrix
  (corresponding to $W=\langle 1,\alpha^3,\alpha^{18}\rangle$).

  Thus the largest subspace codes that are obtained as solutions of
  the RRP for $v=9$ have size $2^{12}+12\cdot 63=4852$ and exceed the
  LMRD code bound $2^{12}+\gauss{6}{3}{2}=4747$.

  Finally note that the collision
  matrices with optimal value $12$ are exactly those which have an entry
  $c_{ij}=4$.
\end{example}
Now we examine the collision matrices $\mat{C}_W$ in more detail.
Clearly $\mat{C}_W$ is nonnegative and integer-valued, and it appears
from the preceding examples that the only values which occur in
$\mat{C}_W$ are $0,1,2,4$ and the distribution of these values in
every column is restricted to a few different types. In our next
theorem we will prove this and several other properties of
$\mat{C}_W$, which facilitate the solution of the optimization problem
\eqref{eq:copt}. In the statement of the theorem we use the \emph{type} (or
\emph{spectrum}) of a row or column, which refers to the multiset of
its entries, with $0$ omitted. Thus, e.g., $1^42^2$ refers to a row or
column of $\mat{C}_W$ containing four $1$'s, two $2$'s and
$2^{n-3}-1-6$ zeros.

\begin{theorem}
  \label{thm:C}
  Let $W$ be a plane in $\PG(\F_{2^n})$ and $\mat{C}_W\in\Z^{m\times
    m}$, $m=2^{n-3}-1$, the associated collision matrix.
  \begin{enumerate}[(i)]
  \item\label{thm:C:types} The columns of $\mat{C}_W$ have type $1^7$, $2^3$, or
    $4^1$. More precisely, a column labeled with $y\in(W^2)^\perp$ has
    type $1^7$ if $y$ is not a missing value of $\sickson_W$ (i.e.,
    $y\neq\dickson(W)/\dickson(Z)^3$ for all lines $Z\subset W$), type
    $2^3$ if $y$ is a missing value of multiplicity $1$ (i.e.,
    $y=\dickson(W)/\dickson(Z)^3$ for exactly one line $Z\subset W$),
    and type $4^1$ if $y$ is a missing value of multiplicity $3$
    (i.e., $y=\dickson(W)/\dickson(Z)^3$ for three lines $Z\subset
    W$). Moreover, Type~$4^1$ does not occur if $n$ is odd, and occurs
    at most once as a column of $\mat{C}_W$ if $n$ is even.
  \item\label{thm:C:subspaces} The support of each column forms a
    subspace of $\F_{2^n}/W$ (a plane if the type is $1^7$, a line if
    the type is $2^3$ and, trivially, a point if the type is $4^1$).
  \item\label{thm:C:parity} All rows of $\mat{C}_W$ have the same
    parity, equal to the parity of the number of columns of type
    $1^7$.\footnote{This property may seem trivial from the shape of
      the collision matrices in Example~\ref{ex:v=9}, but for $v>9$
      there are no all-one columns and hence this property is no
      longer obvious.}
  \end{enumerate}
\end{theorem}
\begin{proof}
  \eqref{thm:C:types} First we show that the multiplicities of $y$ as a
  missing point of $\sickson_W$ an their occurrences
  must be as indicated. The maximum
  multiplicity is $3$, since $\dickson(Z)^3=\dickson(W)/y$ can have at
  most $3$ solutions $Z$ (cf.\ Lemma~\ref{lma:dickson} and
  Theorem~\ref{thm:allthat}\eqref{thm:allthat:2}). If there are
  two different solutions $Z_1$, $Z_2$ then $\omega=\dickson(Z_2)/\dickson(Z_1)$
  must be a primitive 3rd root of unity in $\F_{2^n}$,
  which forces $n\equiv0\pmod{2}$. Moreover, denoting
  the third line in $W$ through the
  intersection point $Z_1\cap Z_2$ by $Z_3$, we then have
  \begin{equation*}
    \omega^2\dickson(Z_1)=\dickson(Z_1)+\dickson(Z_2)=\dickson(Z_3),
  \end{equation*}
  and $Z_3$ is a third solution. Since the line
  $\bigl\{\dickson(Z_1),\dickson(Z_2),\dickson(Z_3)\bigr\}
  =\F_4^\times\dickson(Z_1)$ is a member of the standard line spread of
  $\PG(\F_{2^n})$, the plane $W'=\bigl\{\dickson(Z);Z\subset W\bigr\}$
  cannot contain a further such line, showing that there is at most
  one missing value of multiplicity $3$.

  Next we set $\{Z;Z\subset W\}=\{Z_i;1\leq i\leq 7\}$ and let
  $E_i\supset Z_i$ be the corresponding plane satisfying
  $\dickson(E_i)/\dickson(Z_i)^3=y$ ($E_i=W$ is allowed here).
  Using the alternative expression for $\dickson(E_i)$ in terms of
  $\dickson(Z_i)$ and $\spol_{Z_i}(X)$, we can write these equations
  as $\spol_{Z_i}(x_i)/\dickson(Z_i)^2=y$, $x_i\in E_i\setminus
  Z_i$. Using \eqref{eq:ilinfact}, we obtain
  \begin{align*}
    \spol_W(x_i)&=\spol_{Z_i}(x_i)^2+\spol_{Z_i}(c_i)\spol_{Z_i}(x_i)\\
    &=\bigl(y\dickson(Z_i)^2\bigr)^2+\spol_{Z_i}(c_i)y\dickson(Z_i)^2\\
    &=y^2\dickson(Z_i)^4+y\dickson(W)\dickson(Z_i)
  \end{align*}
  with $c_i\in
  W\setminus
  Z_i$. This shows that
  $\spol_W(x_i)=f\bigl(\dickson(Z_i)\bigr)$
  is in the image of the plane $W'=\bigl\{\dickson(Z);Z\subset
  W\bigr\}=\dickson(W)W^\adj$ under the
  $\F_2$-linear
  transformation $f(x)=y^2x^4+y\dickson(W)x$.
  But $\kernel(f|_{W'})$
  consists of $0$
  and all elements $\dickson(Z_i)$
  satisfying $\dickson(W)/\dickson(Z_i)^3=y$,
  and hence has dimension $0$,
  $1$,
  or $2$.
  Applying the homomorphism theorem for linear maps, the remaining
  assertions of \eqref{thm:C:types} follow.\footnote{It should be
    noted that
    $\spol_W(x_i)=y^2\dickson(Z_i)^4+y\dickson(W)\dickson(Z_i)$
    is equivalent to $\spol_{Z_i}(x_i)/\dickson(Z_i)^2=y
    \vee\spol_{Z_i}(x_i+c_i)/\dickson(Z_i)^2=y$. Both planes $\langle
    Z_i,x_i\rangle$, $\langle
    Z_i,x_i+c_i\rangle$ are in the same solid $T=\langle
    W,x_i\rangle$ and only one of them can be a solution of the
    equation. The number of solutions in any ``point'' $T$
    of $\F_{2^n}/W$
    is thus invariant under the transformation just made.}

  \eqref{thm:C:subspaces} This has been already shown as part of the proof of
  \eqref{thm:C:types}.

  \eqref{thm:C:parity} If $(W^2)^\perp$ contains $\mu$ missing points
  of $\sickson_W$, the number of columns of $\mat{C}_W$ of type
  $1^7$ is equal to $2^{n-3}-1-\mu$.

  On the other hand, consider a
  solid $T_i\supset W$. For any line $Z\subset W$, there are $2$
  planes $E_1,E_2\subset T_i$ such that $E_1\cap W=E_2\cap W=Z$. The image of
  $\{E_1,E_2,W\}$ under $E\mapsto\dickson(E)/\dickson(Z)^3$ is a line in
  $\PG(\F_{2^n})$ through the missing point
  $\dickson(W)/\dickson(Z)^3$. Hence the restriction of $\sickson_W$ to the $14$
  planes $E\neq W$ in $T_i$ determines $7$ lines, one line through
  each missing point,\footnote{Again the missing points are counted
    with their multiplicity.} containing the $14$ values
  $\sickson_W(E)$. Since $(W^2)^\perp$ forms a hyperplane in the image
  $\dickson(Z)^{-2}Z^\opp$, these lines are either contained in
  $(W^2)^\perp$ or meet $(W^2)^\perp$ in a unique point. Hence $Z$
  contributes $0$ or $2$ to the row sum $r_i$ if
  $\dickson(W)/\dickson(Z)^3\in(W^2)^\perp$, and $1$ to $r_i$ if
  $\dickson(W)/\dickson(Z)^3\notin(W^2)^\perp$.\footnote{In the first
    case, the contribution is $0$ if the corresponding line meets
    $(W^2)^\perp$ in the missing point $\dickson(W)/\dickson(Z)^3$,
    and $2$ otherwise.}
  The parity of $r_i$ is thus equal to
  $7-\mu$. But $7-\mu\equiv2^{n-3}-1-\mu\pmod{2}$, and the proof is
  complete.
\end{proof}

As we have seen in Example~\ref{ex:v=9}, knowledge of the number of
rows of $\mat{C}_W$ with $r_i\leq 6$ provides important information
about the optimal solutions of \eqref{eq:copt} and, in particular, can
be used to bound the maximum net gain achievable when using $W$. In
view of its importance, we now state this result in the general case.
The \emph{row-sum spectrum} of $\mat{C}_W$ refers to the multiset of
row sums of $\mat{C}_W$ and is denoted by $0^{m_0}1^{m_1}2^{m_2}\dotsm$ if
there are $m_r$ rows with row sum $r$.
\begin{corollary}
  \label{cor:copt}
  Suppose $\mat{C}_W$ has row-sum spectrum $0^{m_0}1^{m_1}2^{m_2}\dotsm$,
  and the union of the supports of the $m_0+m_1+\dots+m_6$ rows of $\mat{C}_W$
    with $r_i\leq 6$ (equal to the number of nonzero columns of the
    corresponding submatrix $\mat{C}_W(I)$) is lower-bounded by $m'$.
    Then the optimal value $N_1$ of \eqref{eq:copt} (i.e., the maximum
  net gain achievable relative to $W$) satisfies the bounds
  \begin{equation*}
    \sum_{r=0}^5m_r(6-r)+m'\leq N_1\leq\sum_{r=0}^5m_r(6-r)+m,
  \end{equation*}
  where $m=2^{n-3}-1$ is the order of $\mat{C}_W$.
\end{corollary}
\begin{proof}
  This is immediate from Theorem~\ref{thm:copt}.
\end{proof}
As we will see in a moment, the row-sum spectrum of $\mat{C}_W$
depends only on the geometric configuration of the (at most $7$) missing
points contained in $(W^2)^\perp$ and hence is quite
restricted. Finding a good lower bound $m'$ without actually computing
$\mat{C}_W$ seems to be more difficult. A reasonable approach to solve
this problem is to find a good upper bound $l$ for the column sums of
$\mat{C}_W(I)$ and use the obvious fact that the number of nonzero
columns of $\mat{C}_W(I)$ must be at least
$\left(\sum_{r=0}^6rm_r\right)/l$.\footnote{The obvious bound $l\leq
  7$ won't do the job, of course, since the row sums of $\mat{C}_W(I)$
  are $\leq 6$ and hence no constant $l\geq 6$ can improve on the
  trivial bound $m$.}

From Theorem~\ref{thm:cspace} we have that for every line $Z\subset W$
there exists a unique hyperplane $H_Z\supset Z$ which is mapped onto
$(W^2)^\perp$ by $x\mapsto\dickson\bigl(\langle Z,x\rangle\bigr)/\dickson(Z)^3
=\spol_Z(x)/\dickson(Z)^2$.

\begin{theorem}
  \label{thm:HZ}
  \begin{enumerate}[(i)]
  \item\label{thm:HZ:eq} The hyperplane $H_Z$ has equation
    $\trace\bigl(\frac{\dickson(W)}{\dickson(Z)^3}\cdot x^2\bigr)=0$
    and hence is essentially the dual of the corresponding missing
    point of $\sickson_W$ under the trace bilinear form;
  \item\label{thm:HZ:mp} $H_Z\supset W$ iff
    $\frac{\dickson(W)}{\dickson(Z)^3}\in(W^2)^\perp$ iff
    $\trace(a_1c^4/a_0^2)=0$, where $a_1=\dickson_1(Z)$,
    $a_0=\dickson(Z)$ and $c\in W\setminus Z$.
  \item\label{thm:HZ:enum}
    For any solid $T_i\supset W$ the number of planes $E\neq W$
    contained in $T_i$ and satisfying $\sickson_W(E)\in(W^2)^\perp$
    (i.e., the row sum $r_i$ of $\mat{C}_W$) is equal to $7-\mu+2\nu$,
    where $\mu$ denotes the number of missing points of $\sickson_W$
    contained in $(W^2)^\perp$ and $\nu$ the number of
    hyperplanes $H_Z$ that contain $T_i$.
  \end{enumerate}
\end{theorem}
Note that $H_Z$ can contain $T_i$ only if it contains $W$. Hence the
$\nu$ hyperplanes in \eqref{thm:HZ:enum} are among those $\mu$ with
their corresponding missing point in $(W^2)^\perp$, and we can restate
the formula $r_i=7-\mu+2\nu$ in the following way: A hyperplane $H_Z$
contributes $0$, $1$, or $2$ to the row sum $r_i$ if $H_Z\supset
W$ and $H_Z\nsupseteq T_i$, $H_Z\nsupseteq W$, or $H_Z\supset
T_i$, respectively. For conditions equivalent to $H_Z\supset W$ see
\eqref{thm:HZ:mp}.
\begin{proof}[Proof of Theorem~\ref{thm:HZ}]
  \eqref{thm:HZ:eq} $x\in H_Z$ is equivalent to
  \begin{equation*}
    \innpro{x}{y}_Z
    =\trace\left(\frac{\spol_Z(x)y^2}{\dickson(Z)^2}\right)=0\quad
    \text{for all $y\in W$}.
  \end{equation*}
  Since $\innpro{\ }{\ }_Z$ is symmetric and has radical $Z$, this is
  equivalent to $\innpro{x}{c}_Z=\innpro{c}{x}_Z=0$ for any $c\in
  W\setminus Z$, i.e.\ to
  $\trace\left(\frac{\spol_Z(c)x^2}{\dickson(Z)^2}\right)
  =\trace\left(\frac{\dickson(W)x^2}{\dickson(Z)^3}\right)=0$.

  \eqref{thm:HZ:mp}. As in \eqref{thm:HZ:eq}, $H_Z\supset W$ is
  equivalent to
  $\innpro{c}{c}_Z
  =\trace\left(\frac{\spol_Z(c)c^2}{\dickson(Z)^2}\right)=0$,
  which in turn is equivalent to $\innpro{c}{y}_Z=0$ for all $y\in W$
  and hence to
  $\frac{\dickson(W)}{\dickson(Z)^3}
  =\frac{\spol_Z(c)}{\dickson(Z)^2}\in(W^2)^\perp$.
  The second equivalence follows from
  $\innpro{c}{c}_Z=\trace(a_1c^4/a_0^2)$; cf.\ Lemma~\ref{lma:innpro}.

  \eqref{thm:HZ:enum} is proved using \eqref{thm:HZ:mp} and the
  reasoning in the proof of
  Theorem~\ref{thm:C}\eqref{thm:C:parity}. The case
  $\frac{\dickson(W)}{\dickson(Z)^3}\in (W^2)^\perp$ now splits into two
  subcases according to whether the image of
  $\{E_1,E_2,W\}$ is contained in $(W^2)^\perp$ or not. The
  first case is equivalent to $T_i\subseteq H_Z$ and accounts for
  $2$ values $\sickson(E_1)=\sickson(E_2)\in(W^2)^\perp$, the second
  case for $0$ values.
\end{proof}
Using Theorem~\ref{thm:HZ}, the row-sum spectrum of $\mat{C}_W$ can be
determined from the multiset $\mset{m}$ of missing points contained in
$(W^2)^\perp$ in the same way as the weight enumerator of a binary
linear $[\mu,k]$ code with associated multiset $\mset{m}$, represented
by the columns of a generator matrix of the code. For the latter it is
usually assumed that the multiset spans the geometry, which in our
case need not be true. However, it is easy to reduce the spectrum
computation to this case: Denoting by $M$ the hull of $\mset{m}$
(i.e., the subspace generated by the missing points in $(W^2)^\perp$),
we compute the associated weight distribution
$(A_i)_{0\leq i\leq\mu}$, replace nonzero weights $i$ by the
corresponding row sums $2(\mu-i)+7-\mu=\mu+7-2i$ and scale the
frequencies $A_i$ by $2^{n-3-\dim(M)}$.  If $M$ is a proper subspace
of $(W^2)^\perp$, there are in addition $2^{n-3-\dim(M)}$ rows of
$\mat{C}_W$ corresponding to the all-zero codeword. These correspond
to the solids $T_i$ contained in
$(M^{1/2})^\perp=\bigcap\{H_Z;H_Z\supset W\}$ and have maximum row sum
$r_i=7-\mu+2\mu=7+\mu$.\footnote{All other row sums are
  $\leq 7-\mu+2(\mu-1)=5+\mu$.} We will illustrate row-sum spectrum
computations later in the proofs of Theorems~\ref{thm:n=4mod8}
and~\ref{thm:n=0mod8}.

For all even $n$ (i.e., odd packet lengths $v$) explored so far, the
maximum net gain of the RRP is achieved only by planes $W$ whose
collision matrices have an entry $c_{ij}=4$. It is therefore of
interest, to characterize these planes. For the statement of the
following theorem, we denote the trace-zero hyperplane of
$\PG(\F_{2^n})$ by $H_2$.\footnote{Thus
  $H_2=\bigl\{x\in\F_{2^n};\trace_2(x)=0\bigr\}$, where
  $\trace_2(x)=\trace(x)=x+x^2+x^4+\dots+x^{2^{n-1}}$. The index used
  is thus equal to the order of base field of the corresponding
  field extension.}
\begin{theorem}
  \label{thm:cij=4}
  \begin{enumerate}[(i)]
  \item\label{thm:cij=4:gen}
    Suppose $n$ is even and $\omega$ is a generator of the subfield
    $\F_4\subset\F_{2^n}$. A plane $W$ in $\PG(\F_{2^n})$ gives rise to an
    entry $c_{ij}=4$ in the collision matrix $\mat{C}_W$ if and only if
    $W=rW_1$, $r\in\F_{2^n}^\times$, for some plane $W_1=\langle
    1,a,b\rangle$ with $a,b$ satisfying $b^2+b=\omega(a^2+a)$.
  \item\label{thm:cij=4:part} The planes $W_1$ of the type indicated in
    \eqref{thm:cij=4:gen} are contained
    in $H_2$, mutually intersect in the point $\F_2=\F_21$
    of $\PG(\F_{2^n})$, and determine a line spread of the
    ``sub-quotient'' geometry $\PG(H_2/\F_2)\cong\PG(n-3,\F_2)$. In
    particular, the number of such planes is $(2^{n-2}-1)/3$.
  \item\label{thm:cij=4:missing}
    For a plane $W_1$ of the type indicated in
    \eqref{thm:cij=4:gen}, the missing points of $\sickson_{W_1}$ are
    $1$ (of multiplicity $3$) and $(b+\omega a+x)^{-3}$ for $x\in\F_4$ (of
    multiplicity $1$).
\end{enumerate}
\end{theorem}
\begin{proof}
  \eqref{thm:cij=4:gen}
  Since the indicated property is $G$-invariant, we may assume $1\in
  W$ and that the three lines $Z_1,Z_2,Z_3$ containing $1$ give rise
  to the missing point of multiplicity $3$, i.e.\
  $\dickson(Z_1)^3=\dickson(Z_2)^3=\dickson(Z_3)^3$.

  Now let $Z_1=\langle 1,a\rangle$, $Z_2=\langle 1,b\rangle$, and
  hence $Z_3=\langle 1,a+b\rangle$. Then
  $b^2+b=\dickson(Z_2)=\omega^i\dickson(Z_1)=\omega^i(a^2+a)$ for some
  $i\in\{1,2\}$, and by interchanging $a,b$, if necessary, we may
  assume $i=1$.

  Conversely, assume that $W=\langle 1,a,b\rangle$ with $a,b$ having
  the indicated property.
  Then, with $Z_i$ as in
  \eqref{thm:cij=4:gen}, we have $\dickson(Z_i)^3=(a^2+a)^3$,
  $\dickson(W)=\dickson(Z_1)\dickson(Z_2)\dickson(Z_3)/1^2
  =\omega^{0+1+2}(a^2+a)^3=(a^2+a)^3$, and hence the triple missing
  point is $\dickson(W)/\dickson(Z_i)^3=1$.
  Further, since
  $\trace_{\F_{2^n}/\F_4}(a^2+a)=\trace(a)$ and similarly for $b$, we
  must have $\trace(b)=\omega\trace(a)$ and hence
  $\trace(a)=\trace(b)=0$. This implies
  $1\in(W^2)^\perp$, and hence $y=1$ gives rise to a column of
  $\mat{C}_W$ of Type~$4^1$ by
  Theorem~\ref{thm:C}\eqref{thm:C:types}.\footnote{A more geometric
    proof of the fact that a missing point of multiplicity $3$ must be
    in $(W^2)^\perp$ is the following: Consider the $7$ lines
    determined by the restriction of $\sickson_W$ to a fixed solid
    $T_i\supset W$; cf.\ the proof of
    Theorem~\ref{thm:C}\eqref{thm:C:parity}. Since these lines are contained
    in the corresponding spaces $\dickson(Z)^{-2}Z^\opp$, they can
    only intersect in $(W^2)^\perp$. However, the $3$ lines containing
    the triple missing point intersect in this point, and hence this
    point must be in $(W^2)^\perp$.}

  \eqref{thm:cij=4:part} 
  In the proof of \eqref{thm:cij=4:gen} we have seen that such planes
  $W_1$ are contained in $H_2$. Since $\trace(a)=0$ is equivalent to
  $\trace_{\F_{2^n}/\F_4}(a^2+a)=0$, the map $\F_{2^n}\to\F_{2^n}$, $x\mapsto x^2+x$
  induces an isomorphism from $H_2/\F_2$ onto the trace-zero
  subspace $H_4$ of the field extension $\F_{2^n}/\F_4$. By
  Part~\eqref{thm:cij=4:gen}, it maps the set of planes $W_1$ of the
  indicated type onto the standard line spread in
  $\PG(\F_{2^n}/\F_4)$. The result follows.

  \eqref{thm:cij=4:missing}
  We know already that $1$ is the missing
  point of multiplicity $3$.
  Since the plane $\bigl\{\dickson(Z);Z\subset W_1\text{ a line}\bigr\}$ is
  generated by $\F_4(a^2+a)$ and $\dickson\bigl(\langle
  a,b\rangle\bigr)=ab^2+a^2b$, the remaining $4$ missing points are
  \begin{align*}
    \frac{(a^2+a)^3}{\bigl(ab^2+a^2b+x(a^2+a)\bigr)^3}
    &=\frac{(a^2+a)^3}{\bigl(ab+\omega(a^3+a^2)+x(a^2+a)\bigr)^3}
    =\frac{1}{(b+\omega a+x)^3}
  \end{align*}
  with $x\in\F_4$, as asserted.
\end{proof}
\begin{remark}
  \label{rmk:abomega}
  The map $\F_{2^n}\to\F_{2^n}$, $x\mapsto x^2+x$ induces also an
  $\F_2$-isomorphism from $\F_{2^n}/\F_4$ onto $H_2/\F_2$, hence gives rise
  to the chain $\F_{2^n}/\F_4\to H_2/\F_2\to H_4$ of
  $\F_2$-isomorphisms.\footnote{This property
    is reflected in the symbolic factorization
    $X^4+X=(X^2+X)\circ(X^2+X)$.}

  The points $\F_4(b+\omega a)$ with $a\in H_2\setminus\F_2$ and $b$
  determined as in Theorem~\ref{thm:cij=4}\eqref{thm:cij=4:gen} form
  a system of representatives for the nonzero cosets in
  $\F_{2^n}/\F_4$ and hence for the lines in
  $\PG(\F_{2^n}/\F_4)\cong\PG(n/2-1,\F_4)$ that pass through the point
  $\F_4=\F_41$. This can be seen as follows: Since
  \begin{align*}
    (b+\omega a)^2+b+\omega a&=b^2+b+\omega^2a^2+\omega a=a^2,\\
    (\omega b+\omega^2a)^2+\omega b+\omega^2a
                             &=\omega^2b^2+\omega b+\omega a^2+\omega^2a=b^2+a^2,
  \end{align*}
  the line $L=\F_4(b+\omega a)+\F_4$ is mapped to the plane
  $W_1^2=\langle 1,a^2,b^2\rangle$ by $x\mapsto x^2+x$, and the planes
  of this form partition $H_2/F_2$; cf.\
  Theorem~\ref{thm:cij=4}\eqref{thm:cij=4:part}. 

  Moreover, by Theorem~\ref{thm:cij=4}\eqref{thm:cij=4:missing} the missing
  points of $\sickson_{W_1}$ are just the reciprocal cubes of the $5$
  points on the line $L=\F_4(b+\omega a)+\F_4$.\footnote{Note that the
    cube $x^3$ of a point $\F_4x$ is well-defined.}

  These observations imply that each element $\neq 1$ of the index-$3$
  subgroup of $\F_{2^n}^\times$ forms a missing point for precisely
  one plane $W_1$ of the type indicated in
  Theorem~\ref{thm:cij=4}\eqref{thm:cij=4:gen}.  
\end{remark}
As an aside, making the link with Section~\ref{sec:allthat}, we
note that the plane polynomials of the planes
$W_1=\langle 1,a,b\rangle$ in Theorem~\ref{thm:cij=4} are
$\spol_{W_1}(X)=\bigl(X^4+(a^2+a)^3X)\bigr)\circ(X^2+X)
=X^8+X^4+(a^2+a)^3X^2+(a^2+a)^3X$.

\section{Computational Results}\label{sec:comp}

In this section we provide an account of explicit maximum net gain
computations for $10\leq v\leq 15$, which we have done using the
computer algebra package SageMath. The computations were exhaustive
for $v\leq 13$. In the case $v=13$ ($n=10$) there are $633$ $G$-orbits
to process. For each $G$-orbit representative $W$ we have computed the
collision matrix $\mat{C}_W$ (of size $127\times 127$ for $v=13$) and
the bounds for the maximum net gain $N_1$ relative to $W$ stated in
Corollary~\ref{cor:copt}. Then, in a second pass through the list of
$G$-orbit representatives, this time sorted in order of decreasing
lower bounds for $N_1$, we have computed the exact maximum net gains
$N_1$ for those $G$-orbits, for which the upper bound still exceeded
the current ``absolute'' maximum net gain (taken over all $G$-orbits
computed so far). The actual optimization routine used
some greedy heuristic for selecting rows of $\mat{C}_W$ with
row sums $>6$ as part of the next-to-be-tested feasible solution.

For $v\in\{14,15,16\}$ exhaustive computations were not feasible, and we
have restricted the search to those $G$-orbits, which contain a plane
$W_1$ of the type discussed in Theorem~\ref{thm:cij=4}, or a subset
thereof. In Section~\ref{sec:code} we will show that the absolute
maximum net gains obtained for $v\in\{14,15\}$ nevertheless represent
the true maximum as well.

The computational results are summarized in Table~\ref{tbl:netgain},
including the cases $7\leq v\leq 9$ already discussed. The table
contains for each length $v$ the number of $G$-orbits processed (for
$v\geq 13$ equal to the total number of $G$-orbits), the absolute
maximum local net gain $N_1$ computed (with the possible exception of $v=16$
equal to the true maximum), the local net gain equivalent of the LMRD
code bound (``LMRD threshold''), the size of the plane subspace codes
corresponding to optimal solutions of the RRP, and a representative
subspace $W$ giving rise to an (absolutely) optimal solution. The
generators of $W$ are given as powers of a primitive $\alpha$ of
$\F_{2^n}$ (root of the Conway polynomial of degree $n$, as used by
SageMath). The next few paragraphs contain supplementary remarks on
each case.

\paragraph{$v=7$} 
According to \cite{lt:0328,mt:alb80}, there exist solutions
$\mathcal{C}$ of the RRP that can be extended by $28$
further planes meeting $S=\{0\}\times\mathbb{F}_{16}$ in a line to
a currently best known $(7,329,4;3)_2$ code. However, no
rotation-invariant $(7,301,4;3)_2$ code $\mathcal{C}$ has this
property.\footnote{M.~Kiermaier, personal communication}

\paragraph{$v=8$} Using a modified 
beam-search algorithm \cite{braun-reichelt14}, we have found that 
one of the optimal $(8,1117,4;3)_2$ solutions $\mathcal{C}$ of the RRP
can be augmented by 142 extra planes meeting
$S=\{0\}\times\mathbb{F}_{32}$ in a line to a $(8,1259,4;3)_2$
code. This is considerably better than the LMRD code bound
$1024+155=1179$, but it falls short of the currently best known
code of size $1326$.

\paragraph{$v=9$} The seven $7\times 7$ collision matrices
corresponding to the seven $G$-orbits were already listed in
Example~\ref{ex:v=9}. Two $G$-orbits, with orbit representatives
$\langle1,\alpha,\alpha^2\rangle$ and
$\langle1,\alpha^3,\alpha^{18}\rangle$, yield the absolute maximum
local net gain $12$, resulting in $(9,4852,4;3)_2$ codes. We have found that
$162$ planes meeting $S=\{0\}\times\F_{64}$ can be added to one of the
codes, increasing the code size to $5014$. The currently best known
code has size $5986$ \cite{braun-ostergard-wassermann15}.

\paragraph{$v=10$} In this case all fifteen $15\times15$ collision
matrices were computed. The absolute maximum local net gain 20 is obtained from
the three $G$-orbits
with representatives $\langle1,\alpha,\alpha^{24}\rangle$,
$\langle1,\alpha,\alpha^{39}\rangle$ and
$\langle1,\alpha,\alpha^{22}\rangle$, resulting in $(10,18924,4;3)_2$
codes. The size of these codes is smaller than the LMRD code bound 
$2^{14}+\gauss{7}{2}{2}+19051$, but again a further extension step
by planes meeting $\{0\}\times\mathbb{F}_{128}$ in a line (in this
case 1593 codewords can be added to one of the codes) increases the
code size to $20517>19051$. The currently best known code has size
$23870$ \cite{braun-ostergard-wassermann15}.

\paragraph{$v=11$} Here we have $n=8$ and the collision matrices 
have already size $31\times 31$. Among the $53$ $G$-orbits, the
orbit containing 
$W=\langle1,\alpha^{17},\alpha^{34}\rangle
=\bigl\{x\in\F_{16};\trace(x)=0\bigr\}$ uniquely gives the absolute 
maximal local net gain 54, resulting in a subspace code of size
$2^{16}+54\cdot(2^8-1)=79306$. This is better than the LMRD code bound
$2^{16}+10795=76331$, but should also be compared to the size $97526$
of the currently best known code
\cite{braun-ostergard-wassermann15}. The collision matrix $\mat{C}_W$
is shown in Figure~\ref{fig:CWn=8}.
\begin{figure}
  \centering
    \(
    \setlength{\arraycolsep}{1.5pt}
  \renewcommand{\arraystretch}{0.75}
  \left(\begin{array}{ccccccccccccccccccccccccccccccc}
  0 & 0 & 0 & 1 & 1 & 0 & 0 & 2 & 1 & 0 & 0 & 1 & 1 & 0 & 0 & 0 & 0 & 0 & 0 & 0 & 0 & 1 & 1 & 0 & 0 & 0 & 0 & 0 & 0 & 1 & 0 \\
  0 & 0 & 0 & 0 & 0 & 1 & 1 & 0 & 0 & 1 & 1 & 0 & 0 & 0 & 0 & 0 & 0 & 0 & 0 & 0 & 1 & 0 & 0 & 1 & 1 & 0 & 0 & 0 & 0 & 1 & 2 \\
  0 & 2 & 1 & 0 & 0 & 0 & 0 & 0 & 0 & 0 & 0 & 0 & 0 & 0 & 0 & 1 & 1 & 1 & 1 & 0 & 0 & 0 & 0 & 0 & 0 & 0 & 0 & 1 & 1 & 1 & 0 \\
  0 & 0 & 0 & 0 & 0 & 0 & 0 & 0 & 0 & 0 & 0 & 0 & 0 & 0 & 0 & 0 & 1 & 0 & 1 & 0 & 0 & 0 & 0 & 0 & 0 & 0 & 0 & 0 & 0 & 0 & 0 \\
  0 & 0 & 0 & 0 & 0 & 0 & 0 & 0 & 0 & 0 & 0 & 0 & 0 & 1 & 0 & 0 & 0 & 0 & 0 & 0 & 0 & 0 & 0 & 0 & 0 & 1 & 0 & 0 & 0 & 0 & 0 \\
  0 & 0 & 0 & 0 & 1 & 0 & 0 & 0 & 0 & 0 & 0 & 0 & 1 & 0 & 0 & 0 & 0 & 0 & 0 & 0 & 0 & 0 & 0 & 0 & 0 & 0 & 0 & 0 & 0 & 0 & 0 \\
  0 & 0 & 0 & 0 & 1 & 0 & 0 & 0 & 0 & 0 & 0 & 0 & 1 & 0 & 1 & 0 & 1 & 0 & 1 & 0 & 0 & 0 & 0 & 0 & 0 & 0 & 1 & 0 & 0 & 0 & 0 \\
  0 & 0 & 0 & 0 & 0 & 1 & 1 & 0 & 1 & 0 & 0 & 0 & 0 & 0 & 0 & 0 & 0 & 0 & 0 & 0 & 0 & 1 & 1 & 1 & 1 & 0 & 0 & 0 & 0 & 1 & 2 \\
  4 & 0 & 1 & 0 & 0 & 0 & 0 & 0 & 1 & 1 & 1 & 0 & 0 & 0 & 0 & 0 & 0 & 0 & 0 & 0 & 1 & 1 & 1 & 0 & 0 & 0 & 0 & 1 & 1 & 1 & 0 \\
  0 & 0 & 1 & 0 & 0 & 1 & 1 & 0 & 0 & 0 & 0 & 0 & 0 & 0 & 0 & 0 & 0 & 0 & 0 & 0 & 0 & 0 & 0 & 1 & 1 & 0 & 0 & 1 & 1 & 1 & 2 \\
  0 & 0 & 0 & 0 & 0 & 0 & 0 & 0 & 0 & 1 & 1 & 0 & 0 & 1 & 1 & 0 & 0 & 0 & 0 & 2 & 1 & 0 & 0 & 0 & 0 & 1 & 1 & 0 & 0 & 1 & 0 \\
  0 & 0 & 0 & 0 & 0 & 0 & 0 & 0 & 0 & 0 & 0 & 0 & 0 & 0 & 1 & 0 & 0 & 0 & 0 & 0 & 0 & 0 & 0 & 0 & 0 & 0 & 1 & 0 & 0 & 0 & 0 \\
  0 & 0 & 0 & 0 & 0 & 0 & 0 & 0 & 0 & 0 & 0 & 0 & 0 & 0 & 0 & 1 & 0 & 1 & 0 & 0 & 0 & 0 & 0 & 0 & 0 & 0 & 0 & 0 & 0 & 0 & 0 \\
  0 & 0 & 0 & 1 & 0 & 0 & 0 & 0 & 0 & 0 & 0 & 1 & 0 & 1 & 0 & 1 & 0 & 1 & 0 & 0 & 0 & 0 & 0 & 0 & 0 & 1 & 0 & 0 & 0 & 0 & 0 \\
  0 & 0 & 0 & 1 & 0 & 0 & 0 & 0 & 0 & 0 & 0 & 1 & 0 & 0 & 0 & 0 & 0 & 0 & 0 & 0 & 0 & 0 & 0 & 0 & 0 & 0 & 0 & 0 & 0 & 0 & 0 \\
  0 & 0 & 0 & 0 & 0 & 0 & 0 & 0 & 0 & 1 & 1 & 0 & 0 & 0 & 0 & 0 & 0 & 0 & 0 & 0 & 0 & 1 & 1 & 0 & 0 & 0 & 0 & 1 & 1 & 0 & 0 \\
  0 & 0 & 1 & 0 & 0 & 0 & 0 & 0 & 0 & 0 & 0 & 0 & 0 & 1 & 1 & 0 & 0 & 0 & 0 & 2 & 1 & 1 & 1 & 0 & 0 & 1 & 1 & 0 & 0 & 0 & 0 \\
  0 & 0 & 1 & 1 & 1 & 0 & 0 & 2 & 1 & 1 & 1 & 1 & 1 & 0 & 0 & 0 & 0 & 0 & 0 & 0 & 0 & 0 & 0 & 0 & 0 & 0 & 0 & 0 & 0 & 0 & 0 \\
  0 & 0 & 0 & 1 & 1 & 0 & 0 & 2 & 1 & 0 & 0 & 1 & 1 & 0 & 0 & 0 & 0 & 0 & 0 & 0 & 1 & 0 & 0 & 0 & 0 & 0 & 0 & 1 & 1 & 0 & 0 \\
  0 & 0 & 0 & 0 & 1 & 1 & 0 & 0 & 0 & 0 & 0 & 0 & 1 & 1 & 0 & 0 & 0 & 0 & 0 & 0 & 0 & 0 & 0 & 0 & 1 & 1 & 0 & 0 & 0 & 0 & 0 \\
  0 & 0 & 0 & 0 & 1 & 0 & 1 & 0 & 0 & 0 & 0 & 0 & 1 & 0 & 0 & 1 & 0 & 1 & 0 & 0 & 0 & 0 & 0 & 1 & 0 & 0 & 0 & 0 & 0 & 0 & 0 \\
  0 & 0 & 0 & 0 & 0 & 1 & 0 & 0 & 0 & 0 & 0 & 0 & 0 & 0 & 1 & 1 & 0 & 1 & 0 & 0 & 0 & 0 & 0 & 0 & 1 & 0 & 1 & 0 & 0 & 0 & 0 \\
  0 & 0 & 0 & 0 & 0 & 0 & 1 & 0 & 0 & 0 & 0 & 0 & 0 & 0 & 0 & 0 & 0 & 0 & 0 & 0 & 0 & 0 & 0 & 1 & 0 & 0 & 0 & 0 & 0 & 0 & 0 \\
  0 & 2 & 1 & 0 & 0 & 0 & 0 & 0 & 0 & 0 & 0 & 0 & 0 & 0 & 0 & 1 & 1 & 1 & 1 & 0 & 1 & 1 & 1 & 0 & 0 & 0 & 0 & 0 & 0 & 0 & 0 \\
  0 & 0 & 0 & 0 & 0 & 0 & 0 & 0 & 0 & 1 & 1 & 0 & 0 & 0 & 0 & 0 & 0 & 0 & 0 & 0 & 0 & 1 & 1 & 0 & 0 & 0 & 0 & 1 & 1 & 0 & 0 \\
  0 & 0 & 0 & 0 & 0 & 0 & 0 & 0 & 1 & 0 & 0 & 0 & 0 & 1 & 1 & 0 & 0 & 0 & 0 & 2 & 1 & 0 & 0 & 0 & 0 & 1 & 1 & 1 & 1 & 0 & 0 \\
  0 & 2 & 1 & 0 & 0 & 0 & 0 & 0 & 1 & 1 & 1 & 0 & 0 & 0 & 0 & 1 & 1 & 1 & 1 & 0 & 0 & 0 & 0 & 0 & 0 & 0 & 0 & 0 & 0 & 0 & 0 \\
  0 & 0 & 0 & 1 & 0 & 1 & 0 & 0 & 0 & 0 & 0 & 1 & 0 & 0 & 0 & 0 & 1 & 0 & 1 & 0 & 0 & 0 & 0 & 0 & 1 & 0 & 0 & 0 & 0 & 0 & 0 \\
  0 & 0 & 0 & 1 & 0 & 0 & 1 & 0 & 0 & 0 & 0 & 1 & 0 & 0 & 1 & 0 & 0 & 0 & 0 & 0 & 0 & 0 & 0 & 1 & 0 & 0 & 1 & 0 & 0 & 0 & 0 \\
  0 & 0 & 0 & 0 & 0 & 1 & 0 & 0 & 0 & 0 & 0 & 0 & 0 & 0 & 0 & 0 & 0 & 0 & 0 & 0 & 0 & 0 & 0 & 0 & 1 & 0 & 0 & 0 & 0 & 0 & 0 \\
  0 & 0 & 0 & 0 & 0 & 0 & 1 & 0 & 0 & 0 & 0 & 0 & 0 & 1 & 0 & 0 & 1 & 0 & 1 & 0 & 0 & 0 & 0 & 1 & 0 & 1 & 0 & 0 & 0 & 0 & 0
        \end{array}\right)
      \)
    \caption{The collision matrix $\mat{C}_W$ for $v=11$ ($n=8$) and
      $W=\langle1,\alpha^{17},\alpha^{34}\rangle
      =\bigl\{x\in\F_{16};\trace(x)=0\bigr\}$}\label{fig:CWn=8}
  \end{figure}
  This case is particularly important, since it serves as the ``anchor''
  case for the family of packet lengths $v\equiv3\pmod{8}$ considered
  in Theorem~\ref{thm:n=0mod8} and is thus ``responsible'' for the
  constant $81/64$ in the bound in Part~\eqref{thm:main:n=0mod8} of our
  main theorem.

  \paragraph{$v\geq12$} In the cases $v=12,13$ we were still able
  to process all $G$-orbits and compute the absolute maximum local net gains
  directly; cf.\ Table~\ref{tbl:netgain}.
  For lengths $v>13$, however, the amount of calculation is too large
  for processing all $G$-orbits exhaustively. Hence in the remaining
  cases $v=14,15,16$ we have processed only those $G$-orbits which
  appeared to be most ``promising'' in the sense that the lower bound
  in Corollary~\ref{cor:copt} is largest. The lower bound tends to be
  an increasing function of the number $\mu$ of missing points
  contained in the collision space $(W^2)^\perp$ and, in the case of
  odd $v$ (even $n$) to be maximized for the planes $W$ discussed in
  Theorem~\ref{thm:cij=4}.
  
  In the case $v=14$ we processed all $513$ $G$-orbits with $\mu\geq 5$
  missing points in $(W^2)^\perp$. There are $381$, $118$, $14$
  $G$-orbits corresponding to $\mu=5,6,7$, respectively, and the
  absolute maximum local net gain $379$ is attained uniquely at a
  $G$-orbit with $\mu=7$ (as was to be expected).

  For $v=15$, we processed all $34$ $G$-orbits containing planes $W$
  as in Theorem~\ref{thm:cij=4}. It turned out
  the absolute maximum local net gain $924$ is attained at a particular
  $G$-orbit with $\mu=3$, i.e., all $4$ missing points of multiplicity
  $1$ outside $(W^2)^\perp$.

  Finally, for $v=16$ we just processed all $G$-orbits with $\mu=7$
  and found for those an absolute maximum local net gain of
  $1526$. This is better than the LMRD code bound, which is equivalent
  to a local net gain of $1365$.

  Thus it appears that $v=8,10$ are the only cases where the optimal
  solutions of the RRP have size smaller than the LMRD code bound; cf.\ also
  Conjecture~\ref{conj:LMRD} in Section~\ref{sec:bingo}.
  

  \begin{table}[htbp]
  \centering
  \(
  \renewcommand{\arraystretch}{1.25}
  \begin{array}{|r|r||r|r|r|r|c|}
  \hline
  \multicolumn{1}{|c|}{v} 
    &\multicolumn{1}{c||}{n} 
    &\multicolumn{1}{c|}{\#\text{$G$-orbits}}
    &\multicolumn{1}{c|}{N_1} 
    &\multicolumn{1}{c|}{(N_1)_{\text{LMRD}}}
    &\multicolumn{1}{c|}{\#\mathcal{C}}
    &W\\\hline\hline
  7 &  4  & 1 & 3 & 2.33 & 2^8+45 &\langle1,\alpha,\alpha^2\rangle  \\ \hline
  8 &  5  & 1 & 3 & 5.00 & 2^{10}+93 &\langle1,\alpha,\alpha^2\rangle \\ \hline
  9 &  6  & 7 & 12 & 10.33 & 2^{12}+756 &\langle1,\alpha,\alpha^2\rangle
\\ \hline
  10 & 7  & 15 & 20 & 21.00 & 2^{14}+2540
&\langle1,\alpha,\alpha^{22}\rangle
    \\ \hline
  11 & 8  & 53 & 54 & 42.33 & 2^{16}+13770
    &\langle1,\alpha^{17},\alpha^{34}\rangle \\ \hline
  12 & 9  & 177 & 93 & 85.00 & 2^{18}+47523
    &\langle1,\alpha^3,\alpha^{71}\rangle \\ \hline
  13 & 10 & 633 & 234 & 170.33 & 2^{20}+239382  
    &\langle1,\alpha,\alpha^{49}\rangle \\ \hline
  14 & 11 & 513 & 379 & 341.00 & 2^{22}+775813
    &\langle1,\alpha^3,\alpha^{419}\rangle \\ \hline
  15 & 12 & 34 & 924 & 682.33 & 2^{24}+3783708 
    &\langle1,\alpha^{195},\alpha^{1170}\rangle \\ \hline
  16 & 13 & 240 & 1526 & 1365.00 & 2^{26}+12499466
    &\langle1,\alpha^{25},\alpha^{1208}\rangle \\ \hline
  \end{array}
  \)\\[1ex]
  \caption{Summary of maximum net gain computations}\label{tbl:netgain}
\end{table}

\section{Infinite Families of Subspace Codes Exceeding the 
LMRD Code Bound}\label{sec:bingo} 

We are now in a position to compute explicit lower bounds for the
maximum achievable net gain in the general RRP for packet lengths
$v\equiv 3\pmod{4}$ ($n\equiv 0\pmod{4}$), using a careful choice for
the plane $W$. It turns out that the corresponding modified subspace
codes exceed the LMRD code bound. The analysis will be split into
two cases depending on $v\bmod 8$. 
We start with the easier case $v\equiv 7\pmod{8}$.

\begin{theorem}
  \label{thm:n=4mod8}
  For packet lengths $v\equiv 7\pmod{8}$, i.e., $n=v-3\equiv
  4\pmod{8}$, the maximum achievable local net gain $N_1$ in the
  general RRP satisfies $N_1\geq 3\cdot 2^{n-4}=3\cdot 2^{v-7}$, and hence the
  corresponding optimum subspace codes have size
  \begin{equation*}
    \#\mathcal{C}\geq 2^{2(v-3)}+3\cdot 2^{v-7}(2^{v-3}-1).
  \end{equation*}
\end{theorem}
\begin{proof}
  Since $n\equiv 4\pmod{8}$, $\F_{16}$ is a subfield of $\F_{2^n}$ and
  we can choose $W$ as the trace-zero plane in $\F_{16}$.\footnote{The
    actual choice of $W$ does not matter, since
    all planes in $\F_{16}$ are rotated copies of each other (with
    factors $r\in\F_{16}^\times\subseteq\F_{2^n}^\times$) and hence in
    the same $G$-orbit. The subsequent proof, however, is only valid
    for the trace-zero plane, since it uses $W^2=W$.}  Writing
  $\F_{16}=\F_2(\xi)$ with $\xi^4+\xi+1=0$ and $\omega=\xi^5$, we have
  $\F_4=\F_2(\omega)$,
  $W=\{0,1,\xi,\xi^2,\xi^4,\xi^5,\xi^8,\xi^{10}\}=\langle\xi,\omega\rangle$,
  and $\xi^2+\xi=\omega=\omega(\omega^2+\omega)$. This shows that $W$
  is of the type considered in Theorem~\ref{thm:cij=4} with
  $a=\omega$, $b=\xi$.\footnote{Strictly speaking, we should also
    check that $\trace(\omega)=\trace(\xi)=0$ but this is trivial,
    since $\trace(x)=(n/4)\trace_{\F_{16}/\F_2}(x)$ for
    $x\in\F_{2^4}\subseteq\F_{2^n}$.}  Further, from
  $W'=\bigl\{\dickson(Z);Z\subset W\}=W$ and $\dickson(W)=1$ we find
  that the set of missing points of $\sickson_W$ is
  $\bigl\{\dickson(W)/\dickson(Z)^3;Z\subset
  W\bigr\}=\{1,\xi^3,\xi^6,\xi^9,\xi^{12}\}$,
  the missing point of multiplicity $3$ being $1$.

  In what follows, since we have to deal with different traces simultaneously,
  we will adopt the simpler notation $\trace_{2^s}(x)=\trace_{\F_{2^n}/\F_{2^s}}(x)
  =x+x^{2^s}+x^{4^s}+\dotsb$ for $s\mid n$.
  
  The collision space $(W^2)^\perp=W^\perp$ is easily seen to be
  $\bigl\{y\in\F_{2^n};\trace_{16}(y)\in\F_2\bigr\}$
  and intersects $W$ in $\F_2$.\footnote{Here we use that
  $[\F_{2^n}/\F_{16}]=n/4$ is odd and hence
  $\trace(y)=\trace_{16}(y)$ for
  $y\in\F_{16}\subseteq\F_{2^n}$.} This shows that $1$ is the only
  missing point in $(W^2)^\perp$.

  Now Theorem~\ref{thm:HZ}\eqref{thm:HZ:enum} implies that $\mat{C}_W$
  has row sums $4$ and $10$ with corresponding frequencies
  $f_4=2^{n-4}$ and $f_{10}=2^{n-4}-1$. The $2^{n-4}\times(2^{n-3}-1)$
  submatrix $\mat{C}_W(I)$ formed by the rows of weight $4$ has column
  sums $\leq 4$, since the supporting lines and planes in
  $\PG(\F_{2^n}/W)$ (cf.\ Theorem~\ref{thm:C}\eqref{thm:C:subspaces})
  meet the affine subspace $\{T_i;i\in I\}$ in at most $2$ points
  (resulting in a column sum $\leq 2+2=4$), respectively, at most $4$
  points (column sum $\leq 1+1+1+1=4$).\footnote{For the column of
    Type~$4^1$ the bound is trivial.} Hence the
  number of nonzero columns of $\mat{C}_W(I)$ must be at least
  $2^{n-4}$, and we can take $m'=2^{n-4}$ in Corollary~\ref{cor:copt}
  to conclude that
  \begin{equation*}
    N_1\geq 2^{n-4}(6-4)+2^{n-4}=3\cdot 2^{n-4}.
  \end{equation*}
  This completes the proof.
\end{proof}
Part~\eqref{thm:main:n=4mod8} of our main theorem now follows from
Theorem~\ref{thm:n=4mod8} and
\begin{equation*}
  3\cdot
  2^{v-7}(2^{v-3}-1)>\frac{9}{8}\cdot\frac{(2^{v-4}-1)(2^{v-3}-1)}{3}
  =\frac{9}{8}\gauss{v-3}{2}{2}.
\end{equation*}
\begin{remark}
  \label{rmk:n=4mod8}
  In the smallest case $v=7$, in which $\F_{2^n}$ coincides with the
  subfield $\F_{16}$, the maximum local net gain is equal to $3$; cf.\
  Example~\ref{ex:v=7cont}.
  Theorem~\ref{thm:n=4mod8} gives a lower bound for the maximum
  net gain at lengths $v=7+8t$, $t=1,2,\dots$, which scales nicely with
  $v$ and thus can be viewed as ``anchored'' at $v=7$. Indeed, the
  proof of the theorem involves only computations in the subfield
  $\F_{16}$, no matter how large $\F_{2^n}$ is. This point of of view
  will become essential in the case $v\equiv3\pmod{8}$; see the next
  theorem. However, it should be noted that these observations only
  give lower bounds for the maximum net gain and that the actual
  maximum net gain can be substantially larger. For example, in the
  case $v=15$ the maximum net gain is $924>3\cdot 2^8=768$; cf.\
  Table~\ref{tbl:netgain}.
\end{remark}
\begin{theorem}
  \label{thm:n=0mod8}
  For packet lengths $v\equiv 3\pmod{8}$, i.e., $n=v-3\equiv
  0\pmod{8}$, the maximum achievable local net gain $N_1$ in the
  general RRP satisfies $N_1\geq 54\cdot 2^{n-8}=54\cdot 2^{v-11}$, and hence the
  corresponding optimum subspace codes have size
  \begin{equation*}
    \#\mathcal{C}\geq 2^{2(v-3)}+54\cdot 2^{v-11}(2^{v-3}-1).
  \end{equation*}
\end{theorem}
\begin{proof}
  Again taking $W$ as the trace-zero plane in
  $\F_{16}\subset\F_{2^n}$, the proof remains the same as for
  Theorem~\ref{thm:n=4mod8} up to the point where the collision space
  is computed. The explicit formula for $(W^2)^\perp=W^\perp$ obtained
  earlier remains valid, but now the elements in $\F_{16}$ have trace
  zero and hence are in $(W^2)^\perp$. In particular $(W^2)^\perp$ now
  contains all $5$ missing points, and their geometric configuration
  must be taken into account. From $\xi^{12}=\xi^9+\xi^6+\xi^3+1$ it
  is clear that the $5$ points form a projective basis of their hull
  $M=\F_{16}$ (i.e., are $5$ points in general position). Giving the
  triple point homogeneous coordinates $(1:1:1:1)$, the corresponding
  linear $[7,4]$ code has generator matrix
  \begin{equation*}
    \begin{pmatrix}
      1&0&0&0&1&1&1\\
      0&1&0&0&1&1&1\\
      0&0&1&0&1&1&1\\
      0&0&0&1&1&1&1
    \end{pmatrix}
  \end{equation*}
  and weight distribution $A_0=1$, $A_2=6$, $A_4=5$, $A_6=4$. The
  corresponding row-sum spectrum is $m_{14}=2^{n-7}-1$, $m_{10}=6\cdot
  2^{n-7}$, $m_6=5\cdot 2^{n-7}$, $m_2=4\cdot 2^{n-7}$. As before, let
  $\mat{C}_W(I)$ be the submatrix of $\mat{C}_W$ formed by the rows
  with $r_i\leq 6$, i.e.\ $r_i=2$ and $r_i=6$. Our goal is to
  establish a lower bound $m'$ on the number of nonzero columns of
  $\mat{C}_W(I)$, which is more difficult in this case.

  First we note that the solids $\{T_i;i\in I\}$ are determined by
  $\trace(x)=1$ (corresponding to the codewords with $1$ or $3$ nonzero
  entries among the first $4$ coordinates) or $\trace(\xi^{3t}x)=1$ for
  $1\leq t\leq 4$ (corresponding to the codeword $(1111000)$). In
  $\F_{2^n}/M^\perp\cong\PG(3,\F_2)$ these solids determine $9$
  points, the first $8$ of which form an affine subspace (complement
  of the plane $\trace(x)=0$).
  
  Since $M=\F_{16}$, we have $(M^{1/2})^\perp=M^\perp=\F_{16}^\perp
  =\bigl\{x\in\F_{2^n};\trace_{16}(x)=0\bigr\}$ and can
  express the conditions in terms of $\trace_{16}(x)$. Using
  $\trace(x)=\trace_{\F_{16}/\F_2}\bigl(\trace_{16}(x)\bigr)$, we find
  that the last point has equation $\trace_{16}(x)=1$ and the $9$
  points are those with $\trace_{16}(x)\in(\F_{16}\setminus
  W)\cup\{1\}$.
  
  Since $\spol_W(X)=X^8+X^4+X^2+X$, $W'=W$, and $\dickson(W)=1$, we
  have from the proof of Theorem~\ref{thm:C}\eqref{thm:C:types} that
  the entry of $\mat{C}_W$ corresponding to $T=\langle W,x\rangle$ and
  $y\in W^\perp$ is the number of solutions of the equation
  \begin{equation}
    \label{eq:f(x)}
    x^8+x^4+x^2+x=y^2w^4+yw\quad\text{in $W$}.
  \end{equation}
  The above conditions on $\trace_{16}(x)$ translate into conditions
  on $\trace_{256}(x^8+x^4+x^2+x)$; the first into 
  $\trace_{16}(x^8+x^4+x^2+x)=\trace(x)=1$, which is equivalent to
  $\trace_{256}(x^8+x^4+x^2+x)
  \in\bigl\{t\in\F_{256};\trace_{\F_{256}/\F_{16}}(t)=1\bigr\}=t_0+\F_{16}$;
  and the second into
  $\trace_{256}(x^8+x^4+x^2+x)^2+\trace_{256}(x^8+x^4+x^2+x)
  =\trace_{256}(x^{16}+x)=\trace_{16}(x)=1$, i.e.\ 
  $\trace_{256}(x^8+x^4+x^2+x)\in\F_4\setminus\F_2$.
  
  On the other hand,
  \begin{equation*}
    \trace_{256}(y^2w^4+yw)=\trace_{256}(y)^2w^4
    +\trace_{256}(y)w    
  \end{equation*}
  depends only on $\trace_{256}(y)$ and hence is constant on cosets of
  $H_{256}=\bigl\{x\in\F_{2^n};\trace_{256}(x)=0\bigr\}$. 

  Putting the preceding observations together, we conclude
  that the total number of solutions of \eqref{eq:f(x)} with $T_i=\langle
  W,x\rangle$ varying over $i\in I$ is constant on cosets of
  $H_{256}$ as well. This means that the frequencies in the column-sum
  spectrum of $\mat{C}_W(I)$ are obtained from those for $n=8$
  by scaling with $2^{n-8}$. 
  In particular, the number of nonzero columns of $\mat{C}_W(I)$ is
  $2^{n-8}\cdot t$, where $t$ is the corresponding number for the case
  $n=8$. For $n=8$ we find by inspecting $\mat{C}_W$ in
  Figure~\ref{fig:CWn=8} that the $18\times 31$ submatrix
  $\mat{C}_W(I)$ has $22$ nonzero columns ($16$ columns of Type $1^4$
  and $6$ columns of Type $1^2$), resulting in $m'=22$ and
  \begin{equation*}
    N_1\geq 8(6-2)+22=54.\footnotemark
  \end{equation*}
  \footnotetext{In fact $N_1=54$, as shown in
    Section~\ref{sec:comp}. With some effort the column types of
    $\mat{C}_W(I)$ can also be computed by hand. It turns out that the
    $16$ columns corresponding to $y\in W^\perp$ with
    $\trace_{16}(y)=1$ are those of type $1^4$, and the $6$ columns
    corresponding to $y\in W\setminus\F_2\subset\F_{16}$ (which have
    $\trace_{16}(y)=0$) are those of type $1^2$.}%
  In the general case the bound then scales to
  $N_1\geq 2^{n-8}\cdot 54$, as asserted.
\end{proof}
Again comparing the bound of Theorem~\ref{thm:n=0mod8} with the LMRD code
bound, we obtain
\begin{equation*}
  54\cdot
  2^{v-11}(2^{v-3}-1)>\frac{81}{64}\cdot\frac{(2^{v-4}-1)(2^{v-3}-1)}{3}
  =\frac{81}{64}\gauss{v-3}{2}{2}.
\end{equation*}
This proves Part~\eqref{thm:main:n=0mod8} of our main theorem.

The computational results presented in Table~\ref{tbl:netgain} show
that the largest subspace codes obtained by solving the RRP exceed the
LMRD code bound for all $v\in\{7,8,\dots,15\}$ except for $v=8$ and
$v=10$. Although the margin is rather narrow for $v\in\{12,14\}$, we
make the following
\begin{conjecture}
  \label{conj:LMRD}
  For any packet length $v\geq 7$, $v\notin\{8,10\}$, the largest
  subspace codes that can be obtained by solving the RRP exceed the
  LMRD code bound and thus are better than the codes resulting from
  the echelon-Ferrers construction and its variants.
\end{conjecture}
By Theorems~\ref{thm:n=4mod8} and~\ref{thm:n=0mod8},
Conjecture~\ref{conj:LMRD} is true for packet lengths
$v\equiv 3\pmod{4}$. For lengths $v\equiv 1\pmod{4}$, which correspond
to $n\equiv 2\pmod{4}$, the following considerations provide strong
evidence in support of Conjecture~\ref{conj:LMRD}.

Inspecting the proof of Theorem~\ref{thm:n=4mod8}, we see that the
argument remains valid for $n\equiv 2\pmod{4}$, provided there exists
a plane $W$ in $\PG(\F_{2^n})$ which satisfies the conditions of
Theorem~\ref{thm:cij=4} and has $\mu=3$, i.e., there is a triple
missing point inside $(W^2)^\perp$ and $4$ missing points outside
$(W^2)^\perp$. Using the explicit description of the missing points
for the planes $W_1=\langle 1,a,b\rangle$ considered in
Theorems~\ref{thm:HZ} and~\ref{thm:cij=4}, it is easy to test the condition
$P\in(W^2)^\perp$ and compute the number of planes $W_1$ with
$\mu=3$ for small values of $n$.\footnote{Choosing $c=1$ in
  Theorem~\ref{thm:HZ}\eqref{thm:HZ:mp}, the condition
  $P\in(W^2)^\perp$ reduces to $\trace(a_1/a_0^2)=0$ or, somewhat
  easier to handle, $\trace\bigl((b+\omega a+x)^{-3}\bigr)=0$; cf.\
  Theorem~\ref{thm:cij=4}\eqref{thm:cij=4:missing}.}

We have written a small SageMath worksheet for this job. The
results are shown in Table~\ref{tbl:missing}.
\begin{table}
  \centering
  \(\begin{array}{r|rrrrr}
    n\backslash\mu&3&4&5&6&7\\\hline
    4&1&0&0&0&0\\
    6&0&0&3&2&0\\
    8&0&8&12&0&1\\
    10&0&20&45&10&10\\
    12&32&100&90&96&23\\
    14&56&392&483&322&112\\
    16&320&1360&2136&1376&269\\
    18&1392&5388&8121&5546&1398\\
    20&5616&21900&32550&21840&5475\\
    22&22088&86240&131967&87362&21868  
  \end{array}\)
  \caption{Distribution of the number $\mu$
    of missing points inside $(W_1^2)^\perp$}
  \label{tbl:missing}
\end{table}
From the table we see that $\F_{2^n}$ contains a plane $W$ with a
triple missing point and $\mu=3$ for all
$n\in\{12,14,\dots,22\}$. Moreover, the (shifted) frequency distribution of
planes $W_1$ with $\mu$ missing points, normalized by the total number
$(2^{n-2}-1)/3$ of planes $W_1$ (cf.\
Theorem~\ref{thm:cij=4}\eqref{thm:cij=4:part}), seems to converge to
the binomial distribution $(1/16,4/16,6/16,4/16,1/16)$.

In particular, Conjecture~\ref{conj:LMRD} is also true for
$n\in\{10,14,18,22\}$, i.e., for packet lengths
$v\in\{13,17,21,25\}$.\footnote{The case $n=10$, where no plane $W_1$
  with $\mu=3$ exists, is covered by Table~\ref{tbl:netgain}.}

It may be possible to prove Conjecture~\ref{conj:LMRD} for
$v\equiv1\pmod{4}$ with the aid of
character sums and the observations in Remark~\ref{rmk:abomega}. The
case of even $v$, however, seems to be much harder.

\section{The Significance of the Associated 
Linear Code}\label{sec:code}

From the discussion following Theorem~\ref{thm:HZ} and the previous
section we know already that the linear $[\mu,k]$ code $C=C_W$
associated to the multiset of missing points contained in the
collision space $(W^2)^\perp$ plays an important role in computing the
maximum net gain $N_1$ of the RRP relative to $W$. In this section we
add further evidence to this by using $C$ to express the bounds of
Corollary~\ref{cor:copt} in terms of $\mu$, $k$ and showing that this
refinement suffices to complete the solution of the RRP for
$v\in\{14,15\}$; this question was left open in Section~\ref{sec:comp}.

The following lemma is implicit in the remarks following
Theorem~\ref{thm:HZ}.
\begin{lemma}
  \label{lma:copt}
  Given a plane $W$ in $\PG(\F_{2^v})$, let $C$ be a
  binary linear $[\mu,k]$ code associated with the multiset $\mset{m}$
  of missing points of $\sickson_W$ contained in $(W^2)^\perp$ (i.e.,
  $\mu$ is the cardinality of $\mset{m}$ and $k$ the dimension of its
  hull $M$). Then the quantity $\sum_{i=0}^5m_r(6-r)$ in
  Corollary~\ref{cor:copt} can be expressed in terms of the weight
  distribution $A_0,\dots,A_\mu$ of $C$ as follows:
  \begin{equation*}
    \sum_{r=0}^5m_r(6-r)=2^{n-3-k}\times\sum_{i>(\mu+1)/2}(2i-1-\mu)A_i.\footnotemark
  \end{equation*}
  \footnotetext{In the case $\mu=0$, where all rows of $\mat{C}_W$
    have sum $7$, the right-hand side should be
    interpreted as zero.}%
\end{lemma}
\begin{proof}
  Just observe that $r,i$ are related by $6-r=6-(\mu+7-2i)=2i-1-\mu$
  and that the all-zero codeword of $C$, which corresponds to zero or
  more rows with sum $7+\mu$, does not contribute to either side.
\end{proof}
Planes $W$ of the type considered in Theorem~\ref{thm:cij=4} are
distinguished by the fact that the corresponding multiset $\mset{m}$
is not a set; equivalently, the associated $[\mu,k]$ code $C$ is not
projective, or $\dham(C^\perp)=2$.\footnote{Note that
  $\dham(C^\perp)=1$ is not possible, since by definition $C$ has no
  all-zero coordinates.}
\begin{lemma}
  \label{lma:codesum}
For projective $[\mu,k]$ codes $C$, in the parameter range of interest to
us, we have the following upper bounds on the ``code sums''
$\sum_{i>(\mu+1)/2}(2i-1-\mu)A_i$.
  \begin{equation*}
    \begin{array}{c|ccccccc}
      \mu\backslash k&1&2&3&4&5&6&7\\\hline
      1&0\\
      2&0&1\\
      3&0&0&2\\
      4&0&0&3&7\\
      5&0&0&2&10&14\\
      6&0&0&3&9&20&38\\
      7&0&0&0&8&20&40&76
    \end{array}
  \end{equation*}
  Moreover, these bounds are best possible.
\end{lemma}
For substituting the bounds into Lemma~\ref{lma:copt}, it is
convenient to normalize them by $2^{-k}$. The resulting normalized
upper bounds $\gamma_{\mu,k}$ are listed in the following table, in
order of increasing strength.
\begin{table}
  \centering
  \(\begin{array}{l|l}
    \multicolumn{1}{c|}{\gamma_{\mu,k}}&\multicolumn{1}{c}{(\mu,k)}\\\hline
    0.625&(5,4),\,(6,5),\,(7,5),\,(7,6)\\
    0.59375&(6,6),\,(7,7)\\
    0.5625&(6,4)\\
    0.5&(7,4)\\
    0.4375&(4,4),\,(5,5)\\
    0.375&(4,3),\,(6,3)\\
    0.25&(2,2),\,(3,3),\,(5,3)\\
    0&\text{otherwise}
  \end{array}\)
  \caption{Upper bounds on the normalized code sums of projective $[\mu,k]$
    codes}\label{tbl:ncodesum}
\end{table}

\begin{proof}[Proof of Lemma~\ref{lma:codesum}]
  The entries in the diagonal of the table are the code sums obtained for the
  trivial $[\mu,\mu]$ codes. The zero entries are due to the fact that
  a projective $[\mu,k]$ code has $\mu\leq 2^k-1$ and the simplex codes
  ($\mu=2^k-1$) have only codewords of weight $0$ and $(\mu+1)/2$.

  The remaining cases are settled in an ad hoc fashion,
  using codes with a systematic generator matrix $\mat{G}$. The code
  sums yet relevant are
  \begin{equation*}
    \begin{array}{c|c}
      \mu&\text{code sum}\\\hline
      4&A_3+3A_4\\
      5&2A_4+4A_5\\
      6&A_4+3A_5+5A_6\\
      7&2A_5+4A_6+6A_7
    \end{array}
  \end{equation*}
  \paragraph{$k=3$} Viewing the columns of $\mat{G}$ as points in the
  Fano plane $\PG(2,\F_2)$, we have to consider $2$ cases for $\mu=4$
  (a quadrangle and a line with one additional point, both of which
  have code sum $3$) and one case for $\mu=5,6$ (having maximum weight
  $4$ with $A_4=1$ and $A_4=3$ respectively).

  \paragraph{$k=4$} For $\mu=5$ the even-weight subcode of $\F_2^5$
  (with 5th column $(1111)^\tp$ in $\mat{G}$) is the unique code having
  code sum $10$. For $\mu=6$ there are $4$ equivalence classes of
  codes with non-systematic parts 
  $\left(
    \begin{smallmatrix}
      1&1&0&0\\
      0&0&1&1\\
    \end{smallmatrix}
  \right)^\tp$, 
  $\left(
  \begin{smallmatrix}
      1&1&0&1\\
      0&0&1&1\\
    \end{smallmatrix}
  \right)^\tp$,
  $\left(
    \begin{smallmatrix}
      1&0&1&1\\
      0&1&1&1\\
    \end{smallmatrix}
  \right)^\tp$, 
  $\left(
  \begin{smallmatrix}
      1&0&1&0\\
      0&1&1&0\\
    \end{smallmatrix}
  \right)^\tp$ and code sums $9$, $9$, $8$, $8$,
  respectively.\footnote{The equivalence classes are best viewed as
    equivalence classes of the corresponding dual $[6,2]$ codes. Note
    that $C$ is projective iff $C^\perp$ (which is not necessarily
    projective) has minimum weight $\geq 3$. In the case under
    consideration this restricts the column multiplicities of
    $C^\perp$ to values $\leq 3$.}

  For $\mu=7$ there are $5$ equivalence classes of
  codes with non-systematic parts 
  $\left(
    \begin{smallmatrix}
      1&0&0&1\\
      0&1&0&1\\
      0&0&1&1
    \end{smallmatrix}
    \right)$,
    $\left(
    \begin{smallmatrix}
      1&0&0&1\\
      0&1&1&0\\
      0&0&1&1
    \end{smallmatrix}
    \right)$,
    $\left(
    \begin{smallmatrix}
      1&0&1&1\\
      0&1&0&1\\
      0&1&1&0
    \end{smallmatrix}
    \right)$,
    $\left(
    \begin{smallmatrix}
      0&1&1&1\\
      1&0&1&1\\
      1&1&0&1
    \end{smallmatrix}
    \right)$,
    $\left(
    \begin{smallmatrix}
      0&1&1&0\\
      1&0&1&0\\
      1&1&0&0
    \end{smallmatrix}
    \right)$
  and code sums $4$, $8$, $8$, $6$, $6$,
    respectively.\footnote{Now the column multiplicities of $C^\perp$
      are $\leq 2$.}

  \paragraph{$k=5$} For $\mu=6$ the even-weight subcode of $\F_2^6$
  has code sum $20$ is the only such code. For $\mu=7$ there are $8$
  equivalence classes with non-systematic parts
  $\left(
    \begin{smallmatrix}
      1&1&0&0&0\\
      0&0&1&1&0\\
    \end{smallmatrix}
  \right)^\tp$, 
  $\left(
  \begin{smallmatrix}
      1&1&0&1&0\\
      0&0&1&1&0\\
    \end{smallmatrix}
  \right)^\tp$,
  $\left(
    \begin{smallmatrix}
      1&0&1&1&0\\
      0&1&1&1&0\\
    \end{smallmatrix}
  \right)^\tp$, 
  $\left(
  \begin{smallmatrix}
      1&0&1&0&0\\
      0&1&1&0&0\\
    \end{smallmatrix}
  \right)^\tp$,  
    $\left(
    \begin{smallmatrix}
      1&1&1&0&0\\
      0&0&0&1&1\\
    \end{smallmatrix}
  \right)^\tp$, 
  $\left(
  \begin{smallmatrix}
      1&1&1&0&1\\
      0&0&0&1&1\\
    \end{smallmatrix}
  \right)^\tp$,
  $\left(
    \begin{smallmatrix}
      1&1&0&0&1\\
      0&0&1&1&1\\
    \end{smallmatrix}
  \right)^\tp$, 
  $\left(
    \begin{smallmatrix}
      1&1&0&1&1\\
      0&0&1&1&1\\
    \end{smallmatrix}
  \right)^\tp$.
  The code sums are all $\leq 20$, with equality for the 6th and 8th
  equivalence class.

  \paragraph{$k=6$} Here we have only one case to consider, $\mu=7$.
  The even-weight subcode of $\F_2^7$ has code sum $28$ and is not
  ``optimal'' in this case. The codes with 7th column $(111110)^\tp$,
  $(111100)^\tp$ in their $\mat{G}$ have code sum $40$ and are the
  only such codes.
\end{proof}

\begin{theorem}
  \label{thm:n=11,12}
  For $v\in\{14,15\}$ the computed maximum local
  net gain in Table~\ref{tbl:netgain} ($379$ for $v=14$, $924$ for $v=15$) 
  represents the true maximum achievable net gain of the RRP. 
\end{theorem}
\begin{proof}
  (i) $v=14$ ($n=11$). It suffices to show that any plane $W$ in
  $\PG(\F_{2048})$ that has at most $4$ missing points in
  $(W^2)^\perp$ satisfies $N_1<379$.  For $\mu\leq 4$ the maximum
  value of $\gamma_{\mu,k}$ in Table~\eqref{tbl:ncodesum} is
    $\gamma_{4,4}=0.4375$. From Corollary~\ref{cor:copt},
    Lemma~\ref{lma:copt} and the table,
    the maximum net gain relative to $W$ satisfies
  \begin{align*}
    N_1&\leq\sum_{r=0}^5m_r(6-r)+255=256\gamma_{\mu,k}+255\\
    &< 1.4375\times 256=368,
  \end{align*}
  as desired.
  
  (ii) $v=15$ ($n=12$). Here we must show that any plane $W$ in
  $\PG(\F_{4096})$ that is not of the type considered in
  Theorem~\ref{thm:cij=4} has $N_1<924$. Since these planes are
  exactly those for which the associated $[\mu,k]$ code is projective,
  we can use the bound $\gamma_{\mu,k}\leq 0.625$ from
  Table~\eqref{tbl:ncodesum} in Lemma~\ref{lma:copt}. This gives
  \begin{equation*}
    \label{eq:muk}
    N_1<1.625\times 512=832,
  \end{equation*}
  completing the proof of the theorem.
\end{proof}

\section{Conclusion}\label{sec:conc}

We conclude this paper with a list of open problems related to our
work. Only the first problem has been discussed already (in
Section~\ref{sec:bingo}).

\begin{problem}
  \emph{Prove Conjecture~\ref{conj:LMRD}, either partially for odd
    packet lengths $v\equiv1\pmod{4}$ or in full.} The case of odd $v$
  (even $n$) seems more accessible in view of the availability of
  planes with a triple missing point and the overwhelming evidence for
  the existence of such planes with $\mu=3$, which would settle this
  part.  The case of even $v$ includes all cases where $\F_{2^n}$ has
  prime degree over $\F_2$ and hence no nontrivial subfields. In this
  case an approach different from that in
  Theorems~\ref{thm:n=4mod8},~\ref{thm:n=0mod8} must be used, perhaps
  starting with an existence proof of planes $W$ with a large code sum
  in their associated $[\mu,k]$ code (cf.\ Lemma~\ref{lma:copt}) and
  using the lower bound in Corollary~\ref{cor:copt} with a suitable
  constant $m'$. Note that in terms of the size of $\mat{C}_W$, the
  threshold for the local net gain set by the LMRD code bound is
  $\frac{2^{n-1}-1}{3}\approx\frac{4}{3}\times(2^{n-3}-1)$.
\end{problem}

\begin{problem}
  \emph{Improve the expurgation-augmentation method for small packet
    lengths $v$}.  Although our method represents an asymptotic
  improvement of the known constructions of $(v,M,4;3)_2$ codes for
  $v\mapsto\infty$, it is much inferior to the group-invariant
  computational constructions in \cite{braun-ostergard-wassermann15}
  for lengths $v\in\{8,9,10,11\}$. To some extent this can be remedied
  through adding a further computational extension step by planes
  meeting $S$ in a line (cf.\ the remarks in Section~\ref{sec:comp}),
  but the results remain inferior to
  \cite{braun-ostergard-wassermann15}, and with increasing length the
  method soon becomes infeasible.

  To overcome this problem, an algebraic description of the free
  planes relative to an optimal solution of the RRP (or a suitable
  subcode thereof, which avoids ``colliding planes'') would be
  desirable. Another approach, which for $v=8$ at least yields some
  improvement,\footnote{The largest $(8,M,4;3)_2$ code obtained has
    size $M=1286$, compared with $1259$ in Section~\ref{sec:comp} and
    $1326$ in \cite{braun-ostergard-wassermann15}.}  is to compute,
  relative to the expurgated lifted Gabidulin code, the set of all
  free planes meeting $S$ in a point and use a suitable maximum-clique
  algorithm to find the absolutely largest extension of the expurgated
  code by such planes. We have determined experimentally that usually
  there are indeed additional free planes (corresponding to
  non-standard rearrangements of the free lines into new
  planes). However, including those planes in the optimization
  problems destroys its rotation-invariance and makes it much more
  computationally expensive. Again an algebraic description of the set of
  all free planes may help to overcome this problem.
\end{problem}

\begin{problem}
  \emph{Use the expurgation-augmentation method with other LMRD codes
    or subsets thereof.}
  The Gabidulin codes $\gabidulin_W$ considered in this paper are not
  the only MRD codes with these parameters, provided that $v\geq 7$. 
  Therefore the question arises whether one can adapt the
  expurgation-augmentation method for use with other LMRD codes and,
  if so, what the maximum sizes of the corresponding modified subspace
  codes will be. Although it is not directly related to this question,
  the following observation made during the preparation of
  \cite{smt:fq11proc} may be of interest in this regard: One of the
  five isomorphism types of optimal $(6,77,4;3)_2$ codes, named Type~B
  in \cite{smt:fq11proc}, contains a
  set of $16$ planes disjoint from $S=\{\vek{0}\}\times\F_2^3$ at
  mutual subspace distance $4$. This set corresponds to a
  $4$-dimensional constant-rank-two subspace of $\F_2^{3\times 3}$ of
  the type discovered by Beasley \cite{beasley99}. Since Gabidulin
  codes in $\F_2^{3\times 3}$ do not contain such a ``Beasley code'',
  we have that $(6,77,4;3)_2$ codes of Type~B cannot be obtained by
  ordinary expurgation-augmentation as considered in this paper.
\end{problem}

\begin{problem}
  \emph{Generalize the expurgation-augmentation method to subspace
    codes of constant dimension $k>3$.}  As an example we consider the
  smallest length $v=8$, for which this problem is meaningful. For
  $v=8$ there are two cases with $k>3$, where $\smax_2(v,d;k)$ is
  unknown, viz.\ $(v,d;k)=(8,4;4)$ and $(8,6;4)$. In the first case
  the corresponding Gabidulin code $\gabidulin$ provides a set of
  $2^{12}$ solids, which are disjoint from $S=\{\vek{0}\}\times\F_2^4$
  and cover each plane in $\PG(\F_2^8)$ disjoint from $S$ exactly
  once. Since solids disjoint from $S$ contain $15$ such planes, while
  solids meeting $S$ in a point contain only $8$ such planes, it
  should be possible---at least in principle---to rearrange the planes
  in some subset of $\gabidulin$ into new solids meeting $S$ in a
  point and thereby increase the code size significantly. However, the
  details seem a lot more involved than in the case $k=3$, and
  suitable subsets of $\gabidulin$ have yet to be found. The same
  remark applies to the other parameter triple $(8,6;4)$, in which the
  corresponding Gabidulin code $\gabidulin$ consists only of $2^8$ solids
  disjoint from $S$ and covers each line disjoint from $S$ exactly
  once.\footnote{The currently best lower bound in this case is still
    the rather trivial $\smax_2(8,6;4)\geq 257$, coming from
    $\gabidulin\cup\{S\}$, and it may well give the true result.}
\end{problem}




\def\cprime{$'$}
\providecommand{\href}[2]{#2}
\providecommand{\arxiv}[1]{\href{http://arxiv.org/abs/#1}{arXiv:#1}}
\providecommand{\url}[1]{\texttt{#1}}
\providecommand{\urlprefix}{URL }

\medskip
\medskip

\end{document}